\documentclass[12pt,bibtotoc,a4paper]{article}
\usepackage{geometry}
\usepackage[T1]{fontenc}        
\usepackage[utf8x]{inputenc}    
\usepackage{ucs}                
\usepackage{amsmath}            
\usepackage{amsfonts}
\usepackage{amssymb}
\usepackage{amsthm} 
\usepackage{thmtools}

\usepackage{ifthen}
\usepackage{color}
\usepackage{graphicx} 
\usepackage{textcomp}            
\usepackage{palatino}            

\usepackage{eulervm}            
\usepackage{varioref}            
\usepackage[colorlinks=false,plainpages=false]{hyperref}

\usepackage{cleveref}             
\usepackage{fancyhdr}
\usepackage{float}
\usepackage{multirow}
\usepackage{rotating}
\usepackage{subfig}
\usepackage{xr} 
\usepackage{xcite}

\newcommand{\R}{{\mathbb{R}}}         
\newcommand{\E}{{\mathbb{E}}}
 \newcommand{\N}{{\mathbb{N}}}         

\newcommand{\renu}{\mathbb{R}}

\newcommand{\bay}{\begin{array}}
\newcommand{\eay}{\end{array}}

\newcommand{\bqa}{\begin{eqnarray*}}
\newcommand{\eqa}{\end{eqnarray*}}

\newcommand{\bee}{\begin{eqnarray*}}
\newcommand{\eee}{\end{eqnarray*}}

\newcommand{\bea}{\begin{eqnarray*}}
\newcommand{\eea}{\end{eqnarray*}}

\newcommand{\bqan}{\begin{eqnarray}}
\newcommand{\eqan}{\end{eqnarray}}

\newcommand{\be}{\begin{eqnarray}}
\newcommand{\ee}{\end{eqnarray}}

\newcommand{\bit}{\begin{itemize}}
\newcommand{\eit}{\end{itemize}}

\newcommand{\ben}{\begin{enumerate}}
\newcommand{\een}{\end{enumerate}}

\newcommand{\beq}{\begin{equation}}
\newcommand{\eeq}{\end{equation}}

\newcommand{\bdes}{\begin{description}}
\newcommand{\edes}{\end{description}}

\newcommand{\btb}{\begin{tabular}}
\newcommand{\etb}{\end{tabular}}

\newcommand{\bcen}{\begin{center}}
\newcommand{\ecen}{\end{center}}

\newcommand{\bmp}{\begin{minipage}}
\newcommand{\emp}{\end{minipage}}

\newcommand{\cov}{\operatorname{{\it Cov}}}
\newcommand{\Cov}{\operatorname{{\it Cov}}}

\newcommand{\Var}{\operatorname{{\it Var}}}

\newcommand{\tr}{\operatorname{tr}}

\newcommand{\diag}{\operatorname{\it diag}}

\newcommand{\va}{\boldsymbol{a}}
\newcommand{\vb}{\boldsymbol{b}}

\newcommand{\ve}{\boldsymbol{e}}

\newcommand{\vx}{\boldsymbol{x}}
\newcommand{\vy}{\boldsymbol{y}}

\newcommand{\vA}{\boldsymbol{A}}

\newcommand{\vD}{\boldsymbol{D}}
\newcommand{\vE}{\boldsymbol{E}}

\newcommand{\vH}{\boldsymbol{H}}
\newcommand{\vI}{\boldsymbol{I}}
\newcommand{\vJ}{\boldsymbol{J}}

\newcommand{\vM}{\boldsymbol{M}}

\newcommand{\vP}{\boldsymbol{P}}

\newcommand{\vT}{\boldsymbol{T}}

\newcommand{\vV}{\boldsymbol{V}}
\newcommand{\vW}{\boldsymbol{W}}
\newcommand{\vX}{\boldsymbol{X}}
\newcommand{\vY}{\boldsymbol{Y}}
\newcommand{\vZ}{\boldsymbol{Z}}

\newcommand{\vmu}{\boldsymbol{\mu}}

\newcommand{\vpi}{\boldsymbol{\pi}}

\newcommand{\vsigma}{\boldsymbol{\sigma}}
\newcommand{\vSigma}{\boldsymbol{\Sigma}}

\newcommand{\vnull}{{\bf 0}}

\newcommand{\PP}{{\mathbb{P}}}         

\newcommand{\F}{\mathcal {F}}
\newcommand{\Crp}[1]{\hspace{-0,0cm}{\Cref{#1}\hspace{0,1cm}(p.\pageref*{#1})}}

\newcommand{\Cdot}{\cdot}
\newcommand{\sem}{\large{\textbf{;}}}
\renewcommand*\arraystretch{1.2}

\newcommand{\0}{\mathcal{O}}

\newtheoremstyle{Test1}
  {2 \baselineskip}
  {1.5 \baselineskip}
  {\itshape}
  {-0.0ex}
  {\fontfamily{ppl}\fontseries{l}\fontshape{n}}
  {:}
  {\newline}
   {}

\theoremstyle{Test1}
\newtheorem{Sa}{Theorem}[section]
\newtheorem{theorem}{Theorem}[section]
\newtheorem{re}[Sa]{Remark}
\newtheorem{Def}[Sa]{Definition} 
\newtheorem{Le}[Sa]{Lemma}

\newtheorem{SaA}{Theorem}[section]

\newtheorem{LeA}[SaA]{Lemma}

\newtheorem{KoA}[SaA]{Korollar}

\makeatletter
\renewenvironment{proof}[1][\proofname]{\par
  \pushQED{\qed}%
  \fontfamily{ppl}\fontseries{m}\fontshape{it} \topsep6\p@\@plus6\p@\relax
  \trivlist
  \item[\hskip\labelsep
        \bfseries
    #1\@addpunct{:}]\ignorespaces
}{%
  \popQED\endtrivlist\@endpefalse
}
\makeatother

\begin{document}
%
%
\title{Inference For High-Dimensional Split-Plot-Designs:\\ A Unified Approach for Small to Large Numbers of Factor Levels}

{\begin{center}\Huge{Inference For High-Dimensional Split-Plot-Designs:\\ A Unified Approach for Small to Large Numbers of Factor Levels}\\\vspace{1\baselineskip}
\begin{large} Paavo Sattler$^1$ and Markus Pauly$^1$ \end{large}\\
{\tiny ${}^1$University of Ulm, Institute of Statistics}
 \end{center}
}
\vspace*{1\baselineskip}
{\bf Abstract:} 
\normalsize
Statisticians increasingly face the problem to reconsider the adaptability of classical
inference techniques. In particular, divers types of high-dimensional data structures are observed in various
research areas; disclosing the boundaries of conventional multivariate data analysis. 
Such situations occur, e.g., frequently in life
sciences whenever it is easier or cheaper to repeatedly generate a large number $d$ of 
observations per subject than recruiting many, say $N$, subjects. 
In this paper we discuss inference procedures for such situations in general heteroscedastic split-plot designs with $a$ independent groups of repeated measurements. 
These will, e.g., be able to answer questions about the occurrence of certain time, group and interactions effects or about particular
profiles.\\
The test procedures are based on standardized quadratic forms involving suitably symmetrized U-statistics-type estimators which are robust against an increasing number of dimensions $d$ and/or groups $a$.
We then discuss its limit distributions in a general asymptotic framework and additionally propose 
improved small sample approximations. Finally its small sample performance is investigated in simulations and the applicability is illustrated by a real data analysis.\\\\\\


\newpage

\section{Introduction}
In our current century of data, statisticians increasingly face the problem to reconsider the adaptability of classical inferential techniques. 
In particular, divers types of high-dimensional data structures are observed in various research areas; disclosing the boundaries of 
conventional multivariate data analysis. 
Here, the {\it curse of high dimensionality} or the {\it large $d$ small $N$ problem} is especially encountered in life sciences 
whenever it is easier (or cheaper) to repeatedly generate a large number $d$ of observations per subject than recruiting many, say $N$, subjects. 
Similar observations can be made in industrial sciences with subjects replaced by units. Such designs, where experimental units 
are repeatedly observed under different conditions or at different time points, are called {\it repeated measures designs} or (if two or more groups are observed) 
{\it split-plot designs}. In these trials, one likes to answer questions about the occurrence of certain group or time effects or about particular profiles. 
Conventionally, for $d < N$, corresponding null hypotheses are inferred with Hotelling’s $T^2$ (one or two sample case) or Wilks’s $\Lambda$, see e.g. 
\cite{Davis}[Section 4.3] or \cite{johnson} [Section 6.8]. 
Besides normality, these procedures heavily rely on the assumption of equal covariance matrices and particularly break down in high-dimensional settings with $N<d$. 
While there exist several promising approaches to adequately deal with the problem of covariance heterogeneity in the classical case with $d<N$ 
(see e.g. \cite{box, GG1958, GG1959, Huynh1976, Lecoutre1991, VallejoAto, Ahmad_etal_2008, Kenward, Brunner2012, Pesarin:2012, bib22, Kon:2015, Happ, Sarah}) 
most procedures for high-dimensional repeated measures designs rely on certain sparsity conditions (see e.g. \cite{Bai, bib8a, Katayama_2013, Nishiyama, Secchi, Cai14, Harrar} and the references cited therein).
In particular, in an asymptotic $(d,N)\to \infty$ framework, typical assumptions restrict the way the sample size $N$ and/or various powers of traces of the underlying covariances increase with respect to $d$. 
These type of sparsity conditions guarantee central limit theorems that lead to approximations of underlying test statistics by a fixed limit distribution. 
However, as illustrated in \cite{Paper1} for one-sample repeated measures these conditions can in general not be regarded as regularity assumptions. 
In particular, they may even fail for classical covariance structures. To this end, the authors proposed a novel approximation technique that showed considerably accurate results and investigated its asymptotic behavior in a flexible and non-restrictive $(d,N)\to \infty$ framework. 
Here, no assumptions regarding the dependence between $d$ and $N$ or the covariance matrix were made. 
In the current paper, we follow this approach and extend the results of \cite{Paper1} to general heteroscedastic split-plot designs with $a$ independent groups of repeated measurements. 
To even allow for a large number of groups as in \cite{Bathke, Bathkeb} or \cite{Zhan}, we do not only consider the case with a fixed number $a\in\N$ of samples but additionally allow for situations with $a\to\infty$. The latter case is of particular interest if most groups are rather small (as in screening trials) such that a classical test would essentially possess no power for fixed $a$. 
Here increasing the number of groups implies increasing the total sample size from which a power increase might be expected as well. 
This leads to one of the following asymptotic frameworks 
\begin{eqnarray*}
a\in\N\; \text{ fixed}  &\text{and}& (d,N)\to \infty, \label{eq: as framework 1}\\
d\in\N\; \text{ fixed}  &\text{and}& (a,N)\to \infty, \label{eq: as framework 2}\\
&\text{or}& (a,d,N)\to \infty. \label{eq: as framework 3}
\end{eqnarray*}
which we handle simultaneously in the sequel. For all considerations, the adequate and dimension-stable estimation of traces of certain powers of combined covariances turned out to be a major problem. 
It is tackled by introducing novel symmetrized estimates of $U$-statistics-type which possess nice asymptotic properties under all asymptotic frameworks given above. 

The paper is organized as follows. 
The statistical model together with the considered hypotheses of interest are introduced in \Cref{ Statistical Model and the Hypotheses}. The test statistic and its asymptotic behavior is investigated in \Cref{The Test Statistic and its Asymptotics}, where also novel dimension-stable trace estimators are introduced. 
Additional approximations for small sample sizes are theoretically discussed in 
\Cref{Better Approximations} 
and their performance is studied in simulations in \Cref{Simulations}. 
Afterwards, the new methods will be applied to analyze a high-dimensional data set from a sleep-laboratory trial in \Cref{Analysis of the Data sets}. 
The paper closes with a discussion and an outlook. All proofs in this paper are shifted to the supplementary material.

\section{Statistical Model and Hypotheses}\label{ Statistical Model and the Hypotheses}
We consider a split-plot design given by $a$ independent groups of $d$-dimensional random vectors 
\bqan\label{eq: model}
{\vX}_{i,j}= ({X}_{i,j,1},\dots,{X}_{i,j,d})^\top \stackrel{ind}{\sim}\mathcal{N}_d\left(\vmu_i,\vSigma_i\right)\hspace{0,2cm}j=1,\dots,n_i,\hspace{0,2cm} i=1,\dots,a
\eqan
with mean vectors $E(\vX_{i,1})=\vmu_i = (\mu_{i,t})_{t=1}^d \in \R^d$ and positive definite covariance 
matrices $Cov (\vX_{i,1})=\vSigma_i$. 
Here $j=1,\dots,n_i$ denotes the individual subjects or units in group $i=1,\dots, a$, $a,n_i\in \N$, where no specific structure of 
the group-specific covariance matrices $\vSigma_i$ is assumed. In particular, they are even allowed to differ completely. Altogether we have a total number of $N=\sum_{i=1}^a n_i$ random vectors representing observations from independent subjects. 
Within this framework, a factorial structure on the factors group or time can be incorporated by splitting up indices. 
Also, a group-specific random subject effect can be incorporated as outlined in \cite{Paper1}[Equation (2.2)].\\

Writing $\vmu = (\vmu_1^\top,\ldots,\vmu_a^\top)^\top$, linear hypotheses of interest
in this general split-plot model are formulated as
\bqan\label{eq:null hypo}
H_0(\vH) :\vH\vmu=\vnull 
\eqan 
for a proper hypothesis matrix $\vH$. It is of the form $\vH=\vH_S\otimes \vH_W$, where $\vH_S$ and $\vH_W$ refer to 
{\bf s}ubplot (time) and/or {\bf w}hole-plot (group) effects. For theoretical considerations it is often more convenient to reformulate $H_0(\vH)$ by means of the corresponding projection matrix
{$\vT =\vH^\top [\vH \vH^\top]^- \vH$}, see e.g. \cite{Paper1}. 
Here $(\cdot)^-$ denotes some generalized inverse of the matrix and $H_0(\vH)$ can equivalently be written as 
$H_0(\vT):\vT\vmu=\vnull$. 
It is a simple exercise to prove that the matrix $\vT$ is of the form $\vT=\vT_S\otimes\vT_W$ for projection matrices $\vT_S$ and $\vT_W$, see  \Crp{ginverse} in the supplement.
Typical examples are given by {
\bit
\item[(a)] No group effect:\\ $H_0^a: \left(\vP_a\otimes \frac{1}{d}\vJ_d\right)\vmu=\vnull$,
\item[(b)] No time effect:\\ $H_0^b: \left(\frac 1 a\vJ_a\otimes \vP_d\right)\vmu=\vnull$,
\item[(c)] No interaction effect between time and group:\\
$H_0^{ab}: \left(\vP_a\otimes \vP_d\right)\vmu=\vnull$,
\eit}
where $\vJ_d$ is the d-dimensional matrix only containing 1s and $\vP_d:=\vI_d-1/d\cdot \vJ_d$ is the centring matrix. 
For interpretational purposes it is sometimes helpful to decompose 
the component-wise means as
$$
\mu_{i,t}= \mu + \alpha_i + \beta_t + (\alpha\beta)_{it},\quad i=1,\dots,a;\ t=1,\dots,d,
$$
where $\alpha_i\in\R$ represents the $i$-th group effect, $\beta_t\in \R$ the time effect at time point $t$ and 
$(\alpha\beta)_{it}\in \R$ the $(i,t)$-interaction effect between group and time with the usual side conditions 
$\sum_i \alpha_i = \sum_t \beta_t = \sum_{i,t} (\alpha\beta)_{it} = 0$. With this notation the above null hypothesis can be 
rewritten as (a) $H_0^a: \alpha_i\equiv 0 \text{ for all } i$, 
(b) $H_0^b: \beta_t\equiv 0 \text{ for all } t$ and (c) $H_0^{ab}: (\alpha\beta)_{it}\equiv 0 \text{ for all } i,t$, respectively.

These and other hypotheses will be utilized in the data analysis Section~\ref{Analysis of the Data sets}.


\section{The Test Statistic and its Asymptotics}\label{The Test Statistic and its Asymptotics}

We derive appropriate inference procedures for $H_0(\vT)$ and analyze their asymptotic properties 
under 
%
the following asymptotic frameworks
\begin{eqnarray}
a\in\N\; \text{ fixed}  &\text{and}& \min(d,n_1,\dots,n_a)\to \infty, \label{eq: as frame 1}\\
d\in\N\; \text{ fixed}  &\text{and}& \min(a,n_1,\dots,n_a)\to \infty, \label{eq: as frame 2}\\
&\text{or}& \min(a, d, n_1,\dots,n_a)\to \infty, \label{eq: as frame 3}
\end{eqnarray}
 as $N\to \infty$. Here, no dependency on how the dimension $d=d(N)$ in \eqref{eq: as frame 1} and \eqref{eq: as frame 3} 
or the number of groups $a=a(N)$  in \eqref{eq: as frame 2}-\eqref{eq: as frame 3} converges to infinity with respect to the sample sizes $n_i$ and $N$ 
is postulated. In particular, we cover high-dimensional ($d>n_i$ or even $d>N$) as well as low-dimensional settings. 
For a lucid presentation of subsequent results and proofs we additionally assume throughout that
\bqan
\frac{n_i}{N}\rightarrow \rho_i\in (0,1)\quad i=1,\dots,a.
\eqan
However, by turning to convergent subsequences, all results can be shown to hold under the more general condition $$0<\liminf n_i/N\leq \limsup n_i/N <1,\quad (i=1,\dots,a).$$ 
{It is convenient to measure deviations from the null hypothesis $H_0(\vT):\vT\vmu=\vnull$ 
by means of the quadratic form
\bqan\label{eq: stat}
Q_N=N\cdot\overline {\vX}^\top \vT\overline {\vX},
\eqan}
where 
${\overline{\vX}^\top=(\overline\vX_1^\top,\dots\overline\vX_a^\top)}$ with $\overline {\vX}_{i} = n_i^{-1} \sum_{j=1}^{n_i} \vX_{i,j}, i=1,\dots,a,$ denotes the vector of pooled group means. 

Since $Q_N$ is in general asymptotically degenerated under \eqref{eq: as frame 1}-\eqref{eq: as frame 3} we study its standardized version. 
To this end, note that under the null hypothesis it holds that
{\[\begin{array}{ll}\sqrt{N}\cdot \vT\overline {\vX}&\stackrel{H_0}{\sim}\mathcal{N}_{ad}\left(\boldsymbol{0}_{ad}, \vT \left[\bigoplus\limits_{i=1}^a \frac N{n_i} \vSigma_i\right] \vT \right),\end{array}\]}
 due to assumption \eqref{eq: model}. Thus, it follows from classical theorems about moments of quadratic forms, see e.g. \cite{Buch1} or 
\Cref{QF3} in the supplement, that its mean and variance under the null hypothesis can be expressed as
{\begin{eqnarray}\label{eq: muQ}
\E_{H_0}\left(Q_N\right)&=&\tr\left(\vT\left[\bigoplus_{i=1}^a \frac N{n_i} \vSigma_i\right] \right)=\sum_{i=1}^a \frac{N}{n_i}(\vT_W)_{ii}\tr\left(\vT_S\vSigma_i\right),
\end{eqnarray}
\begin{eqnarray}\nonumber
\Var_{H_0}\left(Q_N\right)&{=}&2\tr\left(\left(\vT\left[\bigoplus\limits_{i=1}^a \frac N{n_i} \vSigma_i\right] \right)^2\right)\\\nonumber  &=&2  \sum\limits_{i=1}^a \sum\limits_{r=1}^a \frac{N^2}{n_in_r}(\vT_W)_{ir}(\vT_W)_{ri}\tr\left(\vT_S\vSigma_i\vT_S\vSigma_r\right)
\\\label{eq: VarQ}
&=&2\sum\limits_{i=1}^a \sum\limits_{r=1}^a \frac{N^2}{n_in_r}{(\vT_W)_{ir}}^2\tr\left(\vT_S\vSigma_i\vT_S\vSigma_r\right)
\\\nonumber
&=&4\sum\limits_{i,r=1, r<i}^a  \frac{N^2}{n_in_r}{(\vT_W)_{ir}}^2\tr\left(\vT_S\vSigma_i\vT_S\vSigma_r\right)+2\sum\limits_{i=1}^a  \frac{N^2}{n_i^2}{(\vT_W)_{ii}}^2\tr\left(\left(\vT_S\vSigma_i\right)^2\right).
\end{eqnarray}
Henceworth we investigate the asymptotic behaviour (under $H_0(\vT)$) of the standardized quadratic form 
$\widetilde W_N  =\{Q_N-E_{H_0}(Q_N)\}/\Var_{H_0}\left(Q_N\right)^{1/2}$. 
{Denoting by $\vV_N:=\bigoplus_{i=1}^a \frac N{n_i} \vSigma_i$ the inversely weighted combined covariance matrix 
the representation theorem for quadratic forms  
\cite{Buch1}[p.90], implies that}
{\bqan\label{eq: tileW}
\widetilde W_N  &=& \frac{Q_N-\E_{H_0}(Q_N)}{\Var_{H_0}\left(Q_N\right)^{1/2}} 
\ \stackrel{\mathcal{D}}{=} \ \sum\limits_{s=1}^{ad} 
\frac{\lambda_s}{\sqrt{\sum_{\ell=1}^{ad} \lambda_\ell^2}}\left(\frac{C_s-1}{\sqrt{2}}\right).
\eqan}
Here '$\stackrel{\mathcal{D}}{=}$' denotes equality in distribution, 
 $\lambda_s$ are the eigenvalues of $\vT\vV_N\vT$ in decreasing order, and $(C_s)_s$ is a sequence of independent 
$\chi_1^2$-distributed random variables. Note, that the eigenvalues $\lambda_s$ also depend on the dimension
$d$ and the sample sizes $n_i$. Transferring the results of \cite{Paper1} for the one-group design with $a=1$ to our general setting, we obtain 
the subsequent asymptotic null distributions of the standardized quadratic form for all asymptotic settings \eqref{eq: as frame 1}-\eqref{eq: as frame 3}. 
\begin{theorem}\label{Theorem3}
{Let $\beta_s={\lambda_s}\Big/{\sqrt{\sum_{\ell=1}^{ad}\lambda_\ell^2}}$ for $s=1,\dots,ad$. }
Then $\widetilde W_N$ has, under $H_0(\vT)$, and one of the 
frameworks \eqref{eq: as frame 1}-\eqref{eq: as frame 3} asymptotically
\begin{itemize} 
\item[a)]a standard normal distribution  if
\[\beta_1 = \max_{s\leq ad} \beta_{s} \to 0 \hspace{0.5cm} \text{as}\hspace{0.2cm} N \to \infty,\]
\item[b)]a standardized $\left(\chi_1^2-1\right)/\sqrt{2}$ distribution if
\[\beta_1\to 1 \hspace{0.5cm} \text{as}\hspace{0.2cm} N \to \infty,\]
\item[c)]the same distribution as the random variable  $\sum_{s=1}^\infty b_s \left(C_s-1\right)/\sqrt 2$,  if
\[\text{for all } s\in \N \hspace{0.5cm}\beta_s\to b_s \hspace{0.5cm} \text{as }\hspace{0.2cm} N \to \infty,\]
for a decreasing sequence $(b_s)_s$ in $[0,1]$ with $\sum_{s=1}^\infty b_s^2=1$.
\end{itemize}
\vspace{-.5cm}
\end{theorem}
It is worth to note that the influence of the different asymptotic frameworks are hidden in the corresponding conditions on 
the sequence of standardized eigenvalues $(\beta_s)_s$, which depend on both, $a$ and $d$. 

Since these quantities are unknown in general we cannot apply the result directly. 
In particular, we are not even able to calculate the test statistic ${\widetilde W}_N$, not to mention to choose its correct limit distribution. 
To this end, we first introduce novel unbiased estimates of the unknown traces involved in \eqref{eq: muQ}-\eqref{eq: VarQ} and discuss their mathematical properties. 
Plugging them into \eqref{eq: muQ}-\eqref{eq: VarQ} leads to the 
calculation of adequately standardized test statistics. Finally, the choice of proper critical values is discussed in Section~\ref{Better Approximations}.

\subsection{Symmetrized Trace Estimators}

Here we derive unbiased and ratio-consistent estimates for the unknown traces $\tr\left(\vT_S \vSigma_i\right), 
\tr\left((\vT_S \vSigma_i\right)^2)$ and $\tr\left(\vT_S\vSigma_i \vT_S \vSigma_r\right), i\neq r,$ given in \eqref{eq: muQ}-\eqref{eq: VarQ}.
Since it is not obvious that the usual plug-in estimates that are based on empirical covariance matrices are useful in high-dimensional settings we follow the approach of \cite{Brunner, Paper1} and directly estimate the traces. Different, to the one-sample design studied therein we face the problem of additional nuisance parameters -- the mean vectors $\vmu_i$.  To avoid their estimation we adopt 
Tyler's symmetrization trick from $M$-estimates of scatter (see e.g. \cite{Croux:1994}, \cite{Duembgen:1998}  or \cite{Tyler:2009}) to the present situation, see also 
\cite{Brunner0}. 
In particular, we consider differences of observation pairs $(\ell_1,\ell_2), \ell_1\neq \ell_2,$ from the same group which fulfill
$\left({\vX}_{i,\ell_1}-{\vX}_{i,\ell_2}\right)\sim\mathcal{N}_d\left(\boldsymbol{0}_d,2\vSigma_i\right)$ and introduce 
the following novel estimators for $i= 1,\dots,a:$
{\begin{align}
A_{i,1}&=\frac 1 {2\cdot \binom {n_i} {2}} \sum\limits_{\begin{footnotesize}\substack{\ell_1,\ell_2=1\\ \ell_1>\ell_2}\end{footnotesize}}^{n_i} \left({\vX}_{i,\ell_1}-{\vX}_{i,\ell_2}\right)^\top \vT_S\left({\vX}_{i,\ell_1}-{\vX}_{i,\ell_2}\right),
\label{eq: A1}
\\
A_{i,r,2}&=\frac  1 {4\cdot\binom {n_i} {2}\binom {n_r} {2}} {\sum\limits_{\begin{footnotesize}\substack{\ell_1,\ell_2=1\\ \ell_1>\ell_2}\end{footnotesize}}^{n_i}\sum\limits_{\begin{footnotesize}\substack{k_1,k_2=1\\ k_1>k_2}\end{footnotesize}}^{n_r}\left[\left({\vX}_{i,\ell_1}-{\vX}_{i,\ell_2}\right)^\top \vT_S\left({\vX}_{r,k_1}-{\vX}_{r,k_2}\right)\right]^2},
\label{eq: A2}
\\
A_{i,3}&=\frac 1 {4\cdot6\binom {n_i} {4}} \sum\limits_{\begin{footnotesize}\substack{\ell_1,\ell_2=1\\ \ell_1>\ell_2}\end{footnotesize}}^{n_i}
\sum\limits_{\begin{footnotesize}\substack{k_2=1\\k_2\neq \ell_1\neq \ell_2 }\end{footnotesize}}^{n_i}\sum\limits_{\begin{footnotesize}\substack{k_1=1\\ \ell_2\neq \ell_1\neq k_1>k_2}\end{footnotesize}}^{n_i}
\left[\left({\vX}_{i,\ell_1}-{\vX}_{i,\ell_2}\right)^\top \vT_S\left({\vX}_{i,k_1}-{\vX}_{i,k_2}\right)\right]^2,
\label{eq: A3}
\\
A_4&=\sum_{i=1}^a \left(\frac{N}{n_i }\right)^2 {(\vT_W)_{ii}}^2 A_{i,3}+2\sum_{i=1}^a\sum_{r=1 , r < i}^a \frac{N^2}{n_in_r } {(\vT_W)_{ir}}^2 A_{i,r,2}
\label{eq: A4}.
\end{align}}\\
Here and throughout the paper expressions of the kind  $a\neq b\neq c$  mean that the indices are pairwise different. In this sense all 
estimators \eqref{eq: A1}-\eqref{eq: A4} are {\it symmetrized U-statistics}, where the kernel is given by a specific quadratic or bilinear form. 
Their properties are analyzed below.\vspace{-.5cm}
\begin{LeA}\label{Schae1}
For any $\vmu\in\renu^{ad}$ and $i=1,\dots, a$ it holds that 
{\begin{enumerate}
\item $\widehat{E_{H_0}}(Q_N) := \sum_{i=1}^a \frac {N}{n_i} {(\vT_W)_{ii}}A_{i,1}$  is an unbiased and ratio-consistent estimator for $\E_{H_0}(Q_N)$.
\item $A_4$ is an unbiased and ratio-consistent estimator for $\tr\left(\left(\vT\vV_N \right)^2\right).$ 
\item $A_{i,1},A_{i,r,2}$ and $A_{i,3}$ are unbiased and ratio-consistent estimators for $\tr\left(\vT_S\vSigma_i\right),\tr\left(\vT_S\vSigma_i\vT_S\vSigma_r\right)$ and $\tr\left(\left(\vT_S\vSigma_i\right)^2\right),$ respectively.
\end{enumerate}}
\end{LeA}

\begin{re}
(a) Recall that an $\renu$-valued estimator $\widehat{\theta}_N$ is ratio-consistent for a sequence of real parameters $\theta_N$ iff $\widehat{\theta}_N/\theta_N \to 1$ in probability as $N\to\infty$. Here the estimators and parameters may depend on $a = a(N)$ and/or $d=d(N)$.\\
(b) Studying the proof of Lemma~\ref{Schae1} given in the supplementary material in detail, we see that all estimators are even (dimension-)stable in the sense of \cite{Brunner}, i.e. they fulfill $|\E(\widehat{\theta}_N/\theta_N - 1)| \leq b_N$ and  $\Var(\widehat{\theta}_N/\theta_N) \leq c_N$ for sequences $b_N,c_N\downarrow 0$ not depending on $a$ and $d$.\vspace{-.5cm}
\end{re}
It follows from Lemma~\ref{Schae1} that 
$$
\widehat{Var_{H_0}}(Q_N) := 2\sum_{i=1}^a \left(\frac{N}{n_i }\right)^2{(\vT_W)_{ii}}^2A_{i,3}+4\sum_{i=1}^a\sum_{r=1 , r < i}^a\frac{N^2}{n_in_r} {(\vT_W)_{ir}}^2A_{i,r,2}=2A_4
$$
is an unbiased estimator of $Var_{H_0}(Q_N)$. This motivates to study the standardized quadratic form 
\[ W_N=\frac { Q_N- \widehat{E_{H_0}}(Q_N) }{\widehat{Var_{H_0}}(Q_N)^{1/2}}\]\\
for testing $H_0(T)$. Its asymptotic behaviour under $H_0(T):\vT\vmu=\vnull_{ad}$ is summarized below.
\vspace{-.5cm}
\begin{theorem}\label{Theorem4}
Under $H_0(T):\vT\vmu=\vnull_{ad}$ and one of the 
frameworks \eqref{eq: as frame 1}-\eqref{eq: as frame 3} the statistic $W_N$ has the same asymptotic limit distributions as $\widetilde W_N$, if the respective conditions (a)-(c) from \Cref{Theorem3} are fulfilled.
\vspace{-.5cm}
\end{theorem}

The result shows that it is not reasonable to approximate the unknown distribution of the test statistic with a fixed distribution to obtain a valid test procedure. For example, choosing $z_{1-\alpha}$, the $(1-\alpha)$-quantile of the standard-normal distribution ($\alpha\in(0,1)$), as critical value 
would lead to a valid asymptotic level $\alpha$ test $\psi_z=\mathbf{1}\{W_N > z_{1-\alpha}\}$ in case of $\beta_1\to 0$, i.e. $\E_{H_0}(\psi_z)\to\alpha$. However, for $\beta_1\to 1$ we would obtain $\E_{H_0}(\psi_z)\to P(\chi_1^2>\sqrt{2}z_{1-\alpha} + 1)$ which may lead to an 
asymptotically liberal ($\alpha=0.01$ or $0.05$) or conservative ($\alpha=0.1$) test decision, see Table~\ref{tab: Level}.
Contrary, choosing $c_{1-\alpha}=(\chi_{1;1-\alpha}^2-1)/\sqrt{2}$ as critical value (where $\chi_{1;1-\alpha}^2$ denotes the $(1-\alpha)$-quantile of the $\chi_1^2$-distribution) for the test $\psi_\chi = \mathbf{1}\{W_N > c_{1-\alpha}\}$, 
it follows that 
$\E_{H_0}(\psi_\chi)\to\alpha$ if $\beta_1\to 1$ but $\E_{H_0}(\psi_\chi)\to 1-\Phi(c_{1-\alpha})$ for $\beta_1\to 0$, 
where  $\Phi$ denotes the cumulative distribution function of $\mathcal{N}(0,1)$. 
Again we obtain an asymptotically liberal ($\alpha=0.1$) or extremely conservative ($\alpha= 0.05$ or $0.01$) test decision, see the last column of Table~\ref{tab: Level}.
\begin{table}[h]

\caption{\small Asymptotic levels of the tests $\psi_z$ and $\psi_\chi$ with fixed critical values under the null hypothesis and all asymptotic frameworks \eqref{eq: as frame 1}-\eqref{eq: as frame 3}.}
\label{tab: Level}
\bcen
\begin{tabular}{|c|c|c|c|c|}\hline
 chosen  & \multicolumn{4}{c|}{True asymptotic level of the test}\\
level $\alpha$ & $\psi_z$ ($\beta_1\to 0$) & $\psi_z$ ($\beta_1\to 1$) &  $\psi_\chi$ ($\beta_1\to 0$) & $\psi_\chi$ ($\beta_1\to 1$) \\\hline
0.10 & 0.10 & 0.09354& 0.11391 & 0.10\\
0.05 & 0.05 & 0.06819 &  0.02226  & 0.05\\
0.01 & 0.01 & 0.03834 & 0.00003 & 0.01\\\hline
 \end{tabular}
\ecen
\vspace{-.5cm}
\end{table}

Hence, an indicator (i.e. estimator) for whether $\beta_1\to 0$, $\beta_1\to 1$ or betwixt would be desirable. Nevertheless, even if the tests with fixed critical values are asymptotically correct ($\psi_z$ in case of $\beta_1\to0$ or $\psi_\chi$ in case of $\beta_1\to1$), their true type-$I$-error control may be poor for small sample sizes, see the simulations in \Cref{Asymptotic distribution}.\\
Thus, in any case it seems more appropriate to approximate $W_N$ by a sequence of standardized distributions as already advocated in \cite{Paper1} for the case of $a=1$. We will propose such approximations in the next Sections, where also a check criterion for $\beta_1\to 0$ or  $\beta_1\to 1$ is presented.


\section{Better Approximations}\label{Better Approximations}

To motivate the subsequent approximation, recall from \eqref{eq: tileW} that $\widetilde{W}_N$ is of weighted $\chi^2$-form. Following \cite{Zhang} it is reasonable to approximate statistics of this from by a standardized $(\chi^2_f-1)/\sqrt{2}$-distribution such that the first three moments coincide. Straightforward calculations show that this is achieved by approximating with
\begin{equation}\label{eq: Kf}
 K_{f_P}=\frac{\chi_{f_P}^2-{f_P}}{\sqrt{ 2 {f_P}}} \quad\text{ such that }\quad {f_P}= \frac{\tr^3\left(\left(\vT\vV_N\right)^2\right)}{\tr^2\left(\left(\vT\vV_N\right)^3\right)}.
\end{equation}
In case of $a=1$ this simplifies to the method presented in \cite{Paper1}. There it has already been seen that the approximation \eqref{eq: Kf} performs much better for smaller sample sizes and/or dimensions than the above approaches with a fixed distribution. We will later rediscover this observation in Section~\ref{Simulations} for our present design with general $a$. The next theorem gives a mathematical reason for this approximation.
\vspace{-.5cm}
\begin{theorem}\label{Theorem5}
Under the conditions of Lemma~\ref{Theorem3} and one of the 
frameworks \eqref{eq: as frame 1}-\eqref{eq: as frame 3} we have that $K_{{f_P}}$ given in \eqref{eq: Kf} has, under $H_0:\vT\vmu=\boldsymbol{0}_{ad}$, asymptotically 
\begin{itemize} 
\item[a)]a standard normal distribution  if $\beta_1\to 0$ as $N \to \infty$,
\item[b)]a standardized $\left(\chi_1^2-1\right)/\sqrt{2}$ distribution if $\beta_1\to 1$ as $N \to \infty$.
\end{itemize}
\vspace{-.5cm}
\end{theorem}
Thus, compared to the approximation with a fixed limit distribution, the $K_{f_P}$-approach would at least be asymptotically correct whenever 
$\beta_1\to \gamma \in \{0,1\}$ while always providing a three moment approximation to the test statistic. To apply this result, an estimator for $f$ in \eqref{eq: Kf} is needed. Since we have already found $A_4$ as unbiased and ratio-consistent estimator for $\tr(\left(\vT\vV_N\right)^2)$, 
it remains to find an adequate one for $\tr(\left(\vT\vV_N\right)^3)$. A combination of both will then lead to a proper estimator for ${f_P}$ and $\tau_P={f_P}^{-1}$, respectively. 
Again we prefer a direct estimation of the involved traces. To this end, we introduce normal random vectors
\[\vZ_{(\ell_1,\ell_2,\dots,\ell_{2a})}:=\left(\sqrt{\frac{N}{n_1}}\left(\vX_{1,\ell_1}-\vX_{1,\ell_2}\right)^\top\textbf{,}\dots\textbf{,}\sqrt{\frac{N}{n_a}}\left(\vX_{a,\ell_{2a-1}}-\vX_{a,\ell_{2a}}\right)^\top\right)^\top \]
with $1\leq \ell_{2i-1}\neq \ell_{2i}\leq n_i$ for all $i=1\dots,a$. Note, that this vectors are multivariat normal distributed with $\E(\vZ_{\left(\ell_1,\ell_2,\dots,\ell_{2a-1},\ell_{2a}\right)})=\vnull_{ad}$ and $\cov\left(\vZ_{\left(\ell_1,\ell_2,\dots,\ell_{2a-1},\ell_{2a}\right)}\right)=2\bigoplus_{i=1}^a\frac N{n_i} \vSigma_i = 2 \vV_N$. 
Utilizing their particular form, it is shown in the supplement, that a cyclic combination of these random vectors yield an unbiased estimator for $\tr(\left(\vT\vV_N\right)^3)$. In particular, writing 
$\vZ_{(\ell_1,\ell_2)}$ for $\vZ_{(\ell_1,\ell_2,\ell_1,\ell_2,\dots,\ell_{1},\ell_2)}$ we have
\bqan\label{eq: EV trtv3}
\E\left({\vZ_{(1,2)}}^\top 
 \vT
 \vZ_{(3,4)}
 {\vZ_{(3,4)}}^\top \vT \vZ_{(5,6)}
 {\vZ_{(5,6)}}^\top \vT \vZ_{(1,2)}\right)
 = 8 \tr(\left(\vT\vV_N\right)^3).
\eqan
This motivates the definition of (for $n_i\geq 6$)
\bqan\label{eq:C5}
C_{5}=\sum\limits_{\begin{footnotesize}\substack{\ell_{1,1},\dots, \ell_{6,1}=1\\
\ell_{1,1}\neq\dots\neq \ell_{6,1}}\end{footnotesize}
}^{n_1}\dots\sum\limits_{\begin{footnotesize}\substack{\ell_{1,a},\dots, \ell_{6,a}=1\\
\ell_{1,a}\neq\dots\neq \ell_{6,a}}\end{footnotesize}
}^{n_a}
\frac{\Lambda_{1}(\ell_{1,1},\dots,\ell_{6,a})\cdot \Lambda_{2}(\ell_{1,1},\dots,\ell_{6,a})\Cdot \Lambda_{3}(\ell_{1,1},\dots,\ell_{6,a})}{8\cdot \prod\limits_{i=1}^a\frac {  n_i! }{\left(n_i-6\right)!}},
\eqan
where
\[\Lambda_{1}(\ell_{1,1},\dots,\ell_{6,a})={\vZ_{(\ell_{1,1},\ell_{2,1},\dots,\ell_{1,a},\ell_{2,a})}}^\top \vT \vZ_{(\ell_{3,1},\ell_{4,1},\dots,\ell_{3,a},\ell_{4,a})},\]
\[\Lambda_{2}(\ell_{1,1},\dots,\ell_{6,a})={\vZ_{(\ell_{3,1},\ell_{4,1},\dots,\ell_{3,a},\ell_{4,a})}}^\top \vT \vZ_{(\ell_{5,1},\ell_{6,1},\dots,\ell_{5,a},\ell_{6,a})}, \]
\[\Lambda_{3}(\ell_{1,1},\dots,\ell_{6,a})={\vZ_{(\ell_{5,1},\ell_{6,1},\dots,\ell_{5,a},\ell_{6,a})}}^\top \vT \vZ_{(\ell_{1,1},\ell_{2,1},\dots,\ell_{1,a},\ell_{2,a})}.\]
Its properties together with a consistent estimator for $f_P$ are summarized below.
\vspace{-.5cm}
\begin{Le}\label{Lemma: fP Estimate}
 (a) The estimator $C_{5}$ given in \eqref{eq:C5} is unbiased for $\tr(\left(\vT\vV_N\right)^3)$.\\
 (b) Suppose that $a\in\N$ is fixed. Then $\widehat{\tau}_P := C_{5}^2/A_4^3$ is a consistent estimator for $\tau_P=1/f_P$ as \\ ${\min(d,n_1,\dots,n_d)\to\infty}$, i.e. we have convergence in probability
 \bqan\label{eq:consistency taup}
 \widehat{\tau}_P - \tau_P = \frac{C_{5}^2}{A_4^3} - \frac{\tr^2\left(\left(\vT\vV_N\right)^3\right)}{\tr^3\left(\left(\vT\vV_N\right)^2\right)} \stackrel{p}{\longrightarrow} 0.
 \eqan
 (c) Now suppose that $a\to \infty$ and that there exists some $p>1$ such that $\min(n_1,\dots,n_a)=\0\left(a^p\right)$. Then \eqref{eq:consistency taup} even holds under the asymptotic frameworks \eqref{eq: as frame 2} - \eqref{eq: as frame 3}.
\end{Le}
\begin{theorem}\label{theo: fP}
 Suppose \eqref{eq:consistency taup}. Then,  \Cref{Theorem5} remains valid if we replace $f_P$ by its estimator $\widehat{f}_P = 1/\widehat{\tau}_P$.
\vspace{-.5cm}
\end{theorem}
\begin{re}
 (a) Using similar arguments as in the proof of Lemma~8.1. of \cite{Paper1} we obtain the equivalences
 $\beta_1 \to 0 \Leftrightarrow \tau_P\to 0$ and $\beta_1 \to 1 \Leftrightarrow \tau_P\to 1$. Thus, $\widehat{\tau}_P$ can also be used as check criterion for these two cases. \\
 (b) It is also possible to derive a consistent estimator for $\tau_{CQ} = {\tr(\left(\vT\vV_N\right)^4)}/{\tr^2(\left(\vT\vV_N\right)^2)}$ $ {= 1/f_{CQ}}$, a key quantity in \cite{bib8a}, see the supplement for details concerning the estimator. The corresponding approximation by the sequence $K_{f_{CQ}}$ even shares the same asymptotic properties of the Pearson approximation \eqref{eq: Kf} stated in \Cref{Theorem5} and \Cref{theo: fP}. However, it only provides a two moment approximation which turned out to perform worse in simulations (results not shown).\\
 (c) In the supplement, we additionally present an unbiased estimator $C_7$ for $\tr(\left(\vT\vV_N\right)^3)$ such that $C_7^2/A_4^3$ is consistent for $\tau_P$ in all 
 asymptotic frameworks \eqref{eq: as frame 1} - \eqref{eq: as frame 3}. Particularly, the extra condition $\min(n_1,\dots,n_a)=\0\left(a^p\right)$ is not needed. 
 However, it is computationally more expensive compared to $C_5$ and thus omitted here.
\vspace{-.5cm}
 \end{re}
In practical applications, the computation costs for $C_5$ are nevertheless rather high. This leads to disproportional waiting times for $p$-values of the 
corresponding approximate test $\varphi_N = \mathbf{1}\{W_N > K_{\widehat{f}_P; 1-\alpha}\}$, where the critical value is given as $(1-\alpha)$-quantile of $K_{\widehat{f}_P}$. Therefore, we propose a certain subsampling-type method. Since the unbiasedness of $C_5$ clearly stems from \eqref{eq: EV trtv3} it seems reasonable to proceed as follows: 
For each $i=1,\dots,a$ and $b=1,\dots,B$ we independently draw random subsamples $\{\sigma_{1i}(b),\dots,\sigma_{6i}(b)\}$ of length $6$ from $\{1,\dots,n_i\}$ and store them in a joint random vector $\vsigma(b) = (\sigma_{11}(b),\dots,\sigma_{6a}(b))$. Then, a subsampling-version of the estimator $C_5$ is given by
\[{C_5^\star} = {C_5^\star}\left(B\right)=\frac 1 {8 \cdot B}\sum\limits_{b=1}^B 
 \Lambda_1(\vsigma(b)) \cdot \Lambda_2(\vsigma(b)) \cdot \Lambda_3(\vsigma(b)). 
\]
Letting $B=B(N)\to\infty$ as $N\to \infty$ it is easy to see (see the supplement for details), that ${C_5^\star}$ has the same asymptotic properties as ${C_5}$. 
In particular, it is stated in the supplement that $\widehat{\tau}_P^\star := 1/\hat{f}_P^\star := C_5^{\star 2}/A_4^3$ is a consistent estimator for $\tau_P$ and that the approximation $K_{\hat{f}_P^\star}$ has the same weak limits as $K_{\hat{f}_P}$ stated in \Cref{theo: fP}. 
This leads to $\varphi_N^\star = \mathbf{1}\{W_N > K_{\hat{f}^\star_P; 1-\alpha}\}$ which is an asymptotically exact test whenever $\beta_1\to \gamma\in\{0,1\}$.
The finite sample, dimension and group size performance of this approximation are investigated in the subsequent section.


\section{Simulations}\label{Simulations}
In the previous sections we considered the asymptotic properties of the proposed inference methods which are valid for large sample and fixed or possibly large dimension and/or group sizes.
Here we investigate the small sample properties of our proposed approximation procedure \\${\varphi_N^\star = \mathbf{1}\{W_N > K_{\hat{f}^\star_P; 1-\alpha}\}}$ in comparison to the statistical tests $\psi_z$ and $\psi_\chi$ based on fixed critical values. In particular, we compare these procedures in simulation studies with respect to
\begin{itemize}
 \item[(a)] their type I error rate control under the null hypothesis (Section~\ref{Asymptotic distribution}) and
 \item[(b)] their power behaviour under various alternatives (Section~\ref{seq:power}).
\end{itemize}
All simulations were performed with the help of the R computing environment (R Development Core Team, 2013), each with $n_{sim}=10^4$ simulation runs.

\subsection{Asymptotic distribution and Type I error control}\label{Asymptotic distribution}

First we study the speed of convergence, i.e. type I error control, of the three different tests under the null hypothesis. 
To be in line with the simulation results presented in \cite{Paper1} for the case $a=1$ we also multiplied the statistic $W_N$ by $\sqrt{{N}/{(N-1)}}$ to avoid a slightly liberal behaviour.

Due to the abundance of different split-plot designs and the more methodological focus of the paper, we restrict our simulation study to two specific null hypotheses and a high dimensional and heteroscedastic two-sample setting. 
In particular, we investigate the type-I-error behaviour of all three tests for the null hypotheses 
\begin{itemize}
 \item ${H_0^a: \left(\vP_2\otimes \frac{1}{d}\vJ_d\right)\vmu=\vnull}$ and
 \item ${H_0^b: \left(\frac 1 2\vJ_2\otimes \vP_d\right)\vmu=\vnull}$.
 \end{itemize}
In both cases sample sizes were chosen from $n_1\in\{10,20,50\}$ and $n_2\in\{15,30,75\}$ combined with various choices of dimensions $d\in\{5,10,20,40,70,100,150,200,$ $300,450,600,800\}$. For the covariance  matrices a heteroscedastic setting with  autoregressive structures  $\left(\vSigma_1\right)_{i,j}=0.6^{|i-j|}$ and  $\left(\vSigma_2\right)_{i,j}=0.65^{|i-j|}$ was chosen and for each simulation run $B(N)= 500\Cdot N,$ $N=n_1+n_2,$ subsamples were drawn.

Note that these settings imply $\beta_1\to 1$ for $H_0^a$ and $\beta_1\to 0$ for $H_0^b$, see the supplement for details.

\begin{figure}[H]
\setlength{\abovecaptionskip}{5pt} 
\setlength{\belowcaptionskip}{1pt} 
\begin{minipage}[t]{0.49\textwidth}\vspace{0pt} 
\includegraphics[width=\textwidth]{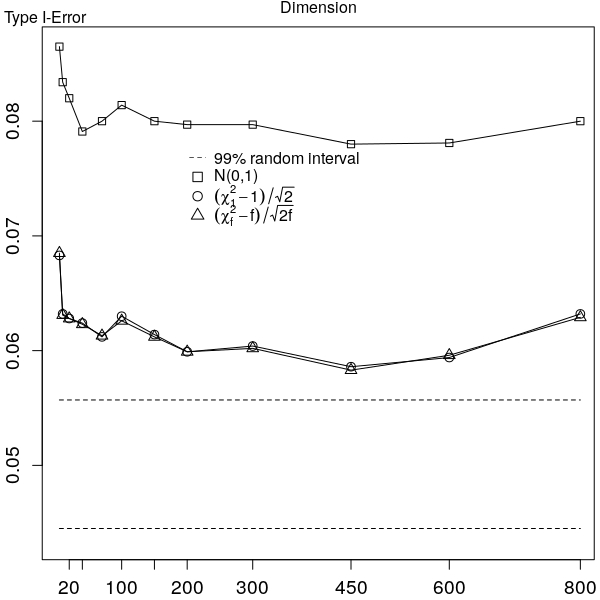} 
\end{minipage}\hfill%
\begin{minipage}[t]{0.49\textwidth}\vspace{0pt} 
\includegraphics[width=\textwidth]{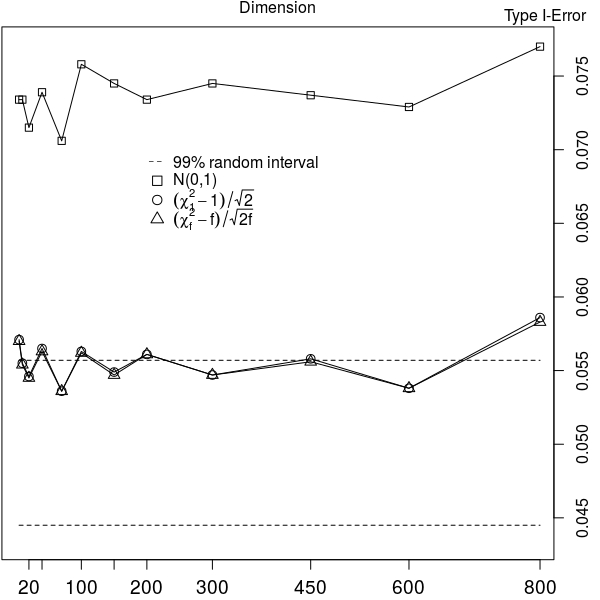} 
\vspace{-1.0cm}\end{minipage} \\
\begin{minipage}[t]{0.49\textwidth}\vspace{0pt} 
\includegraphics[width=\textwidth]{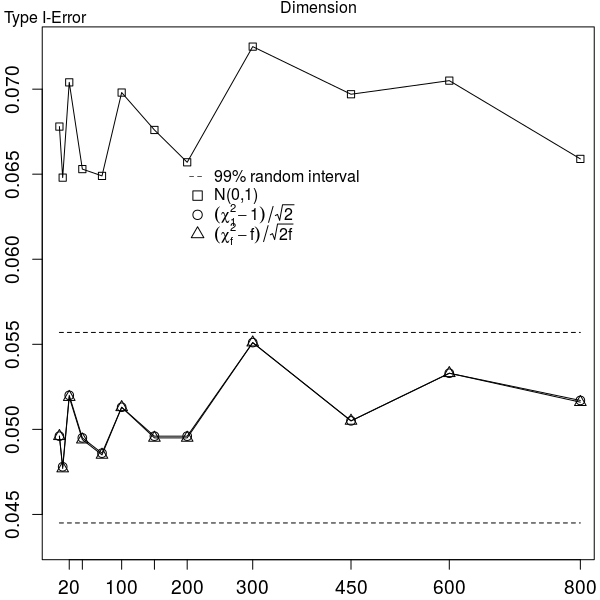} 
\end{minipage} 

\caption{\small Simulated Type I-Error rates ($\alpha=5\%$) for the statistic $W_N\cdot\sqrt{N/(N-1)}$ compared with the critical values of a standard normal, standardized $\chi_1^2$ and $K_f$-distribution under the null hypothesis $H_0^a: \left(\vP_2\otimes \frac{1}{d}\vJ_d\right)\vmu=\vnull$ for increasing dimension. 
 The sample sizes are increased from left ($n_1=10, n_2=15$) to right ($n_1=20, n_2=30$) to bottom ($n_1=50, n_2=75$).} 
\label{BildSimulation1} 
\end{figure} 

Thus, $\varphi_N^\star$ is asymptotically exact in both cases while $\psi_\chi$ and $\psi_z$ posses the asymptotic behaviour given in Table~1. In particular, the $z$-test $\psi_z$ should be rather liberal for testing for $H_0^a$ and $\psi_\chi$ strongly conservative for $H_0^b$. All these theoretical findings can be recovered in our simulations:
The results for $H_0^a$, displayed in Figure~\ref{BildSimulation1}, show an inflated type I error level control of $\psi_z$ around $8\%$ for smaller samples sizes ($N=25$). For larger sample sizes ($N=125$) it stabilizes in the region of its asymptotic level of $6.8\% \pm 0.2\%$. Moreover, the error control is only slightly effected by the varying dimensions under investigation. In comparison, the two asymptotically correct tests $\varphi_N^\star$ and $\psi_\chi$ are slightly liberal for smaller sample sizes and more or less asymptotically correct for moderate ($N=50$) to larger sample sizes. Here, it is astonishing that both procedures are nearly superposable, suggesting a fast convergence of the degrees of freedom estimator $\widehat{f_P}$.


\begin{figure}[H] 

\begin{minipage}[t]{0.49\textwidth}\vspace{0pt} 
\includegraphics[width=\textwidth]{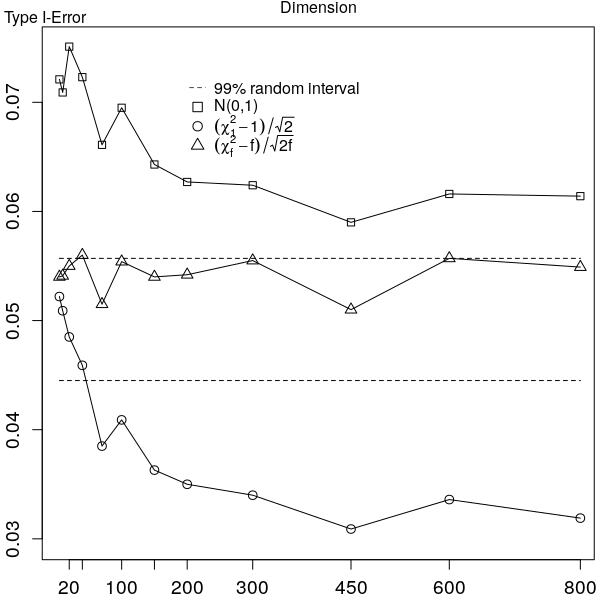} 
\end{minipage}\hfill%
\begin{minipage}[t]{0.49\textwidth}\vspace{0pt} 
\includegraphics[width=\textwidth]{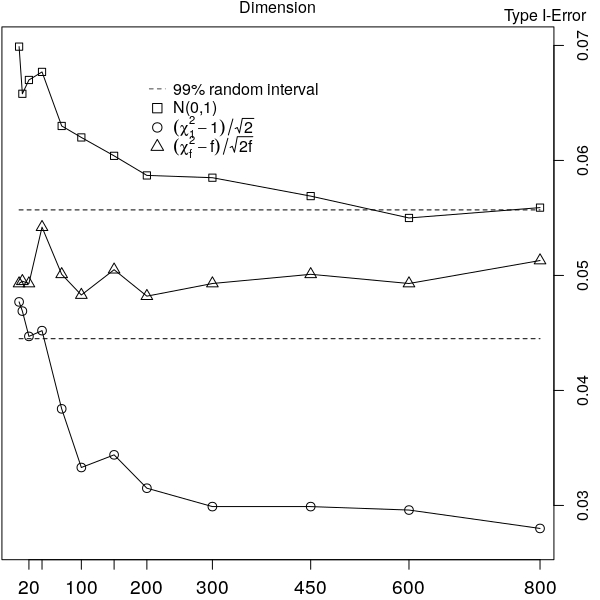} 
\vspace{-2.5cm}\end{minipage} \\
\begin{minipage}[t]{0.49\textwidth}\vspace{0pt} 
\includegraphics[width=\textwidth]{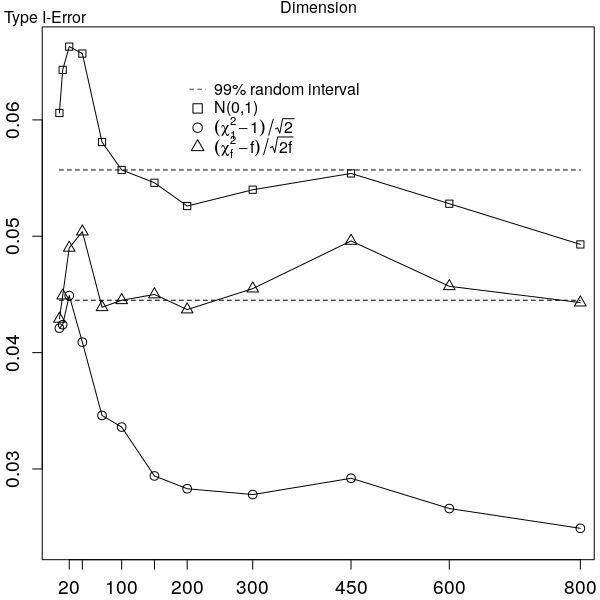} 
\end{minipage} 
\setlength{\abovecaptionskip}{5pt} 
\setlength{\belowcaptionskip}{0pt} 
\caption{\small Simulated Type I-Error rates ($\alpha=5\%$) for the statistic $W_N\cdot\sqrt{N/(N-1)}$ compared with the critical values of a standard normal, standardized $\chi_1^2$ and $K_f$-distribution under the null hypothesis $H_0^b:\left(\frac 1 2\vJ_2\otimes \vP_d\right)\vmu=\vnull$ for increasing dimension. 
The sample sizes are increased from left ($n_1=10, n_2=15$) to right ($n_1=20, n_2=30$) to bottom ($n_1=50, n_2=75$).}  
\label{BildSimulation2} 
\end{figure}

The results for $H_0^b$, presented in Figure~\ref{BildSimulation2}, are slightly different. In particular, both the tests $\psi_\chi$ and $\psi_z$ depending on fixed critical values are more effected by the underlying dimension: For smaller $d<100$ the true level is considerably larger than their asymptotic level given in Table~1; resulting in a rather liberal behaviour of $\psi_z$ and close to exact type I error control for $\psi_\chi$. This effect is decreased with increasing sample sizes. Moreover, for 
larger dimension ($d\geq 200$) both tests approach their asymptotic level. In comparison, the procedure $\varphi_N^\star$ based on the $K_{\hat{f}^\star}$ approximation shows a fairly good $\alpha$ level control through all dimension and sample size settings. Making this the method of choice. 

\subsection{Power Performance}\label{seq:power}
We examined the power of the three procedures. 
Again a  heteroscedastic two group split-plot design with autoregressive covariance structures ( $\left(\vSigma_1\right)_{i,j}=0.6^{|i-j|}$ and  $\left(\vSigma_2\right)_{i,j}=0.65^{|i-j|}$) was selected. 
The alpha level ($5\%$) and the null hypotheses were chosen as above ($H_0^a: \left(\vP_2\otimes \frac{1}{d}\vJ_d\right)\vmu=\vnull$ and
$H_0^b: \left(\frac 1 2\vJ_2\otimes \vP_d\right)\vmu=\vnull$). The investigated alternatives were 
\bit
\setlength\itemsep{0.3em}\setlength\parskip{0em}
\item a trend alternative for both hypotheses with $\vmu_2=\vnull_d$ and $ \mu_{1,t}=t\cdot \delta/d, t\in \N_d$ for $\delta\in[0,3]$ and additionally
\item a shift alternative for $H_0^a$ with $\vmu_2=\vnull_d$ and $\vmu_1=\textbf{1}\cdot \delta$ and
\item a one-point alternative  for $H_0^b$, with $\vmu_2=\vnull_d$ and $\vmu_1=\ve_{1}\cdot \delta$,
\eit
each with increased $\delta\in[0,3]$. We only considered the moderate sample size setting with $n_1=20$ and $n_2=30$ together with three choices of dimensions $d=\{10,40,100\}$. The results can be found in Figures~\ref{BildPowerTAB} and \ref{BildPowerSBOPA}.


\captionsetup[subfloat]{labelformat=empty}
\captionsetup[subfloat]{justification=RaggedRight}
\captionsetup[subfloat]{nearskip=0.3cm}
\captionsetup[subfloat]{farskip=0.2cm}
\captionsetup[subfloat]{margin=15pt}

\begin{figure}[H]
    \subfloat[ Power curves for a trend alternative and ${H_0^a: \left(\vP_2\otimes \frac{1}{d}\vJ_d\right)\vmu=\vnull}$.]
    {\includegraphics[width=0.5\textwidth]{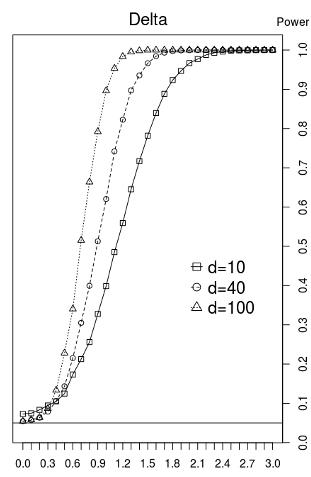}} 
    \subfloat[ Power curves for a trend alternative and ${H_0^b: \left(\frac 1 2\vJ_2\otimes \vP_d\right)\vmu=\vnull}$.]{\includegraphics[width=0.5\textwidth]{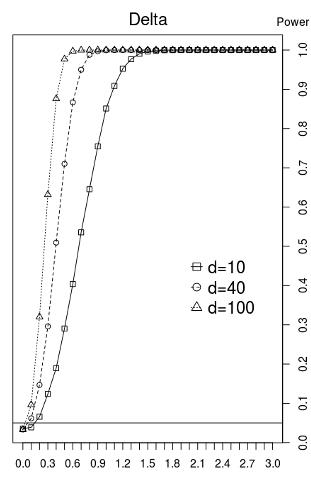}} 
\caption{ \small
Simulated power curves for the Statistic $W_{N}\cdot \sqrt{(N-1)/N}$ in $10^4$ 
simulation runs for different dimensions with $n_1=20, n_2=30$ and an autoregressive structure( $\left(\vSigma_1\right)_{i,j}=0.6^{|i-j|}$ and  $\left(\vSigma_2\right)_{i,j}=0.65^{|i-j|}$).} 
\label{BildPowerTAB} 
\end{figure}

It can be readily seen that the power depends on the type of alternative: For the trend (Figure~\ref{BildPowerTAB}) and the shift alternative (left panel of Figure~\ref{BildPowerSBOPA}) the power gets larger with increasing dimension.   This is essentially apparent for the shift alternative, where the power increases considerably from $d=10$ to $d=40$. Contrary, for the one-point alternative the power becomes smaller for higher dimensions $d$ (right panel of \Cref{BildPowerSBOPA}). However, this is as expected since a difference in one single component can be detected more easily for smaller $d$. 

\begin{figure}[H]
    \subfloat[ Power curves for a shift alternative and ${H_0^a: \left(\vP_2\otimes \frac{1}{d}\vJ_d\right)\vmu=\vnull}$.]{\includegraphics[width=0.5\textwidth]{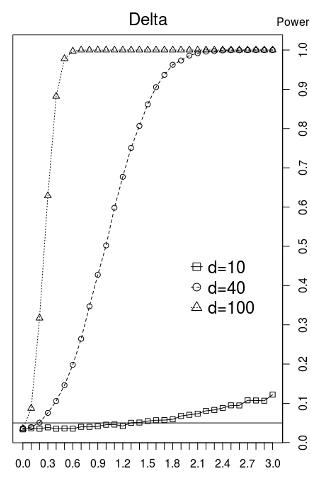}}
    \subfloat[ Power curves for an one-point alternative  and ${H_0^b: \left(\frac 1 2\vJ_2\otimes \vP_d\right)\vmu=\vnull}$.]{\includegraphics[width=0.5\textwidth]{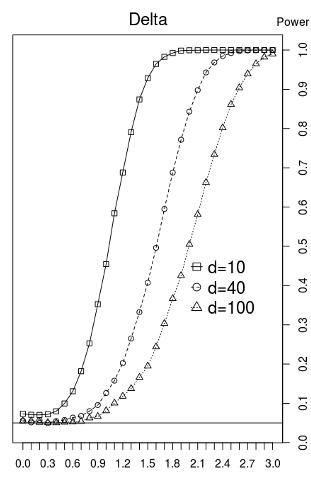}} 
\caption{ \small
Simulated power curves for the Statistic $W_{N}\cdot \sqrt{(N-1)/N}$ in $10^4$ 
simulation runs for different dimensions with $n_1=20, n_2=30$ and an autoregressive structure( $\left(\vSigma_1\right)_{i,j}=0.6^{|i-j|}$ and  $\left(\vSigma_2\right)_{i,j}=0.65^{|i-j|}$).} 
\label{BildPowerSBOPA} 
\end{figure}

 \section{Analysis of a sleep laboratory data set}\label{Analysis of the Data sets}
Finally, the new methods are exemplified on the sleep laboratory trial reported in \cite{Daten}. 
In this two-armed repeated measures trial, the activity of prostaglandin-D-synthase ($\beta$-trace) was measured every 4 hours over a period of 4 days. 
The grouping factor was gender and the above $d=24$ repeated measures were observed on $n_i=10$ young healthy men (group $i=2$) and women (group $i=1$). 
Since each day presented a certain sleep condition the repeated measures are structured by two crossed fixed factors:
\begin{itemize}
\setlength\itemsep{0.3em}\setlength\parskip{0em}
 \item intervention (with $4$ levels: normal sleep, sleep deprivation, recovery sleep and REM sleep deprivation) and 
 \item time (with the $6$ levels/time points $24h, 4h, 8h, 12h, 16h$ and $20h$). 
\end{itemize}
Due to $d>n_i$ we are thus dealing with a high-dimensional split-plot design with  $a=2$ groups and $d=24$ repeated measures. 
The time profiles of each subject are displayed in \Cref{fig:Schlaffrauen} (for the female group $1$) and \Cref{fig:Schlafmaenner} (for the male group $2$). 
We note, that group-specific profile analysis could already be performed by the methods given in \cite{Paper1}. In particular, they found a significant intervention and a borderline time effect for the male group. For the current two-sample design additional questions concern (1) whether there is a gender effect, i.e. the time profiles of the groups differ, 
and if so (2) whether they differ with respect to certain interventions. 
\begin{figure}[h]
    \centering
\includegraphics[width=\textwidth]{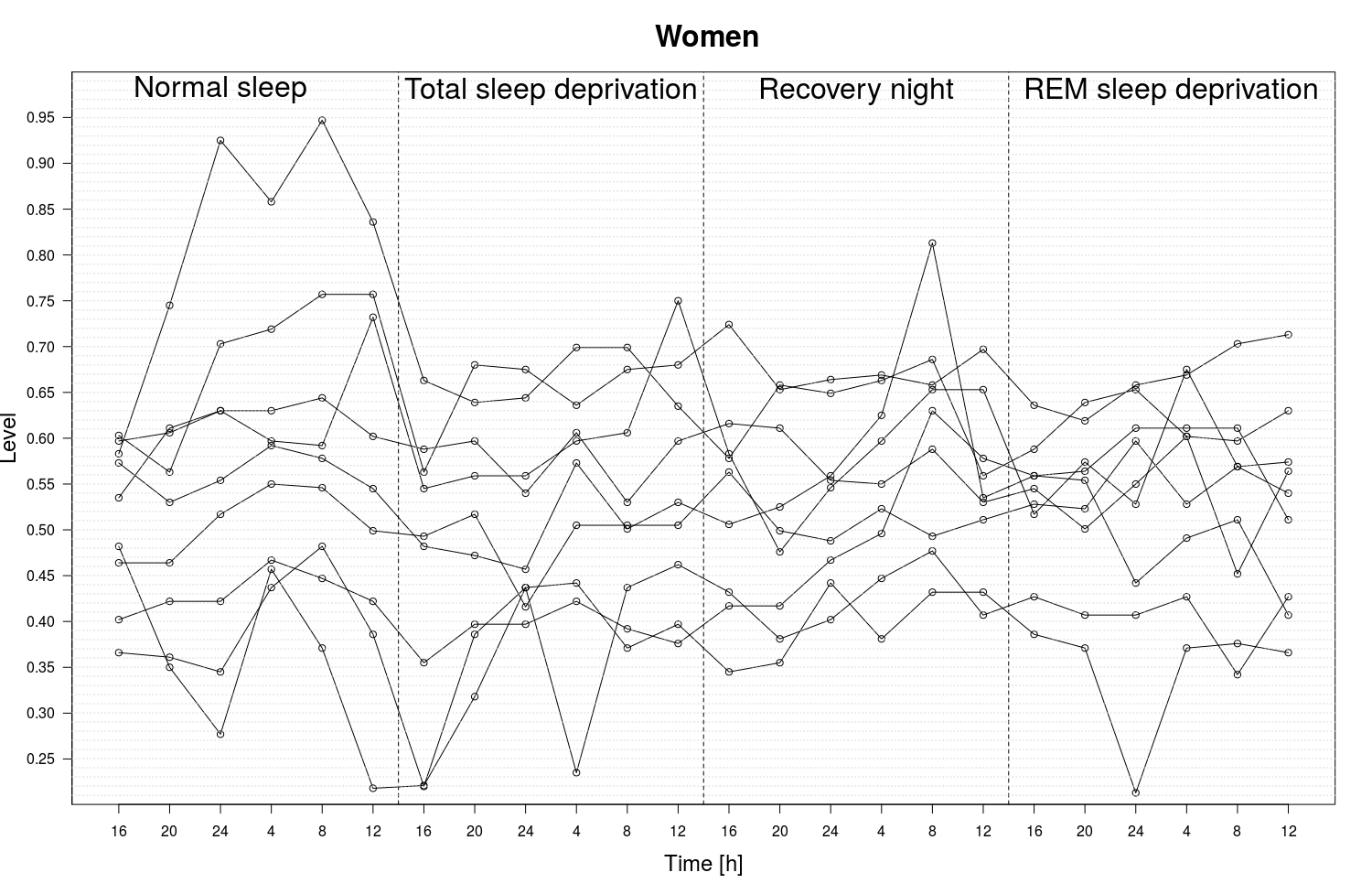}
\caption{\small Prostaglandin-D-synthase (ß-trace) of 10 young women during 4 days  under different sleep conditions.}
\label{fig:Schlaffrauen}
\end{figure}
\begin{figure}[h]
    \centering
\includegraphics[width=\textwidth]{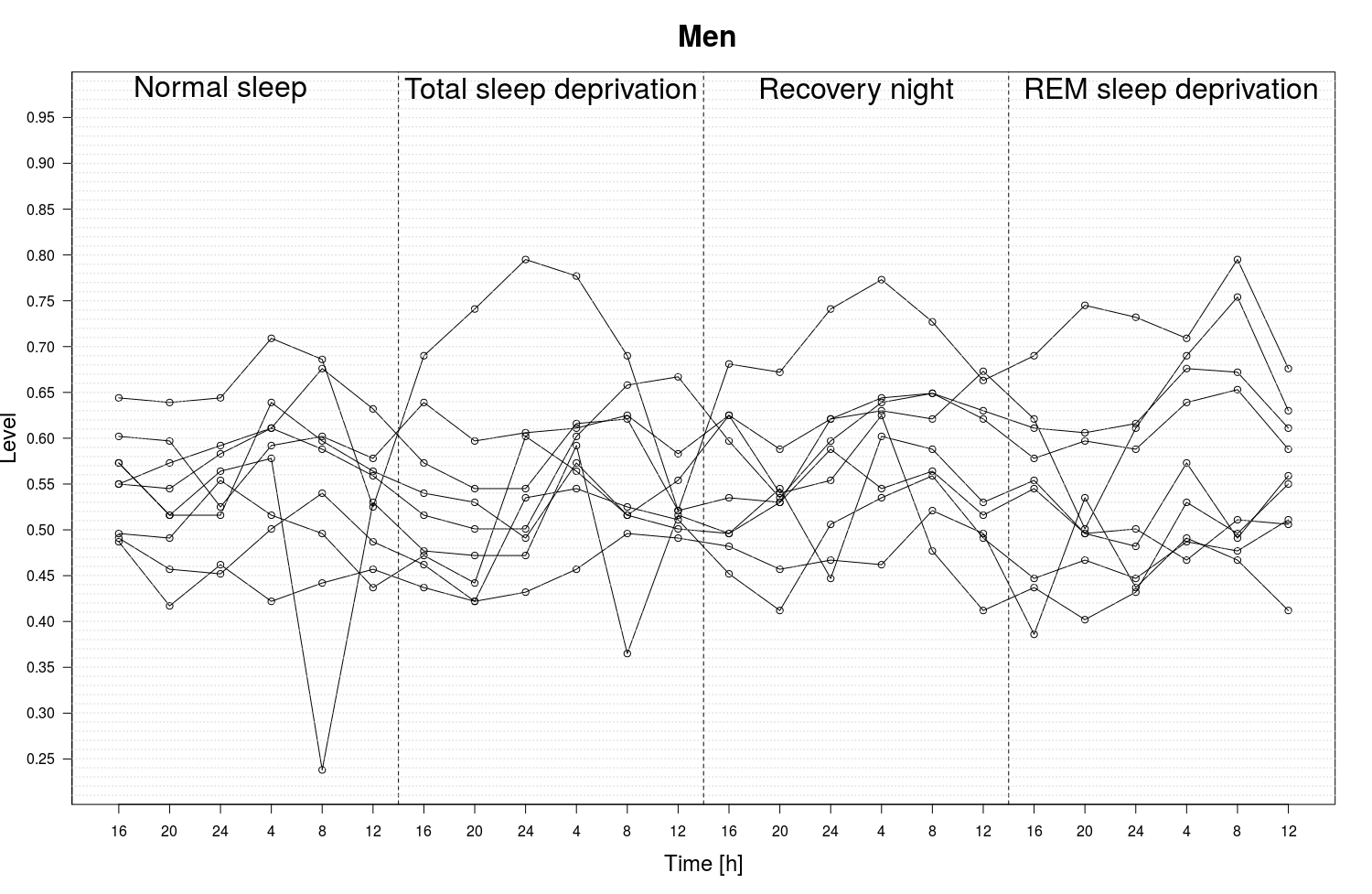}
\caption{\small Prostaglandin-D-synthase (ß-trace) of 10 young men during 4 days under different sleep conditions.}
\label{fig:Schlafmaenner}
\end{figure}
Moreover, investigations regarding (3) a general effect of time and (4) interactions between the different factors are of equal interest. Utilizing the notation from \Cref{ Statistical Model and the Hypotheses}, the corresponding null hypotheses can be formalized via adequate contrast matrices. 
In particular, we are interested in testing the null hypotheses
{
\bit
\item[(a)] No gender effect: $H_0^a: \left(\vP_2\otimes \frac{1}{24}\vJ_{24}\right)\vmu=\vnull,$
\item[(b)] No time effect: $H_0^b: \left(\frac 1 {2}\vJ_2\otimes \vP_{24}\right)\vmu=\vnull$,
\item[(c)] No interaction effect between time and group: 
$H_0^{ab}: \left(\vP_2\otimes \vP_{24}\right)\vmu=\vnull$,
\item[(d)] No time effect for intervention $\ell$,  $\ell\in \{1,\dots,4\}$:\\ $H_0^{t \ell}:\left(\vP_{2}\otimes \left(\left(\ve_l\cdot \ve_l^\top\right) \otimes \vP_6\right)\right)\vmu=\vnull$,
\item[(e)] No effect between interventions $\ell$ and $k$, $\ell, k\in\{1,\dots,4\}$:\\
$H_0^{\ell \times k}:\left(\vP_{2}\otimes \left(\left(\ve_\ell\cdot \ve_\ell^\top-\ve_\ell\cdot \ve_k^\top\right) \otimes \frac{1}{6}\vJ_6\right)\right)\vmu=\vnull$,
\eit where $e_\ell = (\delta_{\ell j})_j$ denotes the  Kronecker delta.} Applying the test $\varphi_N^\star$ based on the standardized quadratic form $W_N$ as test statistic and 
the proposed $K_{\hat{f}_P^\star}$-approximation with $B=50000\cdot N = 100,000$ subsamples we obtain the results summarized in Table~\ref{tab:Dataanalysis} . 

 \renewcommand{\arraystretch}{1.1}
\begin{table}[h]
\normalsize
 \caption{\small Analysis of the sleep lab trial from Figures~\ref{fig:Schlaffrauen}-\ref{fig:Schlafmaenner}: Shown are the values of the test statistic $W_N$ and the 
 estimator $\hat{f}_p^\star$ as well as the $p$-values of the test $\varphi_N^\star =\mathbf{1}\{W_N > K_{\hat{f}^\star_P; 1-\alpha}\}$ for different null hypotheses of interest.}
 \bcen
 \begin{tabular}{|c|c|c|c|c|}
\hline

 Hypothesis&$W_{N}^A$& $\hat{f}_p^\star$&p-value\\
 \hline
$H_0^a$& -0.45671& 1.19030&0.55832\\\hline
$H_0^b$&6.24114&7.07832&\textbf{0.00008}\\\hline
$H_0^{ab}$&0.74578&7.21217&0.20120\\ \hline
$ H_0^{t1}$&-0.795083&461.874&0.784463 \\\hline
 $ H_0^{t2}$&-0.591851&360.048&0.71764\\\hline
 $ H_0^{t3}$&-0.43381&223.24000&0.65845\\\hline
 $ H_0^{t4}$&-1.18382&426.083&0.88385\\\hline
 $ H_0^{{1\times 2}}$&2.37921 &155.89025 &\textbf{0.01285}\\\hline
 $ H_0^{{1\times 3}}$& 0.23757& 156.64141&   0.39240\\\hline
 $ H_0^{{1\times 4}}$&--0.49984& 143.57718 &  0.68099\\\hline
 $ H_0^{{2\times 3}}$&-0.72716& 91.83337 & 0.75968\\\hline
 $ H_0^{{2\times 4}}$&-0.56510 &79.78169 & 0.70183\\\hline
 $ H_0^{{3\times 4}}$&-0.66704& 130.56430 &  0.74046\\\hline
 \end{tabular} \label{tab:Dataanalysis}
 \ecen
 \end{table} 
 \renewcommand{\arraystretch}{1}
 
\normalsize
There it can be readily seen that most hypotheses cannot be rejected at level $\alpha=5\%$. In particular, there is no evidence for an overall gender effect, so that we have not performed post-hoc analyses on the interventions. Only a highly significant time effect, as well as a significant effect between the first two interventions (normal sleep and sleep deprivation), could be detected. However, applying a multiplicity adjustment (Bonferroni or Holm) only the time effect remained significant. 

 


\section{Conclusion \& Outlook}\label{Conclusion & Outlook}

In this paper we have investigated inference procedures for general split-plot models, allowing for unbalanced and/or heteroscedastic covariance settings as well as a factorial structure on the whole- and sub-plot factors. Inspired by the work of \cite{Paper1} for one group repeated measures designs the test statistics were based on standardized quadratic forms. However, different to their work novel symmetrized $U$-statistics were introduced to adequately handle the problem of additional nuisance parameters in the multiple sample case.

To jointly cover low and highdimensional models as well as situations with a small or large number of groups we conducted an 
in-depth study of their asymptotic behaviour under a unified asymptotic framework. In particular, the number of groups $a$ and dimensions $d$ may be fixed as in classical asymptotic settings, or even converge to infinity. Here we do neither postulate any assumptions on how $d$ and/or $a$ and the underlying sample sizes converge to infinity nor any sparsity conditions on the covariance structures since such assumptions are usually hard to check for a practical data set at hand. 
As a consequence it turned out that the test statistic posses a whole continuum of asymptotic limits that depend on the eigenvalues of the underlying covariances. 
We thus argued that an approximation by a fixed critical value is not adequate and proposed an approximation by a sequence of standardized $\chi^2$-distributions with estimated degrees of freedom. For computational efficiency we additionally provided a subsampling-type version of the degrees of freedom estimator. Our approach provides a reasonably good three moment approximation of the test statistic and is even asymptotically exact if the influence of the largest eigenvalue is negligible (leading to a standard normal limit) or decisive (leading to a standardized $\chi_1^2$ limit). 

Apart from these asymptotic considerations we evaluated the finite sample and dimension performance of our approximation technique. 
In particular, for varying combinations of sample sizes and dimensions, we compared its power and type I error control with test procedures based on fixed critical values. In all designs it showed a quite accurate error control over all low- ($d\leq 10$) to highdimensional situations (with up to $d=800$). In comparison, its performance was considerably better than that of the other two tests which partially disclosed a rather liberal or conservative behaviour. 


In future research we like to extend the current results to general highdimensional MANOVA designs, where we also like to relax the involved assumption of multivariate normality and/or even test simultaneously for mean and covariance effects as recently proposed in \cite{liu2017simultaneous}. These investigations, however, require completely different (e.g., martingale) techniques and estimators of the involved traces. Moreover, we also plan to conduct more detailed simulations (especially for larger group sizes $a$ and other covariance matrices) in a more applied paper.

\section*{Acknowledgement} The authors would like to thank Edgar Brunner for helpful discussions. This work was supported by the
German Research Foundation project DFG-PA 2409/4-1.\newpage

{\begin{center}
{\large Supplementary Material to \\[1ex]

{\bf 'Inference For High-Dimensional Split-Plot-Designs:\\ A Unified Approach for Small to Large Numbers of Factor Levels'}}\\[1ex]
Paavo Sattler$^1$ and Markus Pauly$^1$ \\
{\tiny ${}^1$University of Ulm, Institute of Statistics}
 \end{center}
}

{\bf Abstract.} In this supplement we present all theoretical derivations and computations that were omitted in the paper for lucidity.

\newpage

\appendix
\section{Appendix}

We start with some preliminary results and Lemmatas.

\subsection{Basics} 
In Section~2 of the main paper we claimed that the unique projection matrix $\vT$ to the hypothesis matrix $\vH=\vH_S\otimes \vH_W$ 
that equivalently describes the null is given by the product of two projection matrices $\vT_S\otimes \vT_W$. We start with the proof of this claim:

\begin{LeA}\label{ginverse}
%
%
%
%
%


Let be $\vH=\vH_W\otimes\vH_S$ with $\vH\in\R^{ad\times ad}, \vH_W\in\R^{a\times a}, \vH_S\in\R^{d\times d}$.
For each  hypothesis $\vH\vmu=\textbf{0}$ with such a matrix $\vH$  exist projectors $\vT\in\R^{ad\times ad}, \vT_W\in\R^{a\times a}, \vT_S\in\R^{d\times d}$ which can be used to formulate the same null hypothesis $\vT\vmu=\textbf{0}$ with $\vT=\vT_W\otimes\vT_S$.

\end{LeA}
\begin{proof}
{It is known that the projector $\vT=\vH^\top[\vH\vH^\top]^-\vH$ fulfills $\vT\vmu =\textbf{0}\Leftrightarrow \vH\vmu=\textbf{0}$. For this reason and utilizing well known rules ( see for example \cite{Rao} ) for generalized inverses we obtain
\[\begin{array}{ll}
 \vT&=\vH^\top[\vH\vH^\top]^-\vH\\
 &= (\vH_W\otimes \vH_S)^\top [ (\vH_W\otimes \vH_S) (\vH_W\otimes \vH_S)^\top]^- (\vH_W\otimes \vH_S)
 \\[1.1ex]
  &= (\vH_W^\top\otimes \vH_S^\top) [ (\vH_W\otimes \vH_S) (\vH_W^\top\otimes \vH_S^\top)]^- (\vH_W\otimes \vH_S)
  \\[1.1ex]
   &= (\vH_W^\top\otimes \vH_S^\top) [ (\vH_W \vH_W^\top)\otimes (\vH_S\vH_S^\top)]^- (\vH_W\otimes \vH_S)
   \\[1.1ex]
      &= (\vH_W^\top\otimes \vH_S^\top) ([ \vH_W \vH_W^\top)]^- \otimes [\vH_S\vH_S^\top]^-) (\vH_W\otimes \vH_S)
      \\[1.1ex]
      &= (\vH_W^\top\otimes \vH_S^\top) ([ \vH_W \vH_W^\top)]^-\vH_W \otimes [\vH_S\vH_S^\top]^-\vH_S) 
      \\  [1.1ex]    
      &=\vH_W^\top[ \vH_W \vH_W^\top]^-\vH_W \otimes\vH_S^\top [\vH_S\vH_S^\top]^-\vH_S 
      \\    
      &=\vT_W\otimes \vT_S.
\end{array}\]\\
Thus, $\vT_W:=\vH_W^\top[ \vH_W \vH_W^\top]^-\vH_W$ and $\vT_S:=\vH_S^\top [\vH_S\vH_S^\top]^-\vH_S$ are projectors, i.e. idempotent and symmetric. }
\end{proof}\vspace{\baselineskip}

For proofing our main results we have to compare various traces of powers of combinations underlying covariance matrices. To this end, we will particularly apply the following inequalities:

\begin{LeA}\label{Spur1}For positive real numbers a,b and a symmetric matrix $\vA\in \R^{d\times d }$ it holds
\[\tr^2\left(\vA^{a+b}\right)\leq \tr\left(\vA^{2a}\right)\tr\left(\vA^{2b}\right).
\]
For $\vA\in \R^{d\times d }$ symmetric with eigenvalues $\lambda_1,\dots,\lambda_d\geq 0$ it holds that
\[\tr\left(\vA^2\right)\leq \tr^2\left(\vA\right).\]

If $\vSigma_i\in \R^{d\times d }$ is positive definite and symmetric and $\vT\in \R^{d\times d }$ is idempotent and symmetric 
it holds for every $k\in \N$ that
\[\begin{array}{l}\tr\left(\left(\vT\vSigma_i\right)^{2k}\right)\leq \tr^2\left(\left(\vT\vSigma_i\right)^k\right).
\end{array}\]

\end{LeA}
\begin{proof}The first part is an application of the Cauchy–Bunyakovsky–Schwarz inequality, with the Frobenius inner product. Therefore
\[\begin{array}{ll}\tr^2\left(\vA^{a+b}\right)&= \tr^2\left(\vA^{a}\vA^b\right)= \tr^2\left(\vA^{a}{\vA^b}^\top\right)\\&\leq \left(\sqrt{\tr\left(\vA^{a}{\vA^{a}}^\top\right)}\cdot \sqrt{\tr\left(\vA^{b}{\vA^{b}}^\top\right)}\right)^2 ={\tr\left(\vA^{a}{\vA^{a}}\right)}\cdot {\tr\left(\vA^{b}{\vA^{b}}\right)}\\&=\tr\left(\vA^{2a}\right)\tr\left(\vA^{2b}\right).\end{array}\]

The second part just uses the  binomial theorem together with the condition $\lambda_t\geq 0$ for $t=1,\dots,d$:

\[\tr(\vA^2)=\sum\limits_{t=1}^d \lambda_t^2\leq \sum\limits_{t_1=1}^d \lambda_{t_1}^2+\sum\limits_{t_1=1}^d \sum\limits_{t_2=1 , t_2\neq t_1 }^d \lambda_{t_1}\lambda_{t_2}=\left(\sum\limits_{t=1}^d \lambda_t\right)^2=\tr^2(\vA).\]

Finally, the last inequality follows from the second one, if we show that all conditions are fulfilled. With idempotence of $\vT$ and invariance of the trace  under cyclic permutations, it follows for all $k\in \N$ that

\[
\tr\left(\left(\vT\vSigma_i\right)^{2k}\right)
=\tr\left(\vT^2\vSigma_i \cdot \dots \cdot \vT^2 \vSigma_i\right)=
\tr\left(\vT\vSigma_i \vT   \cdot \dots \cdot \vT \vSigma_i \vT\right)= \tr\left(\left(\vT\vSigma_i \vT\right)^{2k}\right). 
\]\\

Thus, it is sufficient to consider this term. Since $\vT\vSigma_i \vT$ is symmetric all powers are symmetric too and
it follows with $ k'=\lfloor k/2\rfloor$ that

\[\begin{array}
{lrl}\forall  \vx\in\R^d: \hspace{0.5cm}& \vx^\top \left(\vT\vSigma_i \vT\right)^{k} \vx&= \vx^\top \left(\vT\vSigma_i \vT\right)^{k'}\vT \vSigma_i^{k-2k'}\vT\left(\vT\vSigma_i \vT\right)^{k'} \vx
\\[1ex]
&&= \left[\vT\left(\vT\vSigma_i \vT\right)^{k'}\vx\right]^\top\vSigma_i^{k-2k'}\left[\vT\left(\vT\vSigma_i \vT\right)^{k'} \vx\right] \geq 0
\end{array}\]
since $\vSigma_i$ and $\vI_d$  are positive definite and ${k-2k'}\in \{0,1\}$. So both conditions of the second inequation are shown and
 \[\tr\left(\left(\vT\vSigma_i\right)^{2k}\right)= \tr\left(\left(\vT\vSigma_i \vT\right)^{2k}\right)=\tr\left(\left[\left(\vT\vSigma_i \vT\right)^{k}\right]^2\right)\leq \tr^2\left(\left(\vT\vSigma_i \vT\right)^{k}\right)=\tr^2\left(\left(\vT\vSigma_i \right)^{k}\right).\]
\end{proof}
Furthermore, an inequality for traces which contain $\vSigma_i$ and $\vSigma_r$ is needed.

\begin{LeA}\label{Spur2}
Let $\vSigma_i,\vSigma_r\in \R^{d\times d }$  be positive definite and symmetric matrices and suppose that $\vT\in \R^{d\times d }$ is idempotent and symmetric. Then it holds for $i\neq r$ that
\[\tr\left(\left(\vT\vSigma_i \vT\vSigma_r\right)^2\right)\leq  \tr^2\left(\vT\vSigma_i \vT\vSigma_r\right).\]

\end{LeA}
\begin{proof}
As shown before $\vT\vSigma_i\vT$ and $\vT\vSigma_r\vT$  are symmetric and positive semidefinite.
For this reason, a symmetric matrix  $\vW$  exists with $\vW\vW=\vT\vSigma_r\vT$.
 Due the fact that all matrices are symmetric it holds

\[(\vW\vT\vSigma_i\vT \vW)^\top= \vW^\top \vT^\top\vSigma_i^\top \vT^\top \vW^\top=\vW\vT\vSigma_i \vT \vW\]  and because  $\vT\vSigma_i\vT$ is positive semidefinite also
\[\forall \vx \in \R^d \hspace*{1cm} \vx^\top \vW\vT\vSigma_i\vT \vW \vx=(\vW \vx)^\top \vT\vSigma_i \vT(\vW\vx)=\vy^\top \vT\vSigma_i \vT \vy \geq 0.\]\\
This allows to use the inequalities from above for this matrix, and again utilizing the invariance of the trace
under cyclic permutations we obtain\\

$\begin{array}{ll}
\tr\left(\left(\vT\vSigma_i\vT\vSigma_r\right)^2\right)&=\tr\left(\vT\vSigma_i\vT\vT\vSigma_r\vT\Cdot \vT\vSigma_i\vT\vT\vSigma_r\vT\right)=\tr\left(\vT\vSigma_i\vT\vW\vW \vT\vSigma_i\vT\vW\vW\right)
\\
&=\tr\left(\vW\vT\vSigma_i\vT\vW\vW \vT\vSigma_i\vT\vW\right)=\tr\left(\left(\vW\vT\vSigma_i\vT\vW\right)^2\right)\\
&\leq \tr^2\left(\vW\vT\vSigma_i\vT\vW\right)
=\tr^2\left(\vT\vSigma_i\vT\vW\vW\right)=\tr^2\left(\vT\vSigma_i\vT\vT\vSigma_r\vT\right)\\&=\tr^2\left(\vT\vSigma_i\vT\vSigma_r\right).\end{array}$\\
\end{proof}

To standardize the quadratic form we also have to calculate its moments. Here, the following theorem helps:

\begin{SaA}\label{QF3}

Let $\vT\in \R^{d\times d}$ be a symmetric matrix and ${\vX}\sim \mathcal{N}_d\left(\vmu_X,\vSigma_X\right),$ where $\vSigma_X$ is positive definite. Then with $r\in \N$ it holds,  
\[\E\left(\left({\vX}^\top \vT{\vX}\right)^r\right)=\sum\limits_{r_1=0}^{r-1} \binom {r-1}  {r_1} g^{\left(r-1-r_1\right)}\sum\limits_{r_2=0}^{r_1-1}\binom {r_1-1}{r_2}g^{\left(r_1-1-r_2\right)}\dots\]
with  $g^{\left(k\right)}=2^kk!\left[\tr\left(\left(\vT\vSigma\right)^{k+1}\right)+\left(k+1\right)\vmu_{X}\left(\vT\vSigma\right)^k \vT\vmu_{X}\right]$ for $k\in \N$ and $g^{\left(0\right)}=\tr\left(\vT\vSigma_{X}\right)+{\vmu_{X}}^\top \vT\vmu_{X}$.
\end{SaA}
\begin{proof}
The proof can be found on page 53 in \cite{Buch1}.
\end{proof}

\begin{KoA}\label{QF4}
Let $\vT\in \R^{d\times d}$ be a symmetric matrix and $\vX\sim \mathcal{N}_d\left(\boldsymbol{0}_d,\vSigma_X\right)$ and ${\vY}\sim \mathcal{N}_d\left(\boldsymbol{0}_d,\vSigma_Y\right)$ independent,  where $\vSigma_X,\vSigma_Y\in \R^{d\times d }$ are positive definite. Then we have for all $n_i,n_r,N\in \N$ that\\

$\begin{array}{ll}
\E\left(\left({\vX}^\top \vT{\vX}\right)^1\right)=\tr\left(\vT\vSigma_X\right),
\\[1.3ex]
\E\left(\left({\vX}^\top \vT{\vX}\right)^2\right)=
 2\tr\left(\left(\vT\vSigma_X\right)^2\right)+\tr^2\left(\vT\vSigma_X\right)\stackrel{\ref{Spur1}}{=}\0\left(\tr^2\left(\vT\vSigma_X\right)\right),
\\[1.8ex]
\Var\left({\vX}^\top \vT{\vX}\right)=\0\left(\tr^2\left(\vT\vSigma_X\right)\right),
 \end{array}$\\
 $\begin{array}{l} 
 \E\left(\left({\vX}^\top \vT{\vY}\right)^1\right)=0,
 \\[1.0ex]
  \E\left(\left({\vX}^\top \vT{\vY}\right)^2\right)=\tr\left(\vT\vSigma_X\vT\vSigma_Y\right),
  \\[1.0ex]
 \E\left(\left({\vX}^\top \vT{\vY}\right)^3\right)=0,
 \\[1.0ex]
  \E\left(\left({\vX}^\top \vT{\vY}\right)^4\right)=6 \tr\left(\left(\vT\vSigma_X\vT\vSigma_Y\right)^2\right)+3\tr^2\left(\vT\vSigma_X\vT\vSigma_Y\right),
 \end{array}$\\
 
$\begin{array}{l}
\Var\left({\vX}^\top \vT{\vY}\right)=\tr\left(\vT\vSigma_X\vT\vSigma_Y\right),
\\[1.0ex]
\Var\left(\left({\vX}^\top \vT{\vY}\right)^2\right)=6 \tr\left(\left(\vT\vSigma_X\vT\vSigma_Y\right)^2\right)+2\tr^2\left(\vT\vSigma_X\vT\vSigma_Y\right),
\\[2ex]
\frac{4N}{n_i^2n_r^2}\Var\left(\left({\vX}^\top \vT{\vY}\right)^2\right)\stackrel{\ref{Spur2}}{=}\0\left( \tr^2\left(\left(\frac{N}{n_i}\vT\vSigma_X \cdot\frac{N}{n_r}\vT\vSigma_Y\right)^2\right)\right).
 \end{array}$\\

Moreover, for $\vSigma_X=\vSigma_Y$\\
\\$\begin{array}{l}
 \Var\left({\vX}^\top \vT{\vY}\right)=\tr\left(\vT\vSigma_X\vT\vSigma_X\right)=\0\left(\tr^2\left(\vT\vSigma_X\vT\vSigma_X\right)\right),
 \\[2ex]
\Var\left(\left({\vX}^\top \vT{\vY}\right)^2\right)\stackrel{\ref{Spur1}}{=}\0\left(\tr^2\left(\vT\vSigma_X\vT\vSigma_X\right)\right).
\end{array}$
\end{KoA}

\begin{proof}

Using the inequalities for traces and with the bilinear form written as
\[{\vX}^\top \vT\vY=\frac 1 2
\begin{pmatrix} {\vX} \\ \vY\end{pmatrix}^\top \begin{pmatrix}
0&\vT\\\vT&0\\
\end{pmatrix}\begin{pmatrix}\vX \\ \vY \end{pmatrix},\hspace*{1cm}\begin{pmatrix} {\vX} \\ \vY\end{pmatrix}\sim \mathcal{N}_{2d}\left(\begin{pmatrix} \vmu_X \\ \vmu_Y\end{pmatrix},\begin{pmatrix} \vSigma_X&\vSigma_{XY} \\ \vSigma_{XY}&\vSigma_Y\end{pmatrix}\right)\]
 all equations follows with the previous theorem.
\end{proof}
 \vspace{\baselineskip}

\begin{LeA}\label{Kons2}
Let $X_n\in\mathcal{L}^2$ be a real random variable with $\E(X_n)=\mu $, $b_{n,d}$ a sequence with ${\lim_{n,d\to \infty}b_{n,d}= 0}$, and $c_{a,d,n_{\min}}$ a sequence with  $\lim_{a,d,n_{\min}\to \infty}c_{n,d}= 0$ then it holds
%
\begin{itemize}
\item $\Var\left({X_n}\right)\leq b_{n,d}\  \ \Rightarrow \ {X}_n \text{ is an consistent  estimator for  }\mu,{\text{ if } n,d\to \infty, }$
\item$\Var\left({X_n}\right)\leq c_{a,d,n_{\min}}\  \Rightarrow {X}_n \text{ is an consistent estimator for  }\mu,     {\text{ if } a,d,n_{\min}\to \infty.}$
\end{itemize}
For $\mu\neq 0$ they are especially ratio-consistent.
\end{LeA}

\begin{proof}
For arbitrary $\epsilon>0$  the Tschebyscheff inequality  leads to
\[\PP\left(|X_n-\mu|\geq \epsilon\right)\leq \frac{\E\left(|X_n-\mu|^2\right)}{\epsilon^2}=\frac{\Var\left(X_n\right)}{\epsilon^2}\leq \frac{b_{n,d}}{\epsilon^2}.\]
Consider the limit for $n,d\to \infty$ justifies the consistency and using this for $X_n/\mu$ leads to ratio-consistency. The second part follows  identically.
\end{proof}\vspace{1 \baselineskip}
This  result is especially true if $b_{n,d} $ or $c_{a,d,n_{\min}}$  only depends on n resp. $n_{\min}$.
%

For completeness we state a straightforward application of the Cauchy–Bunyakovsky–Schwarz inequality:

\begin{LeA}\label{Var1}
For real random variables $X,Y\in \mathcal{L}^2$ it holds
\[\Cov\left(X,Y\right)\leq \sqrt{\Var\left(X\right)}\sqrt{\Var\left(Y\right)}\] and so for $X,Y$ identically distributed
\[\Cov\left(X,Y\right)\leq {\Var\left(X\right)}.\] 
\end{LeA}

The next result gives equivalent conditions for $\beta_1\to a\in \{0,1\}$:
\begin{LeA}\label{Bedingungen}
Let be $\lambda_\ell$ again the eigenvalues of $\vT\vV_N\vT$ sorted so that $\lambda_1$ is the biggest one. Then it follows
\[\lim\limits_{N,d\to \infty}\frac{\lambda_1}{\sqrt{\sum_{\ell=1}^{ad}  \lambda_\ell^2}}= 1 \hspace*{0.2cm}\Leftrightarrow\hspace*{0.2cm} \lim\limits_{N,d\to \infty}\frac{\tr^2\left(\left(\vT\vV_N\right)^3\right)}{\tr^3\left(\left(\vT\vV_N\right)^2\right)}= 1 \hspace*{0.2cm}\Leftrightarrow\hspace*{0.2cm} \lim\limits_{N,d\to \infty}\frac{\tr\left(\left(\vT\vV_N\right)^4\right)}{\tr^2\left(\left(\vT\vV_N\right)^2\right)}\to 1,\]
\[\lim\limits_{N,d\to \infty}\frac{\lambda_1}{\sqrt{\sum_{\ell=1}^{ad}  \lambda_\ell^2}}= 0 \hspace*{0.2cm}\Leftrightarrow\hspace*{0.2cm} \lim\limits_{N,d\to \infty}\frac{\tr^2\left(\left(\vT\vV_N\right)^3\right)}{\tr^3\left(\left(\vT\vV_N\right)^2\right)}= 0 \hspace*{0.2cm}\Leftrightarrow\hspace*{0.2cm} \lim\limits_{N,d\to \infty}\frac{\tr\left(\left(\vT\vV_N\right)^4\right)}{\tr^2\left(\left(\vT\vV_N\right)^2\right)}= 0.\]

Moreover we know $0\leq\frac{\tr^2\left(\left(\vT\vV_N\right)^3\right)}{\tr^3\left(\left(\vT\vV_N\right)^2\right)}=\tau_P\leq 1.$ This Lemma also holds if $\lim_{N,d\to \infty}$ is replaced by $\lim_{a,N}\to \infty$ or $\lim_{a,d,N\to \infty}$. 
\end{LeA}
\begin{proof}
This follows from Lemma 8.1 given in the supplement in \cite{Paper1}[page 21] since their result does not depend on the concrete matrix, i.e. can be directly applied for $\vV_N$. Moreover, the different asymptotic frameworks do not influence the proof since they are hidden within the above convergences. 
\end{proof}\vspace{\baselineskip}

\renewcommand{\arraystretch}{1}

To prove the properties of the subsampling-type  estimators some auxiliaries are needed. In particular, the following lemma allows us to decompose the variances and to use conditional terms for the calculation.

\begin{LeA}\label{bVarianz}
Let $X$ be a real random variable and denote by $\F$ a $\sigma$-field. Then it holds that 

\[\Var(X)=\E\left(\Var\left(X|\F\right)\right)+\Var\left(\E\left(X|\F\right)\right).\]
\end{LeA}
\begin{proof}With the rules for conditional expectations we calculate
\[\begin{array}{ll}\E\left(\Var\left(X|\F\right)\right)&=\E\left(\E\left(X^2|\F\right)\right)-\E\left(\left[\E\left(X|\F\right)\right]^2\right)=\E\left(X^2\right)-\E\left(\left[\E\left(X|\F\right)\right]^2\right),\\
\Var\left(\E\left(X|\F\right)\right)&=\E\left(\left[\E\left(X|\F\right)\right]^2\right)-\left[\E\left(\E\left(X|\F\right)\right)\right]^2=\E\left(\left[\E\left(X|\F\right)\right]^2\right)-\left[\E\left(X\right)\right]^2.\end{array}\]
The result follows by sum up this both parts.
\end{proof}\vspace{\baselineskip}

We will apply the result for certain amounts (i.e. numbers) of pairs below. 
There, for each $i=1,\dots,a$ and $b=1,\dots,B$ we independently draw random subsamples $\{\sigma_{1i}(b),\dots,\sigma_{mi}(b)\}$ of length $m$ from $\{1,\dots,n_i\}$ and store them in a joint random vector $\vsigma(b,m) = (\vsigma_{1}(b,m),\dots,\vsigma_{a}(b,m))= (\sigma_{11}(b),\dots,\sigma_{ma}(b))$. Besides we define $\N_k=\{1,\dots,k\}$.

 \begin{LeA}\label{Menge}

Let $M(B,\vsigma(b,m))$ be the amount of pairs $(k,\ell)\in \N_B\times \N_B$, which fulfill  $\vsigma(k,m)$ and $\vsigma(\ell,m)$ have totally different elements and analogue $M(B,\vsigma_{i}(b,m))$. As long as $m\leq n_i$ for all $i \in\N_a$, it holds

 \[\frac{\E\left(| \N_B\times \N_B \setminus M(B,\vsigma(b,m))|\right)}{B^2}=1-\left(1-\frac{1}{B}\right)\cdot \prod\limits_{i=1}^a\frac{ \binom{n_i-m} {m}}{\binom{n_i} {m}}\]
 and for $m\leq n_i$

 \[\frac{\E\left(| \N_B\times \N_B \setminus M(B,\vsigma_{i}(b,m))|\right)}{B^2}=1-\left(1-\frac{1}{B}\right)\cdot \frac{ \binom{n_i-m} {m}}{\binom{n_i} {m}}\]

where $|\cdot|$ denotes the number of elements.\\

Let $M(B,(\vsigma_{i}(b,m),\vsigma_{r}(b,m)))$ be the amount of pairs $(k,\ell)\in \N_B\times \N_B$ fulfilling $\vsigma_{i}(k,m)$ and $\vsigma_{i}(\ell,m)$ and moreover $\vsigma_{r}(k,m)$ and $\vsigma_{r}(\ell,m)$  have totally different elements. If $m\leq n_i$ it holds

 \[\frac{\E\left(| \N_B\times \N_B \setminus M(B,(\vsigma_{i}(b,m),\vsigma_{r}(b,m)))|\right)}{B^2}=1-\left(1-\frac{1}{B}\right)\cdot \frac{ \binom{n_i-m} {m}}{\binom{n_i} {m}}\cdot \frac{ \binom{n_r-m} {m}}{\binom{n_r} {m}}.\]

%
%
%
%

 \end{LeA}
\begin{proof} 
Because $M(B,\vsigma(b,m))$  never contains pairs of the kind (k,k) the maximal number of elements is $B^2-B$. The fact that two vectors $\va,\vb\in\R^n$ have no element in common, even at different components, is denoted as  $\va{\neq}! \vb$. \\
The number of totally different pairs can be seen as a binomial distribution with $B^2-B$ elements, and to calculate the necessary probability independence is used. With the fact that all combinations in this situation have the same probability it follows that 

\[\begin{array}{ll}
&\PP\left(\vsigma(k,m)\neq !\vsigma(\ell,m)\right)=
\PP\left(\bigcap\limits_{i=1}^a \left(\vsigma_{i}(k,m){\neq}! \hspace{0.1cm}\vsigma_{i}(\ell,m)\right)  \right)
\\[1.5ex]
=&\prod\limits_{i=1}^a\PP\left(\vsigma_{i}(k,m){\neq}! \hspace{0.1cm}\vsigma_{i}(\ell,m)\right) 
=\prod\limits_{i=1}^a\frac{\binom{n_i} {m}\Cdot \binom{n_i-m} {m}}{\binom{n_i} {m}^2}=\prod\limits_{i=1}^a\frac{ \binom{n_i-m} {m}}{\binom{n_i} {m}}.
\end{array}\]
If two times $m$ elements are picked from $\N_{n_i}$ there are $\binom{n_i} {m}^2$ possibilities, where in $\binom{n_i} {m}\Cdot \binom{n_i-m} {m}$ of them both $m$-tuples are totally different. This leads to the stated probability and with the mean of the binomial distribution we get

\[\E\left(|M(B,\vsigma(b,m))|)\right)=(B^2-B)\cdot  \prod\limits_{i=1}^a\frac{ \binom{n_i-m} {m}}{\binom{n_i} {m}}.\]\\
All in all we calculate
 \[\begin{array}{ll}&\frac{\E\left(| \N_B\times \N_B \setminus M(B,\vsigma(b,m))|\right)}{B^2}=\frac{| \N_B\times \N_B|- \E\left(|M(B,\vsigma(b,m))|\right)}{B^2}
 =1-\left(1-\frac{1}{B}\right)\cdot \prod\limits_{i=1}^a\frac{ \binom{n_i-m} {m}}{\binom{n_i} {m}}.\end{array} \]

For $M(B,(\vsigma_{i}(b,m),\vsigma_{r}(b,m)))$ and $M(B,\vsigma_{i}(b,m))$ less multiplications are needed, so the results follow. 
\end{proof}\vspace{\baselineskip}
  If $B(N)\to \infty$ (for example B could be chosen proportional to N) these terms converge to zero, disregarding the number of groups or of m.

\subsection{Proofs of Section~\ref{The Test Statistic and its Asymptotics}}

\begin{proof}[Proof of \Crp{Theorem3}]
The proof of this lemma is very similar to the one from \cite{Paper1}[Theorem 2.1]. Due to the fact that a finite sum of multivariate normally distributed random variables is again multivariate normally distributed, the representation theorem can be used to (distributionally equivalently) express the quadratic form as $W_N=\sum_{s=1}^{ad} \frac{\lambda_s}{\sqrt{\sum_{\ell=1}^{ad} \lambda_\ell^2}}\left(\frac{C_s-1}{\sqrt{2}}\right)$.\\

The only differences to \cite{Paper1}[Theorem 2.1] are that in the case of more groups the eigenvalues do not only depend on $d$ but also on the $n_i$ and $a$ and that there are more terms to sum. The first point has only an influence on the limit of the $\beta_s$. The higher number of summands does not matter because we observe the asymptotic under the asymptotic frameworks \eqref{eq: as frame 2}-\eqref{eq: as frame 3}, for which at least $a$ or $d$ converge to infinity.  The proofs from \cite{Paper1}[Theorem 2.1] only need the representation from above, a number of summations which goes to infinity and the conditions on the limits of the $\beta_s$. Since these are fulfilled the proof can be conducted in the same way.
\end{proof}

\begin{proof}[Proof of \Crp{Schae1}]
Remember that with ${\vY}_{i,\ell,k}:=\vX_{i,\ell}-\vX_{i,k}$ and $i\neq r\in\N_a$, $a>1$ trace estimators were defined by
\\\\ 
$\begin{array}{l}
A_{i,1}\hspace{0.2cm}=\frac 1 {2\cdot \binom {n_i} {2}} \sum\limits_{\begin{footnotesize}\substack{\ell_1,\ell_2=1\\ \ell_1>\ell_2}\end{footnotesize}}^{n_i} \left({\vX}_{i,\ell_1}-{\vX}_{i,\ell_2}\right)^\top \vT_S\left({\vX}_{i,\ell_1}-{\vX}_{i,\ell_2}\right),
\end{array}\\\\
\begin{array}{l}
A_{i,r,2}=\frac  1 {4\cdot\binom {n_i} {2}\binom {n_r} {2}} {\sum\limits_{\begin{footnotesize}\substack{\ell_1,\ell_2=1\\ \ell_1>\ell_2}\end{footnotesize}}^{n_i}\sum\limits_{\begin{footnotesize}\substack{k_1,k_2=1\\ k_1>k_2}\end{footnotesize}}^{n_r}\left[\left({\vX}_{i,\ell_1}-{\vX}_{i,\ell_2}\right)^\top \vT_S\left({\vX}_{r,k_1}-{\vX}_{r,k_2}\right)\right]^2},
\end{array}\\\\
\begin{array}{l}
A_{i,3}\hspace{0.2cm}=\frac 1 {4\cdot6\binom {n_i} {4}} \sum\limits_{\begin{footnotesize}\substack{\ell_1,\ell_2=1\\ \ell_1>\ell_2}\end{footnotesize}}^{n_i}
\sum\limits_{\begin{footnotesize}\substack{k_2=1\\k_2\neq \ell_1\neq \ell_2 }\end{footnotesize}}^{n_i}\sum\limits_{\begin{footnotesize}\substack{k_1=1\\ \ell_2\neq \ell_1\neq k_1>k_2}\end{footnotesize}}^{n_i}
\left[\left({\vX}_{i,\ell_1}-{\vX}_{i,\ell_2}\right)^\top \vT_S\left({\vX}_{i,k_1}-{\vX}_{i,k_2}\right)\right]^2,
\end{array}\\\\
\begin{array}{l}
A_4\hspace{0.4cm}=\sum_{i=1}^a \left(\frac{N}{n_i }\right)^2 {(\vT_W)_{ii}}^2 A_{i,3}+2\sum_{i=1}^a\sum_{r=1 , r < i}^a \frac{N^2}{n_in_r } {(\vT_W)_{ir}}^2 A_{i,r,2}.
\end{array}$\vspace{2\baselineskip}

For $\ell\neq k$ we know ${\vY}_{i,\ell,k}\sim\mathcal{N}\left(\boldsymbol{0}_d,2  \vSigma_i\right)$ and for totally different indices the ${\vY}_{i,\ell,k}$ are  statistically independent. So the previous lemmata can be used to calculate the moments.
The unbiasedness can be shown by calculating the expectation values for each estimator
{\[\begin{array}{ll}\E\left(A_{i,1}\right)&=\frac 1 {2 \cdot \binom {n_i} {2}}\sum\limits_{\begin{footnotesize}\substack{\ell_1,\ell_2=1\\ \ell_1>\ell_2}\end{footnotesize}}^{n_i} \E\left[{{\vY}_{i,\ell_1,\ell_2}}^\top \vT_S{\vY}_{i,\ell_1,\ell_2}\right]\stackrel{\ref{QF4}}{=}
\tr\left(\vT_S\vSigma_i\right).\end{array}\]}
The following argument will be used several times in this work with small differences, so incidentally it will be more detailed.\\

 {To check the variance  we recognize first that  $\Cov\left[{{\vY}_{i,\ell_1,\ell_2}}^\top \vT_S{\vY}_{i,\ell_1,\ell_2}\large{\textbf{;}}\normalsize
{{\vY}_{i,\ell_1',\ell_2'}}^\top \vT_S{\vY}_{i,\ell_1',\ell_2'}\right]$ is 0 if all indices are totally different, so just $\binom {n_i}{2}\left(\binom{n_i} 2-\binom{n_i-2} 2\right) $ combinations  remain.}
Instead of calculating the covariances of the remaining quadratic forms it is easier to use lemmata from above.
By using the fact that all quadratic forms are identically distributed, we can calculate the  variances which are all the same so it is just the number of remaining combinations multiplied with the variances. This leads to:\\

$\begin{array}{ll}
\Var\left(A_{i,1}\right)&=\frac 1 {4 \cdot \binom {n_i} {2}^2}\sum\limits_{\begin{footnotesize}\substack{\ell_1,\ell_2=1\\ \ell_1>\ell_2}\end{footnotesize}}^{n_i}\sum\limits_{\begin{footnotesize}\substack{\ell_1',\ell_2'=1\\ \ell_1'>\ell_2'}\end{footnotesize}}^{n_i} \Cov\left[{{\vY}_{i,\ell_1,\ell_2}}^\top \vT_S{\vY}_{i,\ell_1,\ell_2}\hspace*{0.05cm}\large{\textbf{;}}\normalsize\hspace*{0.05cm}{{\vY}_{i,\ell_1',\ell_2'}}^\top \vT_S{\vY}_{i,\ell_1',\ell_2'}\right]
\\&
\hspace{-0.07cm}\stackrel{\ref{Var1}}{\leq} \frac{\binom{n_i}{2}-\binom{n_i-2}{2}}{4\binom{n_i}{2}} \Var\left[{{\vY}_{i,1,2}}^\top \vT_S{\vY}_{i,1,2}\right]+\frac{\binom{n_i-2} 2}{4\binom{n_i}{2}} \cdot 0
\\[1.5ex]
&\hspace{-0.075cm}\stackrel{\ref{QF4}}{=}\frac{\binom{n_i}{2}-\binom{n_i-2}{2}}{4\binom{n_i}{2}}{\0\left( \tr^2\left(2\vT_S\vSigma_i\right)\right)} 
\\[2ex]&=\0\left(n_i^{-1}\right)\cdot \0\left( \tr^2\left(\vT_S\vSigma_i\right)\right).
\end{array}$\\\\\\
With these values we know for $\vV_N=\bigoplus_{i=1}^a \frac{N}{n_i}\vSigma_i$ that\\

$\begin{array}{ll}
\E\left(\sum\limits_{i=1}^a\frac{N}{n_i} {(\vT_W)_{ii}} A_{i,1}\right)=\sum\limits_{i=1}^a\frac{N}{n_i} {(\vT_W)_{ii}} \E\left(A_{i,1}\right)=\tr\left(\vT\vV_N\right)
\end{array}$\\\\
and\\

$\begin{array}{ll}
\Var\left(\frac{\sum\limits_{i=1}^a \frac{N}{n_i}  {(\vT_W)_{ii}} A_{i,1}}{\E\left(\sum\limits_{i=1}^a \frac{N}{n_i} {(\vT_W)_{ii}}  A_{i,1}\right)}\right)&=\frac{\sum\limits_{i=1}^a \frac{N^2}{n_i^2} {(\vT_W)_{ii}}^2 \Var(A_{i,1})}{\tr^2\left(\vT\vV_N\right)}
\\&
\leq\frac{\sum\limits_{i=1}^a \0\left(n_i^{-1}\right)\cdot \0\left( \tr^2\left(\frac{N}{n_i} {(\vT_W)_{ii}} \vT_S\vSigma_i\right)\right)}{\tr^2\left(\vT\vV_N\right)} 
\end{array}$\\
$\begin{array}{ll}
\textcolor{white}{\Var\left(\frac{\sum\limits_{i=1}^a \frac{N}{n_i}  {(\vT_W)_{ii}} A_{i,1}}{\E\left(\sum\limits_{i=1}^a \frac{N}{n_i} {(\vT_W)_{ii}}  A_{i,1}\right)}\right)}
&
\leq\frac{\0\left(\frac{1}{n_{\min}}\right)\cdot \0\left( \sum\limits_{i=1}^a \tr^2\left(\frac{N}{n_i} {(\vT_W)_{ii}} \vT_S\vSigma_i\right)\right)}{\tr^2\left(\vT\vV_N\right)}\end{array}$ \\
$\begin{array}{ll}
\textcolor{white}{\Var\left(\frac{\sum\limits_{i=1}^a \frac{N}{n_i}  {(\vT_W)_{ii}} A_{i,1}}{\E\left(\sum\limits_{i=1}^a \frac{N}{n_i} {(\vT_W)_{ii}}  A_{i,1}\right)}\right)}&\leq\frac{\0\left(\frac{1}{n_{\min}}\right)\cdot \0\left( \tr^2\left(\sum\limits_{i=1}^a \frac{N}{n_i} {(\vT_W)_{ii}} \vT_S\vSigma_i\right)\right)}{\tr^2\left(\vT\vV_N\right)} 
=\0\left(\frac{1}{n_{\min}}\right).
\end{array}$\\
So the conditions for an unbiased and ratio-consistent estimator are fulfilled.\\\\
The same steps with a different number of remaining combinations leads to\\

$\begin{array}{ll}
\E\left(A_{i,3}\right)&={\frac 1 {4\cdot6\binom {n_i} {4}} \sum\limits_{\begin{footnotesize}\substack{\ell_1,\ell_2=1\\ \ell_1>\ell_2}\end{footnotesize}}^{n_i}
\sum\limits_{\begin{footnotesize}\substack{k_1,k_2=1\\ \ell_2\neq \ell_1\neq k_1>k_2\neq \ell_1\neq \ell_2 }\end{footnotesize}}^{n_i}\E\left(\left[{{\vY}_{i,\ell_1,\ell_2}}^\top \vT_S{\vY}_{i,k_1,k_2}\right]^2\right)}
\\
&\hspace{-0.075cm}\stackrel{\ref{QF4}}{=}
\frac  1 {4\cdot6\binom {n_i} {4}} \Cdot  {6\binom {n_i} {4}}\cdot\tr\left(4\Cdot \left(\vT_S\vSigma_i\right)^2 \right)=\tr\left(\left(\vT_S\vSigma_i\right)^2\right),
\end{array}$\\\\\\

$\begin{array}{ll}\Var\left({A_{i,3}}\right)
&=\sum\limits_{\begin{footnotesize}\substack{\ell_1,\ell_2=1\\ \ell_1>\ell_2}\end{footnotesize}}^{n_i}
\sum\limits_{\begin{footnotesize}\substack{k_1,k_2=1\\ \ell_2\neq \ell_1\neq k_1>k_2\neq \ell_1\neq \ell_2 }\end{footnotesize}}^{n_i}
\sum\limits_{\begin{footnotesize}\substack{\ell_1',\ell_2'=1\\ \ell_1'>\ell_2'}\end{footnotesize}}^{n_i}
\sum\limits_{\begin{footnotesize}\substack{k_1',k_2'=1\\ \ell_2'\neq \ell_1'\neq k_1'>k_2'\neq \ell_1'\neq \ell_2' }\end{footnotesize}}^{n_i}\frac    {\Cov\left(\left[{{\vY}_{i,\ell_1,\ell_2}}^\top \vT_S{\vY}_{i,k_1,k_2}\right]^2\hspace*{0.05cm}\sem\hspace*{0.05cm}\left[{{\vY}_{i,\ell_1',\ell_2'}}^\top \vT_S{\vY}_{i,k_1',k_2'}\right]^2\right)}{4^2\cdot 6^2\cdot\binom {n_i} {4}^2 }
\\[4ex]
&\hspace{-0.07cm}\stackrel{\ref{Var1}}{\leq}\frac  {6 \binom {n_i} 4-6\binom{n_i-4}{4}}  {4^2\cdot 6\cdot\binom {n_i} {4} }\Var\left(\left[{{\vY}_{i,1,2}}^\top \vT_S{\vY}_{i,3,4}\right]^2\right)
\\[2ex]&
\hspace{-0.075cm}\stackrel{\ref{QF4}}{=}\frac  {\binom {n_i} 4-\binom{n_i-4}{4}} {16\binom {n_i} {4} }  \0\left(\tr^2\left(\left(\vT_S\vSigma_i\right)^2\right)\right) \\[2ex]&=\0\left(n_i^{-1}\right) \cdot \0\left(\tr^2\left(\left(\vT_S\vSigma_i\right)^2\right)\right),
\end{array}$\\\\\\

$\begin{array}{ll}
\E\left(A_{i,r,2}\right)&=\frac  1 {4\cdot\binom {n_i} {2}\binom {n_r} {2}} {\sum\limits_{\begin{footnotesize}\substack{\ell_1,\ell_2=1\\ \ell_1>\ell_2}\end{footnotesize}}^{n_i}\sum\limits_{\begin{footnotesize}\substack{k_1,k_2=1\\ k_1>k_2}\end{footnotesize}}^{n_r}\E\left(\left[{{\vY}_{i,\ell_1,\ell_2}}^\top \vT_S
{\vY}_{r,k_1,k_2}\right]^2\right)}
\\
&\hspace{-0.075cm}\stackrel{\ref{QF4}}{=}
\frac  1 {4\cdot\binom {n_i} {2}\binom {n_r} {2}}  \Cdot  \binom {n_i} {2}\Cdot  \binom {n_r} {2}\cdot\tr\left(4\Cdot \vT_S\vSigma_i \vT_S\vSigma_r\right)=\tr\left(\vT_S\vSigma_i\vT_S\vSigma_r\right)
,\end{array}$\\\\\\

$\begin{array}{ll}
\Var\left(\frac {2 N^2}{n_i n_r} A_{i,r,2}\right)&=\frac {4 N^4}{n_i^2 n_r^2}   \sum\limits_{\begin{footnotesize}\substack{\ell_1,\ell_2=1\\ \ell_1>\ell_2}\end{footnotesize}}^{n_1}\sum\limits_{\begin{footnotesize}\substack{k_1,k_2=1\\ k_1>k_2}\end{footnotesize}}^{n_2}\sum\limits_{\begin{footnotesize}\substack{\ell_1',\ell_2'=1\\ \ell_1'>\ell_2'}\end{footnotesize}}^{n_i}\sum\limits_{\begin{footnotesize}\substack{k_1',k_2'=1\\ k_1'>k_2'}\end{footnotesize}}^{n_r}\frac{\Cov\left(\left[{{\vY}_{i,\ell_1,\ell_2}}^\top \vT_S{\vY}_{r,k_1,k_2}\right]^2\hspace*{0.05cm}\sem\hspace*{0.05cm}\left[{{\vY}_{i,\ell_1',\ell_2'}}^\top \vT_S\hspace{0.025cm} {\vY}_{r,k_1',k_2'}\right]^2\right)} {16\cdot\binom {n_i} {2}^2\binom {n_r} {2}^2 } 
\\[3ex]
&\hspace{-0.07cm}\stackrel{\ref{Var1}}{\leq}\frac {4 N^4}{n_i^2 n_r^2} \frac  {\binom{n_i}{2}\binom{n_r} {2}-\binom{n_i-2}{2}\binom{n_r-2}{2}} {16\cdot\binom {n_i} {2}\binom {n_r} {2} } \Var\left(\left[{{\vY}_{i,1,2}}^\top \vT_S\hspace{0.025cm} {\vY}_{r,1,2}\right]^2\right)
\end{array}$\\\\
$\begin{array}{ll}
\textcolor{white}{\Var\left(\frac {2 N^2}{n_i n_r} A_{i,r,2}\right)}
&\hspace{-0.075cm}\stackrel{\ref{QF4}}{\leq} 
\frac  {\binom{n_i}{2}\binom{n_r} {2}-\binom{n_i-2}{2}\binom{n_r-2}{2}} {\binom {n_i} {2}\binom {n_r} {2} }\cdot \0\left(  \tr^2\left(\frac{N}{n_i}\vT_S\vSigma_i \frac{N}{n_r}\vT_S\vSigma_r\right)\right)
\\[1.5ex]&\leq 
\0\left(\frac{1}{n_{\min}}\right)\cdot\0\left(  \tr^2\left(\frac{N}{n_i}\vT_S\vSigma_i \frac{N}{n_r}\vT_S\vSigma_r\right)\right)
.\end{array}$\\\\

Finally, the conditions for $A_4$ have to be checked. 
With the expectation values from above we calculate\\\\
$\begin{array}{ll}
\E\left(A_4\right)&=\sum\limits_{i=1}^a\frac{N^2}{n_i^2}{(\vT_W)_{ii}}^2\E(A_{i,3})+2\sum\limits_{i=1}^a\sum\limits_{r=1,r<i}^a \frac{N^2}{n_in_r}{(\vT_W)_{ir}}^2\E\left(A_{i,r,2}\right)
\\[1.5ex]&=\sum\limits_{i=1}^a\frac{N^2}{n_i^2}{(\vT_W)_{ii}}^2\tr\left(\left(\vT_S\vSigma_i\right)^2\right)+2\sum\limits_{i=1}^a\sum\limits_{r=1,r<i}^a \frac{N^2}{n_in_r}{(\vT_W)_{ir}}^2\tr\left(\vT_S\vSigma_i\vT_S\vSigma_r\right)
=\tr\left(\left(\vT\vV_N \right)^2\right).
\end{array}$\\\\\\
To calculate the variances the following additional inequalities are needed:\\

$\begin{array}{ll}\frac{\Var\left(\sum\limits_{i=1}^a \left(\frac{N}{n_i}\right)^2{(\vT_W)_{ii}}^2 A_{i,3}\right)}{\tr^2\left(\left(\vT\vV_N\right)^2\right)}&=\frac{\sum\limits_{i=1}^a\Var\left(\left(\frac{N}{n_i}\right)^2{(\vT_W)_{ii}}^2  A_{i,3}\right)}{\tr^2\left(\left(\vT\vV_N\right)^2\right)}
\\
&\leq \sum\limits_{i=1}^a \0\left(n_i^{-1}\right) \cdot \frac{\0\left({(\vT_W)_{ii}}^4\tr^2\left(\left(\vT_S\frac{N}{n_i}\vSigma_i\right)^2\right)\right)}{\tr^2\left(\left(\vT\vV_N\right)^2\right)}
\\
&\leq\0\left(\frac{1}{n_{\min}}\right) \frac{\0\left(\tr^2\left(\sum\limits_{i=1}^a{(\vT_W)_{ii}}^2\left(\vT_S\frac{N}{n_i}\vSigma_i\right)^2\right)\right)}{\tr^2\left(\left(\vT\vV_N\right)^2\right)}
=\0\left(\frac{1}{n_{\min}} \right)\end{array}$\\\\
and\\

$\begin{array}{ll}&\frac{\Var\left(2\sum\limits_{r < i\in\N_a} \frac{N^2}{n_in_r }{(\vT_W)_{ir}}^2A_{i,r,2}\right)}{\tr^2\left(\left(\vT\vV_N\right)^2\right)}
\\[1.5ex]
\stackrel{\ref{Var1}}{\leq}& 4\sum\limits_{i<r\in \N_a}\sum\limits_{h<g\in \N_a}\frac{\sqrt{\Var\left(\frac {N^2}{n_{i}n_{r}}{(\vT_W)_{ir}}A_{i,r,2}\right)}\sqrt{\Var\left(\frac {N^2}{n_{h}n_{g}}{(\vT_W)_{gh}}A_{h,g,2}\right)}}{\tr^2\left(\left(\vT\vV_N\right)^2\right)}
\\[2ex]
\end{array}$\\
$\begin{array}{ll}
\leq&\left(\sum\limits_{i\neq r\in \N_a}\frac{\sqrt{\0\left(\frac 1 {n_{\min}}\right)}{(\vT_W)_{ir}}^2\tr \left(\vT_S\frac N {n_i}\vSigma_i\vT_S\frac{N}{n_r}\vSigma_r\right)}{\tr\left(\left(\vT\vV_N\right)^2\right)}\right)^2
\\[2ex]
\leq&\0\left(\frac 1 {n_{\min}}\right)\left(\frac{\0\left(\sum\limits_{i\neq r\in \N_a}{(\vT_W)_{ir}}^2\tr \left(\vT_S\frac N {n_i}\vSigma_i \vT_S\frac{N}{n_r}\vSigma_r\right)\right)}{\sum\limits_{i,r\in\N_a}{(\vT_W)_{ir}}^2\tr\left(\vT_S\frac{N}{n_i}\vSigma_i\frac{N}{n_r}\vT_S\vSigma_r\right)}\right)^2\leq \0\left(\frac 1 {n_{\min}}\right)
.\end{array}$\\\\\\
Together this leads to\\

$\begin{array}{ll}\Var\left(\frac{A_4}{\tr\left(\left(\vT\vV_N\right)^2\right)}\right)&

\stackrel{\ref{Var1}}{\leq}\left[\sqrt{\frac{\Var\left(2\sum\limits_{r < i\in \N_a} \frac{N^2}{n_in_r }{(\vT_W)_{ir}}^2A_{i,r,2}\right)}{\tr^2\left(\left(\vT\vV_N\right)^2\right)}}+\sqrt{\frac{\Var\left(\sum\limits_{i=1}^a\frac{N}{n_i} {(\vT_W)_{ii}}^2 A_{i,3}\right)}{\tr^2\left(\left(\vT\vV_N\right)^2\right)}}\right]^2\\
&\left[\sqrt{\0\left(\frac{1}{n_{\min}}\right)}+ \sqrt{\0\left(\frac{1}{n_{\min}}\right)}\right]^2=\0\left(\frac{1}{n_{\min}}\right)
\end{array}$\\\\
 
 and therefore $ A_4$ is an unbiased and ratio-consistent estimator of $\tr\left(\left(\vT\vV_N\right)^2\right)$.\\
 Moreover, we want to stress that the zero sequences used as upper border for $\widehat{E_{H_0}}(Q_N)$ and $A_4$ do not depend on the number of groups or dimensions, so this estimators can be also used for increasing number of groups.
 \\
 
 With the expectation values and variances from the beginning it follows directly that $A_{i,1},A_{i,r,2},A_{i,3}$ are unbiased, ratio-consistent  estimators of $\tr(\vT_S\vSigma_i),\tr\left(\vT_S\vSigma_i\vT_S\vSigma_r\right)$ and $\tr\left(\left(\vT_S\vSigma_i\right)^2\right)$.\\\

 It is worth to note that all of this estimators also consistent estimators which are even dimension-stable in the sense of \cite{Brunner}.
\end{proof}\vspace{\baselineskip}
For  $A_{i,r,2}$ there exists a alternative form which can be implemented substantially more efficient and was considered in \cite{Brunner2012}. It is based on matrices of the form $\widehat \vM_{i,r}=\vP_{n_i}\left(\vT_S{\vX}_{i,1},\dots,\vT_S{\vX}_{i,n_i}\right)^\top\Cdot \left(\vT_S{\vX}_{r,1},\dots,\vT_S{\vX}_{r,n_r}\right)\vP_{n_r}^\top$. Recalling that $\textbf{1}_{n}$ is the vector of ones and $\#$ denotes the Hadamard-Schur-Product, it can be seen that
 \[A_{i,r,2}=\frac{{\textbf{1}_{n_i}}^\top \left(\widehat{\vM_{i,r}}\#\widehat{\vM_{i,r}}\right)\textbf{1}_{n_i}}{(n_i-1)(n_r-1)}.\]\\

For $A_{i,3}$  there also exists an alternative formula, which expands much longer, but is more efficient:
\[\begin{array}{ll}
A_{i,3}=&\sum\limits_{\begin{footnotesize}\substack{\ell_1,\ell_2=1\\ \ell_1\neq \ell_2}\end{footnotesize}}^{n_i}\frac{\left[{{\vX}_{i,\ell_1}}^\top \vT_S{{\vX}_{i,\ell_2}}\right]^2}{n_i(n_i-1)}-(2n_i+5)\sum\limits_{\begin{footnotesize}\substack{\ell_1,\ell_2,\ell_3=1\\ \ell_1\neq \ell_2\neq \ell_3}\end{footnotesize}}^{n_i}\frac{\left[{\vX}_{i,\ell_1}^\top \vT_S{\vX}_{i,\ell_2} {\vX}_{i,\ell_1}^\top \vT_S{\vX}_{i,\ell_3}\right]}{n_i(n_i-1)(n_i-2)(n_i-3)}
\\[2ex]
&-\sum\limits_{\begin{footnotesize}\substack{\ell_1,\ell_2,\ell_3=1\\ \ell_3\neq \ell_1,\ell_2}\end{footnotesize}}^{n_i}\frac{\left[{\vX}_{i,\ell_1}^\top \vT_S{\vX}_{i,\ell_3} {\vX}_{i,\ell_2}^\top \vT_S({\vX}_{i,\ell_3}+{\vX}_{i,\ell_1})\right]}{n_i(n_i-1)(n_i-2)(n_i-3)}+\sum\limits_{\begin{footnotesize}\substack{\ell_1,\ell_2,\ell_3=1\\ \ell_1\neq \ell_2\neq \ell_3}\end{footnotesize}}^{n_i}\frac{\left[{\vX}_{i,\ell_1}^\top \vT_S{\vX}_{i,\ell_3}{\vX}_{i,\ell_2}^\top \vT_S{\vX}_{i,\ell_2}\right]}{n_i(n_i-1)(n_i-2)(n_i-3)}
\\[2ex]
&-
\sum\limits_{\begin{footnotesize}\substack{\ell_1,\ell_2,\ell_3=1\\ \ell_1\neq \ell_2}\end{footnotesize}}^{n_i}\frac{\left[{\vX}_{i,\ell_1}^\top \vT_S{\vX}_{i,\ell_2}{\vX}_{i,\ell_2}^\top \vT_S{\vX}_{i,\ell_3}\right]}{n_i(n_i-1)(n_i-2)(n_i-3)}-\frac{n_i^2 \left[\overline{{\vX}}_{i}^\top \vT_S\overline{{\vX}}_{i}\right]\left(n_i^2\overline{{\vX}}_{i}^\top \vT_S\overline{{\vX}}_{i}-\sum_{\ell_1=1}^{n_i}\left[{\vX}_{i,\ell_1}^\top \vT_s{\vX}_{i,\ell_1}\right]\right)}{n_i(n_i-1)(n_i-2)(n_i-3)}.
\end{array}\]

To finally prove \Crp{Theorem4} we need another lemma.

\begin{LeA}\label{Schae2}
For the previously defined estimators it holds for $n_{\min}\to \infty $ that
{\[\frac{\sum_{i=1}^a\frac{N}{n_i}{(\vT_W)_{ii}}A_{i,1}-\sum_{i=1}^a\frac{N}{n_i}{(\vT_W)_{ii}}\tr\left(\vT_S\vSigma_i\right)}{\sqrt{2\tr\left(\left(\vT\vV_N\right)^2\right)}} \hspace*{0.3cm}\stackrel{P}{\longrightarrow}0\hspace*{1.0cm} \text{independent of d or a}.\]}

\end{LeA}
\begin{proof}
We know that \\\\
$\begin{array}{ll}
&\E\left(\frac{\sum_{i=1}^a\frac{N}{n_i}{(\vT_W)_{ii}}A_{i,1}-\sum_{i=1}^a\frac{N}{n_i}{(\vT_W)_{ii}}\tr\left(\vT_S\vSigma_i\right)}{\sqrt{2\tr\left(\left(\vT\vV_N\right)\right)}}\right)=\frac{\sum_{i=1}^a\frac{N}{n_i}{(\vT_W)_{ii}}\left(\E\left(A_{i,1}\right)-\tr\left(\vT_S\vSigma_i\right)\right)}{\sqrt{2\tr\left(\left(\vT\vV_N\right)^2\right)}}=0.
\end{array}$\\

Thus, \\
 $\begin{array}{ll}
 \Var\left(\frac{\sum_{i=1}^a\frac{N}{n_i}{(\vT_W)_{ii}}\left(A_{i,1}-\tr\left(\vT_S\vSigma_i\right)\right)}{\sqrt{2\tr\left(\left(\vT\vV_N\right)^2\right)}}\right)
&=\frac{\sum\limits_{i=1}^a\frac{N^2}{n_i^2}{(\vT_W)_{ii}}^2\Var\left(A_{i,1}\right)}{2\tr\left(\left(\vT\vV_N\right)^2\right)}
\\[3ex]
&\hspace*{-0.48cm}\stackrel{\text{Proof of \ref{Schae1}}}{\leq }\0\left(\frac{1}{n_{\min}}\right)\frac{\sum\limits_{i=1}^a\frac{N^2}{n_i^2}{(\vT_W)_{ii}}^2\tr\left( \left(2\vT_S\vSigma_i\right)^2\right)}{2\tr\left(\left(\vT\vV_N\right)^2\right)}
=\0\left(\frac{1}{n_{\min}}\right).
\end{array}$\\
In the last step we used the fact that all terms are non-negative and applied the binomial theorem in the last inequality.  It  is a zero sequence which only depends on $n_{\min}$, so again with \Crp{Kons2} the result is proved.
\end{proof}\vspace{\baselineskip}

\begin{proof}[Proof of \Crp{Theorem4}]
 From \Cref{Kons2} it follows for $n_{\min}\to \infty$ and independent of $a$ or $d$ that $A_4\left/{\tr\left(\left(\vT\vV_N\right)^2\right)}\right.\stackrel{P}{\longrightarrow} 1$ and therefore ${{\tr\left(\left(\vT\vV_N\right)^2\right)}\left/A_4\right.\stackrel{P}{\longrightarrow} 1}$. Moreover, it also follows that $\sqrt{{\tr\left(\left(\vT\vV_N\right)^2\right)}\left/A_4\right.}\stackrel{P}{\longrightarrow} 1$ and with \Cref{Schae2} we deduce $\frac{\sum_{i=1}^a\frac {N}{n_i}{(\vT_W)_{ii}}A_{i,1}-\tr\left(\vT\vV_N\right)}{\sqrt{2\tr\left(\left(\vT\vV_N\right)^2\right)}}\stackrel{P}{\longrightarrow} 0$.\\\\ 

Thus, we can finally calculate the standardized quadratic form as\\

$\begin{array}{ll}
W_N&=\frac{Q_N-\sum_{i=1}^a\frac {N}{n_i}{(\vT_W)_{ii}} A_{i,1}}{\sqrt{2A_4}}
\\[1.ex]
&=\left(\frac{Q_N-\tr\left(\vT\vV_N\right)}
{\sqrt{2\tr\left(\left(\vT\vV_N\right)^2\right)}}-\frac{\sum_{i=1}^a\frac {N}{n_i}{(\vT_W)_{ii}}A_{i,1}-\tr\left(\vT\vV_N\right)}
{\sqrt{2\tr\left(\left(\vT\vV_N\right)^2\right)}}\right)\cdot
\sqrt{\frac{\tr\left(\left(\vT\vV_N\right)^2\right)}{A_4}}
\\[2.5ex]
&=\left(\frac{Q_N-\tr\left(\vT\vV_N\right)}
{\sqrt{2\tr\left(\left(\vT\vV_N\right)^2\right)}}-o_p(1)\right)\cdot
(1+o_p(1))
\\[2.5 ex]
&=\widetilde W_N+\widetilde W_N\cdot o_p(1)-o_p(1)-o_p(1)\cdot o_p(1).
\end{array}$\\

The last two parts converge in probability to zero, so also in distribution and with Slutzky $\widetilde W_N\cdot o_p(1)$ converge in distribution to zero if one of the conditions of \Cref{Theorem3} is fulfilled. Thereby $W_N $ has asymptotical the same distribution as $\widetilde W_N$.\end{proof}\vspace{\baselineskip}

For large numbers of groups many estimators $A_{i,1}, A_{i,r,2}$ and $A_{i,3}$ and  have to be calculated which leads to long computation time. In this cases it is better to again use subsamling-type estimators which leads to $A_{i,1}^\star,{A^\star_{i,r,2}}, {A^\star_{i,3}}$ and therefore to $A^\star_4$.

%

\begin{LeA}\label{ZSchae1}
With the definitions from above let be\\
$\begin{array}{l}
A_{i,1}^\star(B)\hspace*{0.2cm}=\frac 1 {2\Cdot B} \sum\limits_{b=1}^B {{\vY}_{i,\sigma_{i1}(b),\sigma_{i2}(b)}}^\top \vT_S{{\vY}_{i,\sigma_{i1}(b),\sigma_{i2}(b)}},
\end{array}$\\\\
$\begin{array}{l}
A_{i,r,2}^\star(B)\hspace*{0.01cm}=\frac  1 {4\cdot B} {\sum\limits_{b=1}^B\left[{{\vY}_{i,\sigma_{i1}(b),\sigma_{i2}(b)}}^\top \vT_S{{\vY}_{r,\sigma_{r1}(b),\sigma_{r2}(b)}}\right]^2},
\end{array}$\\\\
$\begin{array}{l}
A_{i,3}^\star(B)\hspace*{0.2cm}=\frac  1 {4\cdot B} {\sum\limits_{b=1}^B\left[{{\vY}_{i,\sigma_{i1}(b),\sigma_{i2}(b)}}^\top \vT_S{{\vY}_{i,\sigma_{i3}(b),\sigma_{i4}(b)}}\right]^2},
\end{array}$\\\\
$\begin{array}{l}
A_4^\star(B)\hspace{0.4cm}=\sum\limits_{i=1}^a\frac{N^2}{n_i^2}{(\vT_W)_{ii}}^2\cdot A_{i,3}^\star(B)+ 2\sum_{i=1}^a\sum_{r=1 , r < i}^a \frac{N^2}{n_in_r}{(\vT_W)_{ir}}^2A_{i,r,2}^\star(B).
\end{array}$\\

 If $B(N)\to \infty$,  this estimators and $\sum_{i=1}^a A_{i,1}^\star$ have the same properties as
 $A_{i,1}, A_{i,r,2}, A_{i,3}, A_4$ and $\sum_{i=1}^a {A_{i,1}}$
 which were defined in  \Crp{Schae1} .

 \end{LeA}
\begin{proof}
 For $A_{i,1}^\star(B)$, this lemma will be proved in detail. For all other terms only the major steps are shown.\\

The unbiasedness is clear because the random variables $\sigma_{i1}(b),\sigma_{i2}(b)$  have no influence on the number of terms of the sum and also the terms are identically distributed. Hence, \\

$\begin{array}{ll}\E\left(A_{i,1}^\star(B)\right)&=\frac 1 {2\Cdot B} \sum\limits_{b=1}^B\E\left( {{\vY}_{i,\sigma_{i1}(b),\sigma_{i2}(b)}}^\top \vT_S{{\vY}_{i,\sigma_{i1}(b),\sigma_{i2}(b)}}\right)\\[2ex]&=
\frac 1 {2\Cdot B} \sum\limits_{b=1}^B\E\left( {{\vY}_{i,1,2}}^\top \vT_S{\vY}_{i,1,2}\right)\stackrel{\ref{QF4}}{=}\tr(\vT_S\vSigma_i).\end{array}$\\\\

 The second part is more complicated. Let $\mathcal{F}(\vsigma_{i}(B,m))$ be  the smallest $\sigma$-field which contains $\vsigma_{i}(b,m)\ \forall b\in B$, so obvious  $M(B,\vsigma_{i}(b))$ is $\mathcal{F}(\vsigma_{i}(B))$~-measurable. Identical for $\mathcal{F}(\vsigma_{i}(B,m),\vsigma_{r}(B,m))$ and $\mathcal{F}(\vsigma(B,m))$.
Similar to the previous part, the distribution of the bilinear form does not depend on the index combination. Together with the independence of the normally distributed vectors and $\sigma_{i1}(b),\sigma_{i2}(b)$ this leads to\\\\
$\begin{array}{l}\Var\left(\E\left(A_{i,1}^\star(B)\big|\F(\vsigma_{i}(B,2))\right)\right)=\Var\left(\tr\left(\vT_S\vSigma_i\right)\right)=0.\end{array}$\\\\
With \Crp{bVarianz} we thus obtain\\\\
$\begin{array}{ll} \Var\left(A_{i,1}^\star(B)\right)&=0+\E\left(\Var\left(A_{i,1}^\star(B)|\F(\vsigma_{i}(B,2))\right)\right).
\end{array}$\\

For the calculation of the conditional variance of the sum, it would be useful  finding an upper bound that is based on the variance instead of calculate the covariances. To achieve this, we calculate the number of index combinations which leads to a covariance that is zero. This amount is non-deterministic and we recognize it contains  the amount $M(B,\vsigma_{i}(b,2))$ which was considered before.\\
Again not the amount is important but the number of elements which are contained in $M(B,\vsigma_{i}(b,2))$ since the bilinear forms are identically distributed. Therefore the condition of the variance of the bilinear form disappears since the random indices have no influence on the variance. With the $\mathcal{F}(\vsigma_{i}(B,2))$-measurability of $M(B,\vsigma_{i}(b,2))$  it thus follows that\\\\

$\begin{array}{ll} \Var\left(A_{i,1}^\star(B)\right)&=0+\E\left(\Var\left(A_{i,1}^\star(B)|\F(\vsigma_{i}(B,2))\right)\right)
\\[1.0ex]
&\hspace{-0.075cm}\stackrel{\ref{Var1}}
{\leq}\frac{1}{4 B^2}\E\left(\sum\limits_{(j,\ell)\in \N_B\times \N_B \setminus M(B,(\vsigma_{i}(b,2)))}\Var\left({{\vY}_{i,\sigma_{i1}(j),\sigma_{i2}(j)}}^\top \vT_S{{\vY}_{i,\sigma_{i1}(j),\sigma_{i2}(j)}}\big|\F(\vsigma_{i}(B,2))\right)\right)
\end{array}$\\\\
$\begin{array}{ll}\textcolor{white} {\Var\left(A_{i,1}^\star(B)\right)}&=\frac{1}{4 B^2}\E\left(\sum\limits_{(j,\ell)\in \N_B\times \N_B \setminus M(B,(\vsigma_{i}(b,2)))}\Var\left({{\vY}_{i,1,2}}^\top \vT_S{{\vY}_{i,1,2}}\right)\right)
\\[2.5ex]

&\hspace{-0.075cm}\stackrel{\ref{QF4}}{=}\frac{\E\left(|\N_B\times \N_B \setminus M(B,(\vsigma_{i}(b,2)))|\right)}{B^2}\cdot  \frac{\0\left(\tr^2\left(\vT_S\vSigma_i\right)\right)}{4}
\\[.5 ex]
&\hspace{-0.075cm} \stackrel{\ref{Menge}}{=}\left(1-\left(1-\frac{1}{B}\right)\cdot \frac{ \binom{n_i-2} {2}}{\binom{n_i} {2}}\right) \cdot  \0\left(\tr^2\left(\vT_S\vSigma_i\right)\right)
.\end{array}$\\\\
%

The other values are calculated in a similar way.\\\\

$\begin{array}{ll}\E\left(A_{i,r,2}^\star(B)\right)&=\frac 1 {4\Cdot B} \sum\limits_{b=1}^B\E\left( \left[{{\vY}_{i,\sigma_{i1}(b),\sigma_{i2}(b)}}^\top \vT_S{{\vY}_{r,\sigma_{r1}(b),\sigma_{r2}(b)}}\right]^2\right)
\\[2ex]
&=\frac 1 {4\Cdot B} \sum\limits_{b=1}^B\E\left( \left[{{\vY}_{i,1,2}}^\top \vT_S{\vY}_{r,1,2}\right]^2\right)\stackrel{\ref{QF4}}{=}\tr(\vT_S\vSigma_i\vT_S\vSigma_r).\end{array}$\\

$\begin{array}{l}\Var\left(\E\left(A_{i,r,2}^\star(B)|\F(\vsigma_{i}(B,2),\vsigma_{r}(B,2))\right)\right)=\Var\left(\tr\left(\vT_S\vSigma_i\vT_S\vSigma_r\right)\right)=0.\end{array}$\\\\\\

$\begin{array}{ll} 
\Var\left(A_{i,r,2}^\star(B)\right)&

=0+{\E\left(\Var\left(A_{i,r,2}^\star(B)|\F(\vsigma_{i}(B),\vsigma_{r}(B,2)) \right)\right)}
\\[1ex]
&\leq\frac{\E\left(|\N_B\times \N_B \setminus M(B,\vsigma_{i}(b,2),\vsigma_{r}(b,2))|\right)}{B^2}\cdot \Var\left(\left[{{\vY}_{i,1,2}}^\top \vT_S{{\vY}_{r,1,2}}\right]^2\right)
\end{array}$\\
$\begin{array}{ll} 
\textcolor{white}{\Var\left(A_{i,r,2}^\star(B)\right)}&\hspace{-0.075cm}\stackrel{\ref{QF4}}{\leq}\frac{\E\left(|\N_B\times \N_B \setminus M(B,\vsigma_{i}(b,2),\vsigma_{r}(b,2))|\right)}{B^2}\cdot 
 \0\left({ \tr^2\left(\frac{N}{n_i}\vT_S\vSigma_i \frac{N}{n_r}\vT_S\vSigma_r\right)}\right)
\\[1.5ex]
&\hspace{-0.075cm}\stackrel{\ref{Menge}}{=} \left(1-\left(1-\frac{1}{B}\right)\cdot \frac{ \binom{n_i-2} {2}}{\binom{n_i} {2}}\cdot \frac{ \binom{n_r-2} {2}}{\binom{n_r} {2}}\right) \cdot   \0\left({ \tr^2\left(\frac{N}{n_i}\vT_S\vSigma_i \frac{N}{n_r}\vT_S\vSigma_r\right)}\right)
\\[2ex]
&\leq \left(1-\left(1-\frac{1}{B}\right)\cdot \frac{ \binom{n_{\min}-2} {2}^2}{\binom{n_{\min}} {2}^2}\right) \cdot   \0\left({ \tr^2\left(\frac{N}{n_i}\vT_S\vSigma_i \frac{N}{n_r}\vT_S\vSigma_r\right)}\right).
\end{array}$\\\\\\\\

$\begin{array}{ll}
\E\left(A_{i,3}^\star(B)\right)&=\frac 1 {4\Cdot B} \sum\limits_{b=1}^B\E\left( \left[{{\vY}_{i,\sigma_{i1}(b),\sigma_{i2}(b)}}^\top \vT_S{{\vY}_{i,\sigma_{i3}(b),\sigma_{i4}(b)}}\right]^2\right)
\\[2ex]
&=\frac 1 {4\Cdot B} \sum\limits_{b=1}^B\E\left( \left[{{\vY}_{i,1,2}}^\top \vT_S{\vY}_{i,1,2}\right]^2\right)\stackrel{\ref{QF4}}{=}\tr\left(\left(\vT_S\vSigma_i\right)^2\right).
\end{array}$\\

$\begin{array}[t]{l}\Var\left(\E\left(A_{i,3}^\star(B)|\F(\vsigma_{i}(B,4))\right)\right)=\Var\left(\tr\left(\left(\vT_S\vSigma_i\right)^2\right)\right)=0.\end{array}$\\\\

$\begin{array}[t]{ll}
 \Var\left(A_{i,3}^\star(B)\right)&=0+\E\left(\Var\left(A_{i,3}^\star(B)|\F(\vsigma_{i}(B,4))\right)\right)
\\[0.5ex]
&\hspace{-0.075cm}\stackrel{\ref{Var1}}{\leq}\frac{1}{16 B^2}\E\left(\sum\limits_{(j,\ell)\in \N_B\times \N_B \setminus M(B,\vsigma_{i}(b,4))} \Var\left(\left[{{\vY}_{i,\sigma_{i1}(j),\sigma_{i2}(j)}}^\top \vT_S{{\vY}_{i,\sigma_{i3}(j),\sigma_{i4}(j)}}\right]^2\Big|\F(\vsigma_{i}(B,4))\right)\right)\\[2ex]
&\stackrel{\ref{QF4}}{\leq}\frac{\E\left(|\N_B\times \N_B \setminus M(B,\vsigma_{i}(b,4))|\right)}{B^2}\cdot 
 \frac {\0\left(\tr^2\left(\left( \vT_S\vSigma_i\right)^2\right)\right)} {16 }
\\[1ex]
&\hspace{-0.075cm}\stackrel{\ref{Menge}}{=} \left(1-\left(1-\frac{1}{B}\right)\cdot \frac{ \binom{n_i-4} {4}} {\binom{n_i} {4}} \right)\cdot  \0\left(\tr^2\left(\left( \vT_S\vSigma_i\right)^2\right)\right).
\end{array}$\\\\\\

$\begin{array}{l}
\E\left(\sum\limits_{i=1}^a \frac{N}{n_i}{(\vT_W)_{ii}}A_{i,1}^\star \right)=\sum\limits_{i=1}^a \frac{N}{n_i}{(\vT_W)_{ii}}\E\left(A_{i,1}^\star \right)=\sum\limits_{i=1}^a
\frac{N}{n_i}{(\vT_W)_{ii}}\tr\left(\vT_S\vSigma_i\right).
\end{array}$\\\\

$
\begin{array}{ll}
\Var\left(\frac{\sum\limits_{i=1}^a \frac{N}{n_i}{(\vT_W)_{ii}}A_{i,1}^\star}{\tr\left(\vT\vV_N\right)}\right)&=\frac{\sum\limits_{i=1}^a\frac{N^2}{n_i^2}{(\vT_W)_{ii}}^2\Var\left(A_{i,1}^\star\right)}{\tr^2\left(\vT\vV_N\right)}
\\[-1.1ex]
&=
\frac{\sum\limits_{i=1}^a{(\vT_W)_{ii}}^2\left(1-\left(1-\frac{1}{B}\right)\cdot \frac{ \binom{n_i-2} {2}}{\binom{n_i} {2}}\right) \cdot  \0\left(\tr^2\left(\vT_S\frac{N}{n_i}\vSigma_i\right)\right)}{\tr^2\left(\vT\vV_N\right)}
\\[1.3ex]&\leq
\frac{\sum\limits_{i=1}^a{(\vT_W)_{ii}}^2\left(1-\left(1-\frac{1}{B}\right)\cdot \frac{ \binom{n_{\min}-2} {2}}{\binom{n_{\min}} {2}}\right) \cdot  \0\left(\tr^2\left(\vT_S\frac{N}{n_i}\vSigma_i\right)\right)}{\tr^2\left(\vT\vV_N\right)}
\\[1.ex]&\leq
\left(1-\left(1-\frac{1}{B}\right)\cdot \frac{ \binom{n_{\min}-2} {2}}{\binom{n_{\min}} {2}}\right) \cdot \frac{ \0\left(\tr^2\left(\sum\limits_{i=1}^a\frac{N}{n_i}{(\vT_W)_{ii}}\vT_S\vSigma_i\right)\right)}{\tr^2\left(\vT\vV_N\right)}
\\[1.ex]&=
\left(1-\left(1-\frac{1}{B}\right)\cdot \frac{ \binom{n_{\min}-2} {2}}{\binom{n_{\min}} {2}}\right) \cdot  \0\left(1\right).
\end{array}$\\\\\\
For $B(N)\to \infty$  the first factor is a zero sequence and therefore $\sum_{i=1}^a \frac{N}{n_i}{(\vT_W)_{ii}}A_{i,1}^\star$ a ratio-consistent, unbiased estimator of $\tr\left(\vT\vV_N\right).$\\
$\begin{array}{ll}
&\E\left(\sum\limits_{i=1}^a\frac{N^2}{n_i^2}{(\vT_W)_{ii}}^2A_{i,3}^\star+\sum\limits_{i\neq r \in\N_a}\frac{N^2}{n_in_r}{(\vT_W)_{ir}}^2A_{i,r,2}^\star\right)
\\[2.5ex]=&\sum\limits_{i=1}^a\frac{N^2}{n_i^2}{(\vT_W)_{ii}}^2\E\left(A_{i,3}^\star\right)+\sum\limits_{i\neq r \in\N_a}\frac{N^2}{n_in_r}{(\vT_W)_{ir}}^2\E\left(A_{i,r,2}^\star\right)
=\tr\left(\left(\vT\vV_N\right)^2\right).
\end{array}$\\\\\\

$\begin{array}{ll}
\Var\left(\frac{\sum\limits_{i=1}^a\frac{N^2}{n_i^2}{(\vT_W)_{ii}}^2 A_{i,3}^\star}{\tr\left(\left(\vT\vV_N\right)^2\right)}\right)&=
\frac{\sum\limits_{i=1}^a\Var\left(\frac{N^2}{n_i^2}{(\vT_W)_{ii}}^2 A_{i,3}^\star\right)}{\tr^2\left(\left(\vT\vV_N\right)^2\right)}
\\[-1.0ex]
&\leq\frac{ \sum\limits_{i=1}^a {(\vT_W)_{ii}}^4 \left(1-\left(1-\frac{1}{B}\right)\cdot \frac{ \binom{n_i-4} {4}} {\binom{n_i} {4}} \right)\cdot  \0\left(\tr^2\left(\left( \vT_S\frac{N}{n_i}\vSigma_i\right)^2\right)\right)}{\tr^2\left(\left(\vT\vV_N\right)^2\right)}
\end{array}$\\
$\begin{array}{ll}

\textcolor{white}{\Var\left(\frac{\sum\limits_{i=1}^a\frac{N^2}{n_i^2}{(\vT_W)_{ii}}^2 A_{i,3}^\star}{\tr\left(\left(\vT\vV_N\right)^2\right)}\right)}&\leq  \left(1-\left(1-\frac{1}{B}\right)\cdot \frac{ \binom{n_{\min}-4} {4}} {\binom{n_{\min}} {4}} \right)\cdot \frac{ \sum\limits_{i=1}^a{(\vT_W)_{ii}}^4 \0\left(\tr^2\left(\left( \vT_S\frac{N}{n_i}\vSigma_i\right)^2\right)\right)}{\tr^2\left(\left(\vT\vV_N\right)^2\right)}
\\[-.8ex]
&\leq  \left(1-\left(1-\frac{1}{B}\right)\cdot \frac{ \binom{n_{\min}-4} {4}} {\binom{n_{\min}} {4}} \right)\cdot \frac{  \0\left(\tr^2\left(\left( \sum\limits_{i=1}^a\frac{N}{n_i}{(\vT_W)_{ii}}\vT_S\vSigma_i\right)^2\right)\right)}{\tr^2\left(\left(\vT\vV_N\right)^2\right)}
\\[2ex]
&\leq  \left(1-\left(1-\frac{1}{B}\right)\cdot \frac{ \binom{n_{\min}-4} {4}} {\binom{n_{\min}} {4}} \right)\cdot   \0\left(1\right).
\end{array}$\\\\\\\\

$\begin{array}{ll}
&\Var\left(\frac{\sum\limits_{i\neq r\in \N_a} \frac{N^2}{n_in_r}{(\vT_W)_{ir}}^2A_{i,r,2}^\star}{\tr\left(\vT\vV_N\right)}\right)\\[2.5ex]
\leq&\left( \sum\limits_{i\neq r\in \N_a}\frac{\sqrt{\Var\left(\frac {N^2}{n_{i}n_{j}}{(\vT_W)_{ir}}^2A_{i,r,2}^\star\right)}}{\tr\left(\left(\vT\vV_N\right)^2\right)}\right)^2
\end{array}$\\
$\begin{array}{ll}
\leq&\left(1-\left(1-\frac{1}{B}\right)\cdot \frac{ \binom{n_{\min}-2} {2}^2}{\binom{n_{\min}} {2}^2}\right) \cdot\left(\frac{\sum\limits_{i\neq r\in \N_a}{(\vT_W)_{ir}}^2\sqrt{   \0\left({ \tr^2\left(\frac{N}{n_i}\vT_S\vSigma_i \frac{N}{n_r}\vT_S\vSigma_r\right)}\right)}}{\tr\left(\left(\vT\vV_N\right)^2\right)}\right)^2
\\[2.5ex]
\leq&\left(1-\left(1-\frac{1}{B}\right)\cdot \frac{ \binom{n_{\min}-2} {2}^2}{\binom{n_{\min}} {2}^2}\right) \cdot\left(\frac{\sum\limits_{i\neq r\in \N_a}\0\left({(\vT_W)_{ir}}^2\tr \left(\vT_S\frac N {n_i}\vSigma_i \vT_S\frac{N}{n_r}\vSigma_r\right)\right)}{\sum\limits_{i,r\in\N_a}{(\vT_W)_{ir}}^2\tr\left(\vT_S\frac{N}{n_i}\vSigma_i\frac{N}{n_r}\vT_S\vSigma_r\right)}\right)^2
\\[2ex]\leq&\left(1-\left(1-\frac{1}{B}\right)\cdot \frac{ \binom{n_{\min}-2} {2}^2}{\binom{n_{\min}} {2}^2}\right) \cdot \0(1).
\end{array}$\\\\

$\begin{array}{ll}
&\Var\left(\frac{\sum\limits_{i=1}^a\frac{N^2}{n_i^2}{(\vT_W)_{ii}}^2A_{i,3}^\star+\sum\limits_{i\neq r \in\N_a}\frac{N^2}{n_in_r}{(\vT_W)_{ir}}^2A_{i,r,2}^\star}{\tr^2\left(\left(\vT\vV_N\right)^2\right)}\right)
\\[3ex]
\hspace{-0.075cm}\stackrel{\ref{Var1}}{\leq}& \left[ \sqrt{\frac{\Var\left(2\sum\limits_{r < i\in \N_a} \frac{N^2}{n_in_r }{(\vT_W)_{ir}}^2A_{i,r,2}^\star\right)}{\tr^2\left(\left(\vT\vV_N\right)^2\right)}}+\sqrt{\frac{\Var\left(\sum\limits_{i=1}^a\frac{N}{n_i}  {(\vT_W)_{ii}}^2A_{i,3}^\star\right)}{\tr^2\left(\left(\vT\vV_N\right)^2\right)}}\right]^2
\end{array}$\\
$\begin{array}{ll}
\leq&\left(1-\left(1-\frac{1}{B}\right)\cdot \frac{ \binom{n_{\min}-2} {2}^2}{\binom{n_{\min}} {2}^2}\right) \cdot \0(1).
\end{array}$\\\\

So again this is a zero sequence, and $A_4^\star$ is an unbiased and dimensional stable (i.e. also ratio consistent) estimator of $\tr\left(\left(\vT\vV_N\right)^2\right)$.
\end{proof}

%
%
%
%
%
%
%

\subsection{Proofs of Section~\ref{Better Approximations}}

\renewcommand{\arraystretch}{1}
\begin{LeA}\label{MSchae2}
For
\[\Lambda_{1}(\ell_{1,1},\dots,\ell_{6,a})={\vZ_{(\ell_{1,1},\ell_{2,1},\dots,\ell_{1,a},\ell_{2,a})}}^\top \vT \vZ_{(\ell_{3,1},\ell_{4,1},\dots,\ell_{3,a},\ell_{4,a})},\]
\[\Lambda_{2}(\ell_{1,1},\dots,\ell_{6,a})={\vZ_{(\ell_{3,1},\ell_{4,1},\dots,\ell_{3,a},\ell_{4,a})}}^\top \vT \vZ_{(\ell_{5,1},\ell_{6,1},\dots,\ell_{5,a},\ell_{6,a})}, \]
\[\Lambda_{3}(\ell_{1,1},\dots,\ell_{6,a})={\vZ_{(\ell_{5,1},\ell_{6,1},\dots,\ell_{5,a},\ell_{6,a})}}^\top \vT \vZ_{(\ell_{1,1},\ell_{2,1},\dots,\ell_{1,a},\ell_{2,a})}, \]
we define 
\[C_{5}=\sum\limits_{\begin{footnotesize}\substack{\ell_{1,1},\dots, \ell_{6,1}=1\\\ell_{1,1}\neq\dots\neq \ell_{6,1}}\end{footnotesize}}^{n_1}
\dots\sum\limits_{\begin{footnotesize}\substack{\ell_{1,a},\dots, \ell_{6,a}=1\\\ell_{1,a}\neq\dots\neq \ell_{6,a}}\end{footnotesize}}^{n_a}
\frac{\Lambda_{1}(\ell_{1,1},\dots,\ell_{6,a})\cdot \Lambda_{2}(\ell_{1,1},\dots,\ell_{6,a})\Cdot \Lambda_{3}(\ell_{1,1},\dots,\ell_{6,a})}{8\cdot \prod\limits_{i=1}^a\frac {  n_i! }{\left(n_i-6\right)!} }.\]\\
With this notation it follows that\\
$
\E\left(C_{5}\right)=\tr\left(\left(\vT\vV_N\right)^3\right),\hspace{2cm}
\Var\left(C_{5}\right)\leq\frac{\left(\prod\limits_{i=1}^a {n_i\choose 6}-\prod\limits_{i=1}^a\binom {n_i-6} {6}\right)}{ \prod\limits_{i=1}^a\binom {n_i} {6}}\cdot 27\tr^3\left(\left( \vT\vV_N\right)^2\right).$
\end{LeA}
\begin{proof}
Set
$$\widetilde \vZ_{(\ell_{3,1},\ell_{4,1},\dots,\ell_{3,a},\ell_{4,a})}:=\left(\sqrt 2 \vV_N^{1/2}\right)^{-1}\vZ_{(\ell_{3,1},\ell_{4,1},\dots,\ell_{3,a},\ell_{4,a})}\sim \mathcal{N}_{ad}\left(\boldsymbol{0}_{ad},\vI_{ad}\right).$$ 
It then follows that\\

$ \begin{array}{ll}
&\E\left(\vT\vZ_{(\ell_{3,1},\ell_{4,1},\dots,\ell_{3,a},\ell_{4,a})}\cdot {\vZ_{(\ell_{3,1},\ell_{4,1},\dots,\ell_{3,a},\ell_{4,a})}}^\top \vT^\top\right)
\\
[0.80ex]
=&\E\left(\left(\sqrt 2 \vT\vV_N^{1/2}\widetilde {Z}_{(\ell_{3,1},\ell_{4,1},\dots,\ell_{3,a},\ell_{4,a})}\right)\left(\sqrt 2 \vT\vV_N^{1/2}{\widetilde \vZ_{(\ell_{3,1},\ell_{4,1},\dots,\ell_{3,a},\ell_{4,a})}}\right)^\top\right)
\\[1.5ex]
=&2\vT\vV_N^{1/2}\E\left(\widetilde {\vZ}_{(\ell_{3,1},\ell_{4,1},\dots,\ell_{3,a},\ell_{4,a})} {{\widetilde{\vZ}}_{(\ell_{3,1},\ell_{4,1},\dots,\ell_{3,a},\ell_{4,a})}}^\top\right){\vV_N^{1/2}}^\top \vT
\\[1ex]
=&2\vT\vV_N^{1/2}\vI_{ad}{\vV_N^{1/2}}^\top \vT=2\vT\vV_N\vT.\end{array}$\\\\\\
With the rules for conditional expectation and the involved independence it follows that\\

$\begin{array}{ll}
\E\left(C_{5}\right)&=
\sum\limits_{\begin{footnotesize}\substack{\ell_{1,1},\dots, \ell_{6,1}=1\\
\ell_{1,1}\neq\dots\neq \ell_{6,1}}\end{footnotesize}}^{n_1}\dots\sum\limits_{\begin{footnotesize}\substack{\ell_{1,a},\dots, \ell_{6,a}=1\\
\ell_{1,a}\neq\dots\neq \ell_{6,a}}\end{footnotesize}}^{n_a}\frac{\E\left(
\Lambda_{1}(\ell_{1,1},\dots,\ell_{6,a})\cdot \Lambda_{2}(\ell_{1,1},\dots,\ell_{6,a})\Cdot \Lambda_{3}(\ell_{1,1},\dots,\ell_{6,a})\right)}{8\cdot \prod\limits_{i=1}^a\frac {  n_i! }{\left(n_i-6\right)!} }
\\[3ex]
&=\sum\limits_{\begin{footnotesize}\substack{\ell_{1,1},\dots, \ell_{6,1}=1\\\ell_{1,1}\neq\dots\neq \ell_{6,1}}\end{footnotesize}}^{n_1}
\dots\sum\limits_{\begin{footnotesize}\substack{\ell_{1,a},\dots, \ell_{6,a}=1\\\ell_{1,a}\neq\dots\neq \ell_{6,a}}\end{footnotesize}}^{n_a}\frac{\E\left(
{\vZ_{(1,2)}}^\top \vT\vZ_{(3,4)}\cdot {\vZ_{(3,4)}}^\top \vT \vZ_{(5,6)}\cdot {\vZ_{(5,6)}}^\top \vT \vZ_{(1,2)} \right)}{8 \cdot\prod\limits_{i=1}^a\frac {  n_i! }{\left(n_i-6\right)!} }
\\[4ex]
\vspace{0.15cm}&=\frac 1 8\E\left({\vZ_{(1,2)}}^\top \vT\vZ_{(3,4)}\cdot {\vZ_{(3,4)}}^\top \vT\vZ_{(5,6)}\cdot {\vZ_{(5,6)}}^\top \vT \vZ_{(1,2)} \right)
\end{array}$\\
$\begin{array}{ll}
\textcolor{white}{\E\left(C_{5}\right)}&=\frac 1 8\E\left(\E\left({\vZ_{(1,2)}}^\top \vT \vZ_{(3,4)}\cdot {\vZ_{(3,4)}}^\top \vT\vZ_{(5,6)}\cdot {\vZ_{(5,6)}}^\top \vT \vZ_{(1,2)}\bigm| \vZ_{(1,2)} \right)\right)

\\[1.0ex]
&=\frac 1 8\E\left({\vZ_{(1,2)}}^\top\E\left(
 \vT\vZ_{(3,4)}\cdot {\vZ_{(3,4)}}^\top \vT \vZ_{(5,6)}\cdot {\vZ_{(5,6)}}^\top \vT\right) \vZ_{(1,2)}\right)
\\[0.8ex]
&=\frac 4 8\E\left({\vZ_{(1,2)}}^\top   \vT\vV_N\vT \vT\vV_N\vT \vZ_{(1,2)} \right)
=\frac 1 2\tr((\vT\vV_N\vT \vT\vV_N \vT) 2\vV_N)=\tr\left(\left(\vT\vV_N\right)^3\right).
\end{array}$\\\\
Due to the fact that all ${\vX}_{i,j}$ are identically distributed we can neglect the concrete indices, as long as we maintain the structure of dependence of the bilinear forms. The last term fulfills the requirements from \Crp{QF4} with $\vZ_{(1,2)}\sim \mathcal{N}\left(\boldsymbol{0}_{ad},2 \vV_N\right)$ and the matrix $\vT\vV_N\vT \vT\vV_N\vT$. \\

For the calculation of the variance it is useful to diagonalize the matrix ${\vV_N^{1/2}}^\top \vT\vV_N^{1/2}$: It exists an orthogonal matrix $\vP$ with $\vP {\vV_N^{1/2}}^\top \vT\vV_N^{1/2}\vP^\top= \vD=\diag\left(\lambda_1,\dots,\lambda_{ad}\right)$, where $\lambda_i$ are the eigenvalues of ${\vV_N^{1/2}}^\top \vT {\vV_N^{1/2}}$. We define $\vE_i:=\vP \widetilde \vZ_{(i,j)}$ so with the properties of the standard normal distribution $\vE_i\sim\mathcal{N}_{ad}(\boldsymbol{0}_{ad},\vI_{ad})$, where the $\vE_i$ are independent for different indices. Thus, we can rewrite \\\\
$\begin{array}{l}
{\vZ_{(1,2)}}^\top \vT\vZ_{(3,4)}={\widetilde {\vZ}_{(1,2)}}^\top 2{\vV_N^{1/2}}^\top \vT\vV_N^{1/2} \widetilde \vZ_{(3,4)}
=2 {\widetilde \vZ_{(1,2)}}^\top \vP^\top \vD \vP\widetilde \vZ_{(3,4)}=2\vE_1^\top \vD \vE_3.
\end{array}$\\\\\\
With this argument for all three random variables it follows for the second moment that\\

$\begin{array}{ll}
&\E\left(\left[\vE_1^\top \vD \vE_3\vE_3^\top \vD \vE_5\vE_5^\top \vD \vE_1\right]^2\right)
\\[1.2ex]
=&\E\left(\left[\sum_{i=1}^{ad}\lambda_i E_1^{(i)}E_{3}^{(i)}\right]^2\left[\sum_{j=1}^{ad}\lambda_jE_{3}^{(j)}E_5^{(j)}\right]^2\left[\sum_{h=1}^{ad}\lambda_h E_5^{(h)}E_1^{(h)}\right]^2\right)
\\[1.5ex]
=&\sum\limits_{i_1,i_2,j_1,j_2,h_1,h_2=1}^{ad}\lambda_{i_1}\lambda_{i_2}\lambda_{j_1}\lambda_{j_2}\lambda_{h_1}\lambda_{h_2}\E\left(E_1^{(i_1)}E_3^{(i_1)}E_1^{(i_2)}E_3^{(i_2)}
E_3^{(j_1)}E_5^{(j_1)}E_3^{(j_2)}E_5^{(j_2)}E_5^{(h_1)}E_1^{(h_1)}E_5^{(h_2)}E_1^{(h_2)}\right).
\end{array}$\\\\
Now we consider the expectation value for the different combinations. If all indices are equal, it is given by
\[\E\left(\left(E_1^{(1)}\right)^4\left(E_3^{(1)}\right)^4\left(E_5^{(1)}\right)^4\right)
=3^3=27.\]\\

Moreover, for $i_1=i_2\neq h_1=h_2$ and $h_2\neq j_1=j_2\neq i_1$ it holds that
\[\begin{array}{ll}
&
\E\left(\left(E_1^{(1)}\right)^2\left(E_3^{(1)}\right)^2\left(E_3^{(2)}\right)^2\left(E_5^{(2)}\right)^2\left(E_1^{(3)}\right)^2\left(E_5^{(3)}\right)^2\right)
=1^6=1.\end{array}\]
Next, the case $i_1=i_2=j_1=j_2\neq h_1=h_2$ is considered (noting this result can also be used for both analogue combinations):

\[\E\left(\left(E_1^{(1)}\right)^2\left(E_3^{(1)}\right)^4\left(E_5^{(1)}\right)^2\left(E_1^{(2)}\right)^2\left(E_5^{(2)}\right)^2\right)=3^1\cdot 1^4=3.\]\\
 Finally, we consider the combination $i_1=j_1=h_1\neq i_2=j_2=h_2$ and obtain
\[\begin{array}{c}\E\left(\left[E_1^{(1)}E_3^{(1)}E_1^{(2)}E_3^{(2)}E_5^{(1)}E_5^{(2)}\right]^2\right)=\prod\limits_{i=1}^2\E\left(\left[E_1^{(i)}\right]^2\right)\E\left(\left[E_3^{(i)}\right]^2\right)
\E\left(\left[E_5^{(i)}\right]^2\right)={1^3}^2.\end{array}\]

This is also true for $i_1=j_2=h_1\neq i_2=j_1=h_2$ and the analogue combinations, so, all in all, we have 4 combinations of this kind. All other index combinations lead to expectation zero because in this combinations at least one index appears just one time in the product. Therefore with the independence and the fact that all random variables $E_i$ are centered it is true that\\

$\begin{array}{ll}
&\E\left(\left[\vE_1^\top \vD \vE_3\vE_3^\top \vD \vE_5\vE_5^\top \vD \vE_1\right]^2\right)
\\[2ex]
=&\sum\limits_{i=1}^{ad}\lambda_i^6 \cdot 27+\sum\limits_{\begin{footnotesize}\substack{i,j=1\\ i\neq j}\end{footnotesize}}^{ad} \lambda_i^3\lambda_j^3\cdot 1\cdot 4+\sum\limits_{\begin{footnotesize}\substack{i,j=1\\ i\neq j}\end{footnotesize}}^d \lambda_{i}^2\lambda_j^4 \cdot 9+\sum\limits_{\begin{footnotesize}\substack{i,j,h=1\\ i\neq j\neq h}\end{footnotesize}}^{ad}\lambda_i^2\lambda_j^2\lambda_h^2 
\\[3ex]
=&23\sum\limits_{i=1}^{ad}\lambda_i^6 +4\left(\sum\limits_{\begin{footnotesize}\substack{i,j=1\\ i\neq j}\end{footnotesize}}^{ad} \lambda_i^3\lambda_j^3+\sum\limits_{i=j=1}^{ad}\lambda_i^3 \lambda_j^3\right)+9\sum\limits_{\begin{footnotesize}\substack{i,j=1\\ i\neq j}\end{footnotesize}}^{ad} \lambda_{i}^2\lambda_j^4 +\sum\limits_{\begin{footnotesize}\substack{i,j,h=1\\ i\neq j\neq h}\end{footnotesize}}^{ad}\lambda_i^2\lambda_j^2\lambda_h^2
\end{array}$\\
$\begin{array}{ll}
=&17\sum\limits_{i=1}^{ad}\lambda_i^6 +4\tr^2\left(\left( \vT\vV_N\right)^3\right)+3\sum\limits_{\begin{footnotesize}\substack{i,j=1\\ i\neq j}\end{footnotesize}}^{ad} \lambda_{i}^2\lambda_j^4+6\left(\sum\limits_{\begin{footnotesize}\substack{i,j=1\\ i\neq j}\end{footnotesize}}^{ad} \lambda_{i}^2\lambda_j^4+\sum\limits_{i=j=1}^{ad}\lambda_i^2\lambda_j^4\right)+\sum\limits_{\begin{footnotesize}\substack{i,j,h=1\\ i\neq j\neq h}\end{footnotesize}}^{ad}\lambda_i^2\lambda_j^2\lambda_h^2
\\[3ex]
=&17\sum\limits_{i=1}^{ad}\lambda_i^6 +4\tr^2\left(\left( \vT\vV_N\right)^3\right)+3\sum\limits_{\begin{footnotesize}\substack{i,j=1\\ i\neq j}\end{footnotesize}}^{ad} \lambda_{i}^2\lambda_j^4+6\tr\left(\left( \vT\vV_N\right)^4\right)\tr\left(\left( \vT\vV_N\right)^2\right)+\sum\limits_{\begin{footnotesize}\substack{i,j,h=1\\ i\neq j\neq h}\end{footnotesize}}^{ad}\lambda_i^2\lambda_j^2\lambda_h^2\\[3ex]
\hspace{-0.10cm}\stackrel{\ref{Spur1}}{\leq}&17\sum\limits_{i=1}^{ad}\lambda_i^6 +4\tr^2\left(\left( \vT\vV_N\right)^3\right)+3\sum\limits_{\begin{footnotesize}\substack{i,j=1\\ i\neq j}\end{footnotesize}}^{ad} \lambda_{i}^2\lambda_j^4+6\tr^3\left(\left( \vT\vV_N\right)^2\right)+\sum\limits_{\begin{footnotesize}\substack{i,j,h=1\\ i\neq j\neq h}\end{footnotesize}}^{ad}\lambda_i^2\lambda_j^2\lambda_h^2
\end{array}$\\
$\begin{array}{ll}
\hspace{-0.10cm}\stackrel{\ref{Spur1}}{\leq}&20\tr^2\left(\left( \vT\vV_N\right)^3\right)+6\tr^3\left(\left( \vT\vV_N\right)^2\right)+\left(\sum\limits_{\begin{footnotesize}\substack{i,j,h=1\\ i\neq j\neq h}\end{footnotesize}}^{ad}\lambda_i^2\lambda_j^2\lambda_h^2+3\sum\limits_{\begin{footnotesize}\substack{i,j=1\\ i\neq j}\end{footnotesize}}^{ad}
\lambda_{i}^2\lambda_j^4+\sum\limits_{i=1}^{ad}\lambda_i^6\right)
\\[2.5ex]
=&20\tr^2\left(\left( \vT\vV_N\right)^3\right)+7\tr^3\left(\left( \vT\vV_N\right)^2\right)
\\
\hspace{-0.10cm}\stackrel{\ref{Spur1}}{\leq} &20\tr\left(\left( \vT\vV_N\right)^4\right)\tr\left(\left( \vT\vV_N\right)^2\right)+7\tr^3\left(\left(\vT\vV_N\right)^2\right)
\\
\hspace{-0.10cm}\stackrel{\ref{Spur1}}{\leq} &27\tr^3\left(\left( \vT\vV_N\right)^2\right).
\end{array}$\\\\\\
So we can control the variance by\\\\
$\begin{array}{ll}
\Var(C_{5})&
\hspace{-0.075cm}\stackrel{\ref{Var1}}
{\leq} \frac{\Var\left(\Lambda_{1}(1,2,3,4,5,6,\dots,5,6)\cdot \Lambda_{2}(1,2,3,4,5,6,\dots,5,6)\cdot \Lambda_{3}(1,2,3,4,5,6,\dots,5,6)\right)}{{64\cdot \prod\limits_{i=1}^a\binom {n_i} 6}\cdot{\left(\prod\limits_{i=1}^a\binom {n_i} 6-\prod\limits_{i=1}^a\binom {n_i-6} 6\right)^{-1}}}
\\[3ex]
&{\leq} \frac{\E\left(\left[\Lambda_{1}(1,2,3,4,5,6,\dots,5,6)\cdot \Lambda_{2}(1,2,3,4,5,6,\dots,5,6)\cdot \Lambda_{3}(1,2,3,4,5,6,\dots,5,6)\right]^2\right)}{{64\cdot \prod\limits_{i=1}^a\binom {n_i} 6}\cdot{\left(\prod\limits_{i=1}^a\binom {n_i} 6-\prod\limits_{i=1}^a\binom {n_i-6} 6\right)^{-1}}}
\end{array}$\\\\
$\begin{array}{ll}
\textcolor{white}{\Var(C_{5})}&= \frac{\E\left(\left[2^3\cdot \vE_1^\top \vD \vE_3\vE_3^\top \vD \vE_5\vE_5^\top \vD \vE_1\right]^2\right)}{{64\cdot \prod\limits_{i=1}^a\binom {n_i} 6}\cdot{\left(\prod\limits_{i=1}^a\binom {n_i} 6-\prod\limits_{i=1}^a\binom {n_i-6} 6\right)^{-1}}}
\\[3.5ex]

&\leq \frac{\left(\prod\limits_{i=1}^a {n_i\choose 6}-\prod\limits_{i=1}^a\binom {n_i-6} {6}\right)}{ \prod\limits_{i=1}^a\binom {n_i} {6}}\cdot 27\tr^3\left(\left( \vT\vV_N\right)^2\right).
\end{array}$\\
\end{proof}\vspace{\baselineskip}
With this result, we can construct an estimator for $\tau_P$ step by step:

\begin{LeA}\label{Schae7}
For $C_5$ as previously defined, it holds for fixed $a$ that

 \[\frac {C_5}{\tr^{3/2}\left(\left(\vT\vV_N\right)^2\right)}-\frac{\tr\left(\left(\vT\vV_N\right)^3\right)}{\tr^{3/2}\left(\left(\vT\vV_N\right)^2\right)}\hspace{0,3cm}\stackrel{P}{\longrightarrow}0\hspace*{1.0cm}\min(d,n_{\min})\to \infty  .\]\\
  It even holds in the asymptotic frameworks \eqref{eq: as frame 2}-\eqref{eq: as frame 3} if  $p>1$ exists with $n_{\min}=\0(a^p)$.
\end{LeA}

\begin{proof} From the previous lemma, we know that\\

$\begin{array}{ll}\E\left(\frac {C_5}{\tr^{3/2}\left(\left(\vT\vV_N\right)^2\right)}-\frac{\tr\left(\left(\vT\vV_N\right)^3\right)}{\tr^{3/2}\left(\left(\vT\vV_N\right)^2\right)}\right)&=
\E\left(\frac {C_5}{\tr^{3/2}\left(\left(\vT\vV_N\right)^2\right)}\right)-\frac{\tr\left(\left(\vT\vV_N\right)^3\right)}{\tr^{3/2}\left(\left(\vT\vV_N\right)^2\right)}
=0,\end{array}$\\\\
$\begin{array}{ll}
\Var\left(\frac {C_5}{\tr^{3/2}\left(\left(\vT\vV_N\right)^2\right)}-\frac{\tr\left(\left(\vT\vV_N\right)^3\right)}{\tr^{3/2}\left(\left(\vT\vV_N\right)^2\right)}\right)&=\frac {\Var(C_5)}{\tr^3\left(\left(\vT\vV_N\right)^2\right)}
\stackrel{\ref{MSchae2}}{\leq} 27\cdot\frac{\left(\prod\limits_{i=1}^a {n_i\choose 6}-\prod\limits_{i=1}^a\binom {n_i-6} {6}\right)}{ \prod\limits_{i=1}^a\binom {n_i} {6}}.
\end{array}$\\\\\\
For fixed $a$ this is a zero sequence.  If we consider $a\to \infty$ we need the existence of $p>1$ and $n_{\min}=\0(a^p)$ to guarantee that the upper border is a zero sequence.\\
So in both cases \Crp{Kons2} can be used.
\end{proof}
\begin{LeA}\label{Schae8}

Moreover $C_5$ holds for fixed $a$
\[ \frac{C_5^2}{\tr^{3}\left(\left(\vT\vV_N\right)^2\right)}-\tau_P \stackrel{P}{\longrightarrow} 0\hspace*{1.0cm} d,n_{\min}\to \infty .\]\\ If  $p>1$ exists  with $n_{\min}=\0(a^p)$, the convergence even holds in the asymptotic frameworks \eqref{eq: as frame 2}-\eqref{eq: as frame 3}.

\end{LeA}

\begin{proof} 

With the last lemma it follows for both cases that\\\\
$\begin{array}{ll}
\frac{C_5^2}{\tr^{3}\left(\left(\vT\vV_N\right)^2\right)}-\tau_P&=\left(\frac{C_5}{\tr^{3/2}\left(\left(\vT\vV_N\right)^2\right)}\right)^2-\left(\frac{\tr\left(\left(\vT\vV_N\right)^3\right)}{\tr^{3/2}\left(\left(\vT\vV_N\right)^2\right)}\right)^2
\\[2ex]
&\vspace{.25cm}=\left[\frac{C_5}{\tr^{3/2}\left(\left(\vT\vV_N\right)^2\right)}-\frac{\tr\left(\left(\vT\vV_N\right)^3\right)}{\tr^{3/2}\left(\left(\vT\vV_N\right)^2\right)}\right]\left[\frac{C_5}{\tr^{3/2}\left(\left(\vT\vV_N\right)^2\right)}+\frac{\tr\left(\left(\vT\vV_N\right)^3\right)}{\tr^{3/2}\left(\left(\vT\vV_N\right)^2\right)}\right]
\end{array}$\\
$\begin{array}{ll}
\vspace{.25cm}\textcolor{white}{\frac{C_5^2}{\tr^{3}\left(\left(\vT\vV_N\right)^2\right)}-\tau_P}&=o_P(1)\cdot \left[\frac{C_5}{\tr^{3/2}\left(\left(\vT\vV_N\right)^2\right)}-\frac{\tr\left(\left(\vT\vV_N\right)^3\right)}{\tr^{3/2}\left(\left(\vT\vV_N\right)^2\right)}+2\frac{\tr\left(\left(\vT\vV_N\right)^3\right)}{\tr^{3/2}\left(\left(\vT\vV_N\right)^2\right)}\right]
\\[1.2ex]
&=o_P(1)\cdot \left[o_P(1)+2\cdot\frac{\tr\left(\left(\vT\vV_N\right)^3\right)}{\tr^{3/2}\left(\left(\vT\vV_N\right)^2\right)}\right]
=o_P(1).
\end{array}$\\\\
For the last step we used that $\tau_P\in[0,1]$ which is known from \Crp{Bedingungen} and hence ${\tr\left(\left(\vT\vV_N\right)^3\right)}\left/{\tr^{3/2}\left(\left(\vT\vV_N\right)^2\right)}\right.\in[-1,1]$. As a product of a bound term and a term which converges to zero in probability, it also converges to zero in probability and with Slutzky's Lemma the result follows.
\end{proof}\vspace{\baselineskip}


%
%

\begin{proof}[Proof of \Cref{Lemma: fP Estimate} ]\label{Proof fP}

From  \Crp{Schae1} together with \Crp{Kons2} it follows
\[ \frac{A_4}{\tr\left(\left(\vT\vV_N\right)^2\right)}\stackrel{P}{\longrightarrow} 1\hspace*{0,5cm}\text{ and therefore }\hspace*{0,5cm}\frac{\tr^3\left(\left(\vT\vV_N\right)^2\right)}{A_4^3}\stackrel{P}{\longrightarrow} 1\hspace{01.0cm}\text{for}\hspace{0.5cm} n_{\min}\to \infty,\] independent of $d$ or $a$.
With \Crp{Schae8} it follows\\
\[ \frac{C_5^2}{\tr^{3}\left(\left(\vT\vV_N\right)^2\right)}-\tau_P \stackrel{P}{\longrightarrow} 0\hspace{01.0cm}\text{for}\hspace{0.5cm} d,n_{\min}\to \infty\] or under the additional condition also in the asymptotic frameworks \eqref{eq: as frame 2} -\eqref{eq: as frame 3} .\\

With these limits in both cases we can calculate\\\\
$\begin{array}{ll}
\frac{C_5^2}{A_4^3}-\tau_P&=\frac{C_5^2}{\tr^{3}\left(\left(\vT\vV_N\right)^2\right)}\cdot \frac{\tr^3\left(\left(\vT\vV_N\right)^2\right)}{A_4^3}-\tau_P
\\[1.2ex]
&=\frac{C_5^2}{\tr^{3}\left(\left(\vT\vV_N\right)^2\right)}\cdot (1+o_P(1))-\tau_P
\\[1.2ex]
&=\frac{C_5^2}{\tr^{3}\left(\left(\vT\vV_N\right)^2\right)}-\tau_P+\left(\frac{C_5^2}{\tr^{3}\left(\left(\vT\vV_N\right)^2\right)}-\tau_P+\tau_P\right)\cdot o_P(1)
\\[1.8ex]
&=o_P(1)+o_P(1)\cdot o_P(1)+\tau_P\cdot o_P(1)
=o_P(1).
\end{array}$\\\\
As in the previous lemma we used $\tau_P\in[0,1]$ and Slutzky.
\end{proof}\vspace{\baselineskip}
For ${C_5^\star}$  the properties are shown in a similar way as in \Crp{ZSchae1}.

\begin{LeA}\label{ZSchae2} For
\[\Lambda_{1}(\ell_{1,1},\dots,\ell_{6,a})={\vZ_{(\ell_{1,1},\ell_{2,1},\dots,\ell_{1,a},\ell_{2,a})}}^\top \vT \vZ_{(\ell_{3,1},\ell_{4,1},\dots,\ell_{3,a},\ell_{4,a})},\]
\[\Lambda_{2}(\ell_{1,1},\dots,\ell_{6,a})={\vZ_{(\ell_{3,1},\ell_{4,1},\dots,\ell_{3,a},\ell_{4,a})}}^\top \vT \vZ_{(\ell_{5,1},\ell_{6,1},\dots,\ell_{5,a},\ell_{6,a})}, \]
\[\Lambda_{3}(\ell_{1,1},\dots,\ell_{6,a})={\vZ_{(\ell_{5,1},\ell_{6,1},\dots,\ell_{5,a},\ell_{6,a})}}^\top \vT \vZ_{(\ell_{1,1},\ell_{2,1},\dots,\ell_{1,a},\ell_{2,a})},\]
define 
\[{C_5^\star}\left(B\right)=\frac 1 {8 \cdot B}\sum\limits_{b=1}^B 
 \Lambda_1(\vsigma(b,6)) \cdot \Lambda_2(\vsigma(b,6)) \cdot \Lambda_3(\vsigma(b,6)).\] \\
Then it holds
 
 \[\begin{array}{l}
 \E\left({C_5^\star}(B)\right)=\tr\left(\left(\vT\vV_N\right)^3\right),
 \\[2ex]
\Var\left({C_5^\star}(B)\right)\leq\left(1-\left(1-\frac{1}{B}\right)\cdot\prod\limits_{i=1}^a\frac{ \binom{n_i-6} {6}}{\binom{n_i} {6}}\right)\cdot 27\tr^3\left(\left(\vT\vV_N\right)^2\right)
.\end{array}\]\end{LeA}

\begin{proof}With the same steps as in the previous lemma and by using the fact that expectation and variance do not depend on the concrete indices but rather on the structure of independences we get\\

$\begin{array}{ll}
\E\left({C_5^\star}(B)\right)

&=\frac{1}{8B}\sum\limits_{b=1}^B\E\left( \Lambda_1(\vsigma(b,6)) \cdot \Lambda_2(\vsigma(b,6)) \cdot \Lambda_3(\vsigma(b,6))\right)
\\[1.8ex]
&=\frac{1}{8B}\sum\limits_{b=1}^B\E\left(\Lambda_{1}(\ell_{1,1},\dots,\ell_{6,a})\cdot\Lambda_{2}(\ell_{1,1},\dots,\ell_{6,a})\cdot\Lambda_{3}(\ell_{1,1},\dots,\ell_{6,a})\right).
\\[1.8ex]

&\hspace{-0.12cm}\stackrel{\ref{MSchae2}}{=}\frac{1}{8B}\sum\limits_{b=1}^B
\tr\left(\left(2\vT\vV_N\right)^3\right)=\tr\left(\left(\vT\vV_N\right)^3\right).
\end{array}$\\\\

$\begin{array}{l}
\Var\left(\E\left({C_5^\star}(B)|\F(\vsigma(B,6))\right)\right)
=\Var\left(\tr\left(\left(\vT\vV_N\right)^3\right)\right)=0.
\end{array}$\\\\\\
$\begin{array}{ll}
\Var\left({C_5^\star}(B)\right)
&=0+\E\left(\Var\left({C_5^\star}(B)|\F(\vsigma(B,6))\right)\right)
\\[1.0ex]&
\hspace{-0.075cm}\stackrel{\ref{Var1}}{\leq}\frac{1}{64 B^2}\E\left(\sum\limits_{(j,\ell)\in \N_B\times \N_B \setminus M(B,\vsigma(b,6))} \Var\left(\Lambda_{1}(\vsigma(j,6))\Lambda_2(\vsigma(j,6))\Lambda_{3}(\vsigma(j,6))|\F(\vsigma(B,6))\right)\right)
\\[2.4ex]
&=\frac{\E\left(|\N_B\times \N_B \setminus M(B,\vsigma(b,6))|\right)}{B^2}\cdot \frac{\Var\left({\vZ_{(1,2)}}^\top \vT\vZ_{(3,4)}\cdot {\vZ_{(3,4)}}^\top \vT\vZ_{(5,6)}\cdot{\vZ_{(5,6)}}^\top \vT\vZ_{(1,2)}\right)}{64}
\\[1.0ex]
&\hspace{-0.12cm}\stackrel{\ref{MSchae2}}{\leq}
\left(1-\left(1-\frac{1}{B}\right)\cdot \prod\limits_{i=1}^a\frac{ \binom{n_i-6} {6}}{\binom{n_i} {6}}\right)\cdot 27\tr^3\left(\left(\vT\vV_N\right)^2\right).
\end{array}$\\

\end{proof}\vspace{\baselineskip}

\begin{proof}[Proof of \Crp{theo: fP}]
 With   \Cref{ZSchae2}  we recognize ${\tau_P \to 1\Leftrightarrow\widehat{\tau_P}\stackrel{P}{\longrightarrow} 1}$ and $\tau_P \to 0\Leftrightarrow\widehat{\tau_P}\stackrel{P}{\longrightarrow} 0$. Therefore $f_P \to 1 \Leftrightarrow \widehat{f_P}\stackrel{P}{\longrightarrow} 1$ and
$f_P\to \infty\Leftrightarrow\widehat f_P\stackrel{P}{\longrightarrow} \infty$. This is the only condition  needed for the proof of  \cite{Paper1}[Theorem 3.1], so the result follows.
\end{proof}\vspace{\baselineskip}
 
Although $n_{\min}=\0(a^p)$ with $p>0$ is not too critical in most settings we additionally developed an estimator which can be used without any restrictions.\\\\
For this estimator another random vector has to be introduced: The random vector $\pi_{j,i}$ represents a random permutation of the numbers $1,\dots,n_i,$ where $\pi_{j,i}$ are independent for different $i$ or $j$ and $\pi_{j,i}(l) $ denotes its $l$-th element. Then we define
 \[C_{7}\left(w\right)=\frac 1 {w}\sum\limits_{j=1}^{w}\sum\limits_{\ell_1\neq\dots\neq\ell_6=1}^{n_{\min}}
\frac{\Lambda_{4}\left(j;\ell_{1},\dots,\ell_{6}\right)\cdot \Lambda_{5}\left(j;\ell_{1},\dots,\ell_{6}\right)\Cdot \Lambda_{6}\left(j;\ell_{1},\dots,\ell_{6}\right)}{8\cdot \frac {  n_{\min}! }{\left(n_{\min}-6\right)!} }\]
 with

  \[\begin{array}{l}\Lambda_{4}\left(j;\ell_{1},\dots,\ell_{6}\right)={\vZ^{\vpi_j}_{(\ell_{1},\ell_2)} }^\top \vT \vZ^{\vpi_j}_{(\ell_{3},\ell_4)} ,
\\[1.2ex]
\Lambda_{5}\left(j;\ell_{1},\dots,\ell_{6}\right)={\vZ^{\vpi_j}_{(\ell_{3},\ell_4)} }^\top \vT \vZ^{\vpi_j}_{(\ell_{5},\ell_6)} , 
\\[1.2ex]
\Lambda_{6}\left(j;\ell_{1},\dots,\ell_{6}\right)={\vZ^{\vpi_j}_{(\ell_{5},\ell_6)} }^\top \vT \vZ^{\vpi_j}_{(\ell_{1},\ell_2)} .\end{array}\]and
\[\vZ^{\vpi_j}_{(\ell_{1},\ell_2)} :=\vZ_{\left(\pi_{j,1}({\ell_1}),\pi_{j,1}({\ell_2}),\pi_{j,2}({\ell_1}),\dots,\pi_{j,a}({\ell_1}),\pi_{j,a}({\ell_2})\right)}\] \\
This estimator again uses Z, but different to $C_5$ the indices are the same for all groups. However the highest index is $n_{\min}$ and some index combinations are unachievable. For this reason, the above random permutations were used. So first the observations in each group were rearranged randomly and with this rearranged samples we calculated the sum of the used terms. Thereafter, we again rearrange the observations and the same terms as before are calculated. If these values were summed up and divided by the number of rearrangements we get an alternative for $C_5$ which is shown in the following lemma.

{
\begin{LeA}\label{MSchae6}
For $C_7$ as defined before
 it holds  \[\E\left(C_7(w)\right)=\tr\left(\left(\vT\vV_N\right)^3\right)\hspace{1cm}\Var\left(C_7(w)\right)\leq \left(\frac{\frac {  n_{\min}! }{\left(n_{\min}-6\right)!} -\frac {  \left(n_{\min}-6\right)! }{\left(n_{\min}-12\right)!} }{\frac {  n_{\min}! }{\left(n_{\min}-6\right)!} }\right)\cdot \0\left(\tr^3\left(\left(\vT\vV_N\right)^2\right)\right).\]
 \end{LeA}}

{
\begin{proof}
Again we calculate\\
$\begin{array}{ll}\E\left(C_7\left(w\right)\right)
&=\frac 1 {w}\sum\limits_{j=1}^{w}\sum\limits_{\ell_1\neq\dots\neq\ell_6=1}^{n_{\min}}
\frac{\E\left(\Lambda_{4}\left(j;\ell_{1},\dots,\ell_{6}\right)\cdot \Lambda_{5}\left(j;\ell_{1},\dots,\ell_{6}\right)\Cdot \Lambda_{6}\left(j;\ell_{1},\dots,\ell_{6}\right)\right)}{8\cdot \frac {  n_{\min}! }{\left(n_{\min}-6\right)!} }
\\[2.5ex]
&=\frac 1 {w}\sum\limits_{j=1}^{w}\sum\limits_{\ell_1\neq\dots\neq\ell_6=1}^{n_{\min}}
\frac{\E\left(\Lambda_{4}\left(j;1,\dots,6\right)\cdot \Lambda_{5}\left(j;1,\dots,6\right)\Cdot \Lambda_{6}\left(j;1,\dots,6\right)\right)}{8\cdot \frac {  n_{\min}! }{\left(n_{\min}-6\right)!} }
=\tr\left(\left(\vT\vV_N\right)^3\right).
\end{array}$\\\\\\
Because of the fact that all groups use the same indices, the number of remaining indexcombinations simplifies and we receive
\[\begin{array}{ll}&\Var\left(\sum\limits_{\ell_1\neq\dots\neq\ell_6=1}^{n_{\min}}
\frac{\Lambda_{4}\left(j;\ell_{1},\dots,\ell_{6}\right)\cdot \Lambda_{5}\left(j;\ell_{1},\dots,\ell_{6}\right)\Cdot \Lambda_{6}\left(j;\ell_{1},\dots,\ell_{6}\right)}{8\cdot \frac {  n_{\min}! }{\left(n_{\min}-6\right)!} }\right)
\\[2.5ex]
\leq&\left(\frac{\frac {  n_{\min}! }{\left(n_{\min}-6\right)!} -\frac {  \left(n_{\min}-6\right)! }{\left(n_{\min}-12\right)!} }{\frac {  n_{\min}! }{\left(n_{\min}-6\right)!} }\right)\Var\left(
\Lambda_{4}\left(j;\ell_{1},\dots,\ell_{6}\right)\cdot \Lambda_{5}\left(j;\ell_{1},\dots,\ell_{6}\right)\Cdot \Lambda_{6}\left(j;\ell_{1},\dots,\ell_{6}\right)\right)
\\[2.5ex]
\leq&\left(\frac{\frac {  n_{\min}! }{\left(n_{\min}-6\right)!} -\frac {  \left(n_{\min}-6\right)! }{\left(n_{\min}-12\right)!} }{\frac {  n_{\min}! }{\left(n_{\min}-6\right)!} }\right)\cdot \0\left(\tr^3\left(\left(\vT\vV_N\right)^2\right)\right).
\end{array}\]\\\\
For the sum this leads to\\\\
$\begin{array}{ll}\Var\left(C_7\left(w\right)\right)&=\Var\left(\frac 1 {w}\sum\limits_{j=1}^{w}\sum\limits_{\ell_1\neq\dots\neq\ell_6=1}^{n_{\min}}
\frac{\Lambda_{4}\left(j;\ell_{1},\dots,\ell_{6}\right)\cdot \Lambda_{5}\left(j;\ell_{1},\dots,\ell_{6}\right)\Cdot \Lambda_{6}\left(j;\ell_{1},\dots,\ell_{6}\right)}{8\cdot \frac {  n_{\min}! }{\left(n_{\min}-6\right)!} }\right)
\\[2.5ex]
&\hspace{-0.075cm}\stackrel{\ref{Var1}}{\leq} \frac{1}{w^2}\sum\limits_{j_1,j_2=1}^ {w}\Var\left(\sum\limits_{\ell_1\neq\dots\neq\ell_6=1}^{n_{\min}}
\frac{\Lambda_{4}\left(j;\ell_{1},\dots,\ell_{6}\right)\cdot \Lambda_{5}\left(j;\ell_{1},\dots,\ell_{6}\right)\Cdot \Lambda_{6}\left(j;\ell_{1},\dots,\ell_{6}\right)}{8\cdot \frac {  n_{\min}! }{\left(n_{\min}-6\right)!} }\right)\end{array}$\\\\
$\begin{array}{ll}\textcolor{white}{\Var\left(C_7\left(w\right)\right)}
&\leq \frac{1}{w^2}\sum\limits_{j_1,j_2=1}^ {w}\left(\frac{\frac {  n_{\min}! }{\left(n_{\min}-6\right)!} -\frac {  \left(n_{\min}-6\right)! }{\left(n_{\min}-12\right)!} }{\frac {  n_{\min}! }{\left(n_{\min}-6\right)!} }\right)\cdot \0\left(\tr^3\left(\left(\vT\vV_N\right)^2\right)\right)
\\[2.5ex]
&= \left(\frac{\frac {  n_{\min}! }{\left(n_{\min}-6\right)!} -\frac {  \left(n_{\min}-6\right)! }{\left(n_{\min}-12\right)!} }{\frac {  n_{\min}! }{\left(n_{\min}-6\right)!} }\right)\cdot \0\left(\tr^3\left(\left(\vT\vV_N\right)^2\right)\right).
\end{array}$\\
\end{proof}\vspace{1\baselineskip}}
Simulations (not shown here) show that higher values for $w$ lead to better estimations.

{
\begin{LeA}\label{MSchae7}
For $C_7$ as previously  defined, it holds \[ \frac{C_7^2}{\tr^{3}\left(\left(\vT\vV_N\right)^2\right)}-\tau_P \stackrel{P}{\longrightarrow} 0\hspace{1cm} \text{for} \hspace*{0.5cm} n_{\min}\to \infty,\]
independent of a or d. Therefore this holds for  the asymptotic frameworks \eqref{eq: as frame 1}-\eqref{eq: as frame 3}.
\end{LeA}
\begin{proof}
With the previous lemma we know\\\\
$\begin{array}{ll}\E\left(\frac {C_7(w)}{\tr^{3/2}\left(\left(\vT\vV_N\right)^2\right)}-\frac{\tr\left(\left(\vT\vV_N\right)^3\right)}{\tr^{3/2}\left(\left(\vT\vV_N\right)^2\right)}\right)&=
\E\left(\frac {C_7(w)}{\tr^{3/2}\left(\left(\vT\vV_N\right)^2\right)}\right)-\frac{\tr\left(\left(\vT\vV_N\right)^3\right)}{\tr^{3/2}\left(\left(\vT\vV_N\right)^2\right)}
=0,\end{array}$
\\\\
$\begin{array}{ll}
\Var\left(\frac {C_7(w)}{\tr^{3/2}\left(\left(\vT\vV_N\right)^2\right)}-\frac{\tr\left(\left(\vT\vV_N\right)^3\right)}{\tr^{3/2}\left(\left(\vT\vV_N\right)^2\right)}\right)&=\frac {\Var\left(C_7(w)\right)}{\tr^3\left(\left(\vT\vV_N\right)^2\right)}
{\leq}\left(\frac{\frac {  n_{\min}! }{\left(n_{\min}-6\right)!} -\frac {  \left(n_{\min}-6\right)! }{\left(n_{\min}-12\right)!} }{\frac {  n_{\min}! }{\left(n_{\min}-6\right)!} }\right)\cdot\0\left(1\right).
\end{array}$
\\\\\\
So exactly the same steps as in the proof of \Cref{Lemma: fP Estimate} , which in this case uses that the zero sequence not depends on $a$ or $d$, leads to the result.
\end{proof}}\vspace{\baselineskip}

But for the calculation of this estimator we need  $w\cdot{ n_{\min}! }/{\left(n_{\min}-6\right)!}$ summations.
Thus, a subsampling-type version of $C_7$ is necessary which is now defined.
\begin{LeA}\label{ZSchae4}
For each $b=1,\dots,B$ we independently draw random subsamples $\vsigma_0(b,6)$ of length $6$ from $\{1,\dots,n_{\min}\}$ and define\\\\
${C_7^\star}\left(w,B\right)=\sum\limits_{j=1}^{w}\sum\limits_{b=1}^{B}
\frac{\Lambda_{4}\left(j;\vsigma_0(b,6)\right) \Lambda_{5}\left(j;\vsigma_0(b,6)\right) \Lambda_{6}\left(j;\vsigma_0(b,6)\right)}{8w B}
$\\\\
which holds

\[\E\left({C_7^\star}(w,B)\right)=\tr\left(\left(\vT\vV_N\right)\right)\hspace*{0.4cm}
\Var\left({C_7^\star}(w,B)\right)=\left(1-\left(1-\frac{1}{B}\right) \frac{\binom{n_{\min}-6} {6}}{\binom{n_{\min}} {6}}\right) 27\tr^3\left(\left(\vT\vV_N\right)^2\right).\]

\end{LeA}
\begin{proof} The proof for this subsampling-type estimator takes the same steps as before, with another amount $M(B,\vsigma_0(b,6))$. At the beginning we calculate expectation value and an upper bound for the variance of the inner sum. We get\\\\
$\begin{array}{ll}&\E\left(\sum\limits_{b=1}^{B}
\frac{\Lambda_{4}\left(j;\vsigma_0(b,6)\right)\cdot \Lambda_{5}\left(j;\vsigma_0(b,6)\right)\Cdot \Lambda_{6}\left(j;\vsigma_0(b,6)\right)}{8B}\right)
\\[2.5ex]
=&\sum\limits_{b=1}^{B}
\frac{\E\left(\Lambda_{4}\left(j;1,\dots,6\right)\cdot \Lambda_{5}\left(j;1,\dots,6\right)\Cdot \Lambda_{6}\left(j;1,\dots,6\right)\right)}{8B} 
=\tr\left(\left(\vT\vV_N\right)^3\right).
\end{array}$\\\\\\

 $\begin{array}{l}
\Var\left(\E\left(\sum\limits_{b=1}^{B}
\frac{\Lambda_{4}\left(j;\vsigma_0(b,6)\right)\cdot \Lambda_{5}\left(j;\vsigma_0(b,6)\right)\Cdot \Lambda_{6}\left(j;\vsigma_0(b,6)\right)}{8B}\Big\lvert\F\left(\vsigma_0(B)\right)\right)\right)
=\Var\left(\tr\left(\left(\vT\vV_N\right)^3\right)\right)=0.

\end{array}$\\\\\\
$\begin{array}{ll}
&\Var\left(\sum\limits_{b=1}^{B}
\frac{\Lambda_{4}\left(j;\vsigma_0(b,6)\right)\cdot \Lambda_{5}\left(j;\vsigma_0(b,6)\right)\Cdot \Lambda_{6}\left(j;\vsigma_0(b,6)\right)}{8B}\right)
\\[2ex]
=&0+\E\left(\Var\left(\sum\limits_{b=1}^{B}
\frac{\Lambda_{4}\left(j;\vsigma_0(b,6)\right)\cdot \Lambda_{5}\left(j;\vsigma_0(b,6)\right)\Cdot \Lambda_{6}\left(j;\vsigma_0(b,6)\right)}{8B}\right)\Big\lvert\F\left(\vsigma_0(B)\right)\right)\end{array}$\\

$\begin{array}{ll}
=&\frac{\E\left(|\N_{B}\times \N_{B} \setminus M\left(B,\vsigma_0(b,6)\right)|\right)}{B^2}\cdot \frac{\Var\left(\Lambda_{4}\left(j;{1},\dots,{6}\right)\cdot \Lambda_{5}\left(j;{1},\dots,{6}\right)\Cdot \Lambda_{6}\left(j;{1},\dots,{6}\right)\right)}{64}
\\[1.0ex]
\hspace{-0.12cm}\stackrel{\ref{MSchae2}}{\leq}&
\left(1-\left(1-\frac{1}{B}\right)\cdot \frac{\binom{n_{\min}-6} {6}}{\binom{n_{\min}} {6}}\right)\cdot 27\tr^3\left(\left(\vT\vV_N\right)^2\right).
\end{array}$\\\\\\
With these values we can consider the whole estimator\\\\
$\begin{array}{ll}\E\left({C_7^\star}\left(w,B\right)\right)&=\frac 1 w\sum\limits_{j=1}^{w}\E\left(\sum\limits_{b=1}^{B}
\frac{\Lambda_{4}\left(j;\vsigma_0(b,6)\right)\cdot \Lambda_{5}\left(j;\vsigma_0(b,6)\right)\Cdot \Lambda_{6}\left(j;\vsigma_0(b,6)\right)}{8\cdot B}\right)
=\tr\left(\left(\vT\vV_N\right)^3\right),
\end{array}$\\\\\\

$\begin{array}{ll}
\Var\left({C_7^\star}(w,B)\right)&

\leq \frac{1}{w^2}\left(\sum\limits_{j=1}^{w}\sqrt{\Var\left(\sum\limits_{b=1}^{B}
\frac{\Lambda_{4}\left(j;\vsigma_0(b,6)\right)\cdot \Lambda_{5}\left(j;\vsigma_0(b,6)\right)\Cdot \Lambda_{6}\left(j;\vsigma_0(b,6)\right)}{8B}\right)}\right)^2
\\[2.5ex]
&\leq \frac{1}{w^2}\left(\sum\limits_{j=1}^{w}\sqrt{\left(1-\left(1-\frac{1}{B}\right)\cdot \frac{\binom{n_{\min}-6} {6}}{\binom{n_{\min}} {6}}\right)\cdot 27\tr^3\left(\left(\vT\vV_N\right)^2\right)}\right)^2
\\[2.5ex]
&=\left(1-\left(1-\frac{1}{B}\right)\cdot \frac{\binom{n_{\min}-6} {6}}{\binom{n_{\min}} {6}}\right)\cdot 27\tr^3\left(\left(\vT\vV_N\right)^2\right).
\end{array}$\\
\end{proof}\vspace{\baselineskip}
The next lemma shows that the version of the estimators with random indices has all the properties the classical ones possess.

\begin{LeA}\label{ZSchae5}

The statements of \Cref{Schae2}, \Cref{Schae7}, \Cref{Schae8}, \Cref{Lemma: fP Estimate}  and \Cref{MSchae7} are also true, if all or only a part of the estimators are replaced by the subsampling-type estimators.
 \\
 Moreover, \Cref{Theorem3} , \Cref{Theorem4} and \Cref{theo: fP} hold, if all or only a part of the estimators are replaced by the subsampling-type estimators.

\end{LeA}
\begin{proof}
For the proofs of the classical estimators from the first paragraph, only the expectation values are used together with upper bounds for the variances which are zero sequences. With random indices, the expectation is the same and for the variance, all traces are the same but the zero sequence changes. So the proofs of the subsampling-type estimators work identically.\\
For the second paragraph, only some convergences are necessary, which the subsampling-type estimators also fulfills.
\end{proof}





\subsection{On the asymptotic distribution in our simulation designs}


To chose the convenient  test for our simulation the limit of $\beta_1$ has to be considered. Instead of this we calculate the value of $\tau_P={\tr^2\left(\left(\vT\vV_N\right)^3\right)}\Big/{\tr^3\left(\left(\vT\vV_N\right)^2\right)}$ and because $\vV_N$ is known  no estimation is needed. The ratio $n_1/N$ and $n_2/N$ are the same for all our sample sizes, so the different numbers $n_1,n_2$ has no influence on the values of $\tau_P$. The results can be seen in \Cref{TauS1} and \Cref{TauS2} which leads to the assumption $\tau_P\to 1$  for $H_a^0$ and $\tau_P\to 0$  for $H_b^0$. With \Crp{Bedingungen} this is equivalent to $\beta_1 \to 1$ under $H_0^a$ resp. $\beta_1\to 0$ under $H_0^b$.
\begin{table}[h]\caption{$\tau_P$ for $\vT=\left(\vP_2\otimes \frac{1}{d}\vJ_{d}\right)\vmu$ } 
\begin{tabular}{|c|c|c|c|c|c|c|c|c|c|c|c|c|}
\hline
d&5&10&20&40&70&100&150&200&300&450&600&800\\
\hline
$\tau_P$
& 1
& 1
& 1& 1& 1& 1& 1& 1& 1& 1& 1& 1\\
\hline
\end{tabular}\label{TauS1}
\end{table}

\begin{table}[h]\caption{$\tau_P$ for $\vT=\left(\frac 1 2\vJ_2\otimes \vP_d\right)$ }
\begin{tabular}{|c|c|c|c|c|c|c|c|c|c|c|c|c|}
\hline
d&5&10&20&40&70&100&150&200&300&450&600&800\\
\hline
$\tau_P$
& 0.50
& 0.36
& 0.21
& 0.11
& 0.064
& 0.045
& 0.03
& 0.022
& 0.015
& 0.010
& 0.0074
& 0.0056\\
\hline
\end{tabular}
\label{TauS2}
\end{table}
\subsection{On the Chen-Qui-Condition}

We can also develop an estimator for $\tau_{CQ}=\tr\left(\left(\vT\vV_N\right)^4\right)/\tr^2\left(\left(\vT\vV_N\right)^2\right)=1/f_{CQ}$  on an analogical way as before. This leads to:
%
\begin{LeA}\label{MSchae3}
Let be
\[\begin{array}{ll}
C_6=&\sum\limits_{\begin{footnotesize}\substack{\ell_{1,1},\dots, \ell_{8,1}=1\\\ell_{1,1}\neq\dots\neq \ell_{8,1}}\end{footnotesize}}^{n_1}
\dots\sum\limits_{\begin{footnotesize}\substack{\ell_{1,a},\dots, \ell_{8,a}=1\\\ell_{1,a}\neq\dots\neq \ell_{8,a}}\end{footnotesize}}^{n_a}\left[\frac 1 6\frac{\Lambda_{7}(\ell_{1,1},\dots,\ell_{8,a})}{ {16\cdot \prod\limits_{i=1}^a \frac{n_i!}{(n_i-8)!}}}-\frac 1 2\frac{\Lambda_{8}(\ell_{1,1},\dots,\ell_{8,a})}{ {16\cdot \prod\limits_{i=1}^a \frac{n_i!}{(n_i-8)!}}}\right]
\end{array}\]

with\\\\
$\begin{array}{ll}
\Lambda_{7}(\ell_{1,1},\dots,\ell_{8,a})&=\left[\vZ_{(\ell_{1,1},\ell_{2,1},\dots,\ell_{1,a},\ell_{2,a})}^\top \vT \vZ_{(\ell_{3,1},\ell_{4,1},\dots,\ell_{3,a},\ell_{4,a})}\right]^4,
\\
\Lambda_{8}(\ell_{1,1},\dots,\ell_{8,a})&=\left[\vZ_{(\ell_{1,1},\ell_{2,1},\dots,\ell_{1,a},\ell_{2,a})}^\top \vT \vZ_{(\ell_{3,1},\ell_{4,1},\dots,\ell_{3,a},\ell_{4,a})}\right]^2\cdot \left[Z_{(\ell_{5,1},\ell_{6,1},\dots,\ell_{5,a},\ell_{6,a})}^\top \vT \vZ_{(\ell_{7,1},\ell_{8,1},\dots,\ell_{7,a},\ell_{8,a})}\right]^2.
\end{array}$\\\\\\
Then we know
\[\E(C_6)=\tr\left(\left(\vT\vV_N\right)^4\right)\hspace{1.5cm}
\Var(C_6)\leq\frac{\prod\limits_{i=1}^a\binom{n_i}{8}-\prod\limits_{i=1}^a\binom{n_i-8}{8}}{16^2\cdot\prod\limits_{i=1}^a\binom{n_i}{8}}{\0\left( \tr^4\left(\left(\vT\vV_N\right)^2\right)\right)}.\]
\end{LeA}
\begin{proof}
\hspace*{1cm}\newline
$\begin{array}{ll}
\E(C_6)&=\frac{\E\left(\left[\vZ_{(1,2)}^\top \vT \vZ_{(3,4)}\right]^4\right)}{6\cdot 16} - \frac { \E\left(\left[\vZ_{(1,2)}^\top \vT \vZ_{(3,4)}\right]^2\left[\vZ_{(5,6)}^\top \vT \vZ_{(7,8 )}\right]^2\right)}{2\cdot 16}
\\[1ex]
&\hspace{-0.075 cm}\stackrel{\ref{QF3}}{=}\frac{1}{6\cdot 16}\left(6\tr\left(\left(2\vT\vV_N\right)^4\right)+3 \tr^2\left(\left(2\vT\vV_N\right)^2\right)\right)- \frac 1 {2\cdot 16}\tr^2\left(\left(2\vT\vV_N\right)^2\right)= \tr\left(\left(\vT\vV_N\right)^4\right)
\end{array}$\\\\

For the second inequality, the variance of parts is calculated. Like  before with \Crp{Spur1} and \Crp{QF3} we calculate\\

$\Var\left(\frac 1 6\left[{\vZ_{(1,2)}}^\top \vT \vZ_{(3,4)}\right]^4\right)=  \0\left(\tr^4\left(\left(\vT\vV_N\right)^2\right)\right)$\\\\ and \\\\
$\begin{array}{ll}
&\Var\left(\frac 1 2\left[{\vZ_{(1,2)}}^\top \vT \vZ_{(3,4)}\right]^2\left[{\vZ_{(5,6)}}^\top \vT \vZ_{(7,8)}\right]^2\right)
\\[1.5ex]
\leq& \frac 1 4\cdot\E\left(\left[{\vZ_{(1,2)}}^\top \vT \vZ_{(3,4)}\right]^4\left[{\vZ_{(5,6)}}^\top \vT \vZ_{(7,8)}\right]^4\right)
\\[1.5ex]
=&\frac 1 4 \left(6\tr\left(\left(2\vT\vV_N\right)^4\right)+3 \tr^2\left(\left(2\vT\vV_N\right)^2\right)\right)^2= \0\left( \tr^4\left(\left(\vT\vV_N\right)^2\right)\right)
.\end{array}$\\\\\\
With \Crp{Var1} it is known\\\\
$\begin{array}{ll}\Var(A-B)&\leq \Var(A)+\Var(B)+2|\Cov(A,B)|\leq \left(\sqrt{\Var(A)}+\sqrt{\Var(B)}\right)^2\end{array}$\\\\
and therefore\\

$\begin{array}{ll}\Var(C_6)&\leq\frac{\prod\limits_{i=1}^a\binom{n_i}{8}-\prod\limits_{i=1}^a\binom{n_i-8}{8}}{16^2\cdot\prod\limits_{i=1}^a\binom{n_i}{8}}\Var \left(\frac 1 6\Lambda_{7}(1,\dots,8)
-\frac 1 2\Lambda_{8}(1,\dots,8)\right)\end{array}$\\
$\begin{array}{ll}
\textcolor{white}{\Var(C_6)}
&\leq\frac{\prod\limits_{i=1}^a\binom{n_i}{8}-\prod\limits_{i=1}^a\binom{n_i-8}{8}}{16^2\cdot\prod\limits_{i=1}^a\binom{n_i}{8}}  \left(\sqrt{\0\left( \tr^4\left(\left(\vT\vV_N\right)^2\right)\right)}+\sqrt{\0\left( \tr^4\left(\left(\vT\vV_N\right)^2\right)\right)}\right)^2
\end{array}$\\
$\begin{array}{ll}
\textcolor{white}{\Var(C_6)}&=\frac{\prod\limits_{i=1}^a\binom{n_i}{8}-\prod\limits_{i=1}^a\binom{n_i-8}{8}}{16^2\cdot\prod\limits_{i=1}^a\binom{n_i}{8}}{\0\left( \tr^4\left(\left(\vT\vV_N\right)^2\right)\right)}.
\end{array}
$\\
\end{proof}\vspace{1\baselineskip}

\begin{LeA}\label{MSchae4}
With the estimators introduced in the previous lemmata  it holds for fixed $a$
\[\frac{C_6}{A_4^2}-\frac{\tr\left(\left(\vT\vV_N\right)^4\right)}{\tr^2\left(\left(\vT\vV_N\right)^2\right)}
\stackrel{P}{\longrightarrow} 0\hspace{01.0cm}\text{for}\hspace{0.5cm} d,n_{\min}\to \infty.\]
If $p>1$ exists with $n_{\min}=\0(a^p)$, the convergence even holds in the asymptotic frameworks \eqref{eq: as frame 2}-\eqref{eq: as frame 3}.

\end{LeA}
\begin{proof}
Again we first consider the parts:\\\\
$\begin{array}{ll}
\E\left(\frac{C_6}{\tr^2\left(\left(\vT\vV_N\right)^2\right)}-\frac{\tr\left(\left(\vT\vV_N\right)^4\right)}{\tr^2\left(\left(\vT\vV_N\right)^2\right)}\right)=
\frac{\E\left(C_6\right)}{\tr^2\left(\left(\vT\vV_N\right)^2\right)}-\frac{\tr\left(\left(\vT\vV_N\right)^4\right)}{\tr^2\left(\left(\vT\vV_N\right)^2\right)}=
0.
\end{array}
$\\\vspace{\baselineskip}
$\begin{array}{ll}
\Var\left(\frac{C_6}{\tr^2\left(\left(\vT\vV_N\right)^2\right)}-\frac{\tr\left(\left(\vT\vV_N\right)^4\right)}{\tr^2\left(\left(\vT\vV_N\right)^2\right)}\right)
&\leq\frac{\prod\limits_{i=1}^a\binom{n_i}{8}-\prod\limits_{i=1}^a\binom{n_i-8}{8}}{16^2\cdot\prod\limits_{i=1}^a\binom{n_i}{8}}\frac{\0\left( \tr^4\left(\left(\vT\vV_N\right)^2\right)\right)}{\tr^4\left(\left(\vT\vV_N\right)^2\right)}
\leq\frac{\prod\limits_{i=1}^a\binom{n_i}{8}-\prod\limits_{i=1}^a\binom{n_i-8}{8}}{\prod\limits_{i=1}^a\binom{n_i}{8}}\cdot \0(1).
\end{array}$\\\\
So with \Crp{Kons2} for fixed $a$ and $d,n_{\min}\to \infty$ and moreover if the additional condition is fulfilled even for  the asymptotic frameworks \eqref{eq: as frame 2}-\eqref{eq: as frame 3}, it follows

\[\frac{C_6}{\tr^2\left(\left(\vT\vV_N\right)^2\right)}-\frac{\tr\left(\left(\vT\vV_N\right)^4\right)}{\tr^2\left(\left(\vT\vV_N\right)^2\right)}\stackrel{P}{\longrightarrow} 0.\]\\

Analogue to the proof of \Cref{Lemma: fP Estimate} it follows ${\tr^2\left(\left(\vT\vV_N\right)^2\right)}\Big/{A_4^2}\stackrel{P}{\longrightarrow} 1.$\\

Together this leads to
\[\begin{array}{l}\frac{C_6}{A_4^2}-\frac{\tr\left(\left(\vT\vV_N\right)^4\right)}{\tr^2\left(\left(\vT\vV_N\right)^2\right)}=\frac{C_6}{\tr^2\left(\left(\vT\vV_N\right)^2\right)}\cdot \frac{\tr^2\left(\left(\vT\vV_N\right)^2\right)}{A_4^2}-\frac{\tr\left(\left(\vT\vV_N\right)^4\right)}{\tr^2\left(\left(\vT\vV_N\right)^2\right)}
\\[1.2ex]
=\frac{C_6}{\tr^2\left(\left(\vT\vV_N\right)^2\right)}\cdot (1+o_P(1))-\frac{\tr\left(\left(\vT\vV_N\right)^4\right)}{\tr^2\left(\left(\vT\vV_N\right)^2\right)}=o_P(1)+o_P(1)=o_P(1).\end{array}\]
\end{proof}\vspace{1\baselineskip}

Again in most cases the subsampling-type version of this estimator should be used.

\begin{LeA}\label{ZSchae3}
Let be
\[ {C_6^\star}(B)=\frac{1}{16B}\sum\limits_{b=1}^B \left(\frac{\Lambda_{7}(\vsigma(b,8))}{6}-\frac{{\Lambda_{8}}(\vsigma(b,8))}{2}\right).\]\\
Then it holds
 \[\begin{array}{l}
 \E\left({C_6^\star}(B)\right)=\tr\left(\left(\vT\vV_N\right)^4\right),
 \\[2ex]
\Var\left({C_6^\star}(B)\right)\leq\left(1-\left(1-\frac{1}{B}\right)\cdot\prod\limits_{i=1}^a\frac{ \binom{n_i-8} {8}}{\binom{n_i} {6}}\right)\cdot \0\left(\tr^4\left(\left(\vT\vV_N\right)^2\right)\right)
.\end{array}\]\end{LeA}
\begin{proof}
By using the same steps as before it holds\\\\
$\begin{array}{ll}
\E\left({C_6^\star}(B)\right)&=\frac{1}{16B}\sum\limits_{b=1}^B\E\left(\frac{\Lambda_{7}(\ell_{1,1},\dots,\ell_{8,a})}{6}-\frac{\Lambda_{8}(\ell_{1,1},\dots,\ell_{8,a})}{2}\right)\\
[1.6ex]
&=\frac 1{16B}\sum\limits_{b=1}^B\E\left(\left[{\vZ_{(1,2)}}^\top \vT\vZ_{(3,4)}\right]^2\cdot\left(\frac{\left[{\vZ_{(1,2)}}^\top \vT\vZ_{(3,4)}\right]^2}{6}-\frac{\left[{\vZ_{(5,6)}}^\top \vT\vZ_{(7,8)}\right]^2}{2}\right)\right)
\\[1.6ex]
&\hspace{-0.12cm}\stackrel{\ref{MSchae3}}{=}\frac 1 {16B}\sum\limits_{b=1}^B
\tr\left(\left(2\vT\vV_N\right)^4\right)=\tr\left(\left(\vT\vV_N\right)^4\right).
\end{array}$\\\\

$\begin{array}{l}
\Var\left(\E\left({C_6^\star}(B)|\F(\vsigma(B,8))\right)\right)=\Var\left(\tr\left(\left(\vT\vV_N\right)^4\right)\right)=0.
\end{array}$\\\\

$\begin{array}{ll}
\Var\left({C_6^\star}(B)\right)&=0+\E\left(\Var\left({C_6^\star}(B)|\F(\vsigma(B,8))\right)\right)
\\[0.5ex]
&\hspace{-0.075cm}\stackrel{\ref{Var1}}{\leq}\frac{1}{16^2 B^2}\E\left(\sum\limits_{(j,\ell)\in \N_B\times \N_B \setminus M(B,\vsigma(b,8))} \Var\left(\frac{\Lambda_{7}(\vsigma(j,8))}{6}-\frac{\Lambda_{8}(\vsigma(j,8))}{2}\Big\lvert\F(\vsigma(B,8))\right)\right)\\[.8ex]
\end{array}$\\\vspace{0.3\baselineskip}

$\begin{array}{ll}
\textcolor{white}{\Var\left({C_6^\star}(B)\right)}&
= \frac{\Var\left(\frac{\Lambda_{7}(\ell_{1,1},\dots,\ell_{8,a})}{6}-\frac{\Lambda_{8}(\ell_{1,1},\dots,\ell_{8,a})}{2}\right)}{16^2 B\cdot\left(\E\left(|\N_B\times \N_B \setminus M(B,\vsigma(b,8))|\right)\right)^{-1}}
\\[1.5ex]
&\hspace{-0.12cm}\stackrel{\ref{MSchae3}}{\leq}
\left(1-\left(1-\frac{1}{B}\right)\cdot \prod\limits_{i=1}^a\frac{ \binom{n_i-8} {8}}{\binom{n_i} {8}}\right)\cdot \0\left(\tr^4\left(\left(\vT\vV_N\right)^2\right)\right).

\end{array}$

\end{proof}\vspace{\baselineskip}

With \Cref{ZSchae4}
we get an estimator for $\tau_{CQ}$ with $\widehat {\tau_{CQ}}({C_6^\star},A_4)={{C_6^\star}}/{A_4^2}$ and once more for a large number of groups ${A_4^\star}$  should be used. 

\begin{LeA}\label{Invarianzgegenschaetzer}

 \Cref{Theorem5} is also valid if $f_P$ is replaced by $f_{CQ}$ or by $(\widehat {\tau_{CQ}}({C_6},A_4))^{-1}$ . Using ${C_6^\star}$ or ${A_4^\star}$  also doesn't change the result. Identical the result of \Cref{MSchae4} remains true if one or all estimators are replaced by their subsampling version.

\end{LeA} 
\begin{proof}
 With \Cref{Bedingungen} we know $f_p\to 1\Leftrightarrow f_{CQ} \to 1$ and $f_p\to 0\Leftrightarrow f_{CQ} \to 0$ so in both cases $K_{f_P}$ is asymptotically identic with  $K_{f_{CQ}}$.\\
 
From \Cref{MSchae4} we know that  $\widehat{\tau_{CQ}}-\tau_{CQ}$ converges in probability to zero so this result follows identically to \Cref{Theorem5}. At last the subsampling versions have the same properties like the standard estimators.
\end{proof}\vspace{\baselineskip}

Therefore this is a second way to test the hypotheses and moreover, it provides an indicator for the choice of the limit distribution, because of \Cref{Bedingungen}. For situation c) from \Cref{Theorem3} there is no proof that this approach can be used but in the case of just one group it leads to good results.

\end{document}